\documentclass[graybox]{svmult}

\usepackage{mathptmx}       
\usepackage{helvet}         
\usepackage{courier}        
\usepackage{type1cm}        
\usepackage{graphicx}        
\usepackage{multicol}        
\usepackage[bottom]{footmisc}

\usepackage[utf8]{inputenc}
\usepackage[T1]{fontenc}
\usepackage{amsmath}
\usepackage{amssymb}
\usepackage[english]{babel}

\usepackage{cite}

\usepackage{enumitem}  
\usepackage{calc}
\setlist{labelindent=1pt,itemsep=.5em}
\setlist[itemize]{leftmargin=1.2cm}
\setlist[enumerate]{itemindent=0em,leftmargin=1cm}
\setlist[enumerate,1]{label={\upshape(\roman*)}}

\allowdisplaybreaks[3]

\smartqed

\newcommand{\id}{\mathrm{Id}}

\begin{document}

\title*{Double constructions of biHom-Frobenius algebras}
\titlerunning{Double constructions of biHom-Frobenius algebras}
\author{Mahouton Norbert Hounkonnou, Gb\^ev\`ewou Damien  Houndedji, Sergei  Silvestrov}
\authorrunning{Mahouton Norbert Hounkonnou, Gb\^ev\`ewou Damien  Houndedji, Sergei  Silvestrov}
\institute{Mahouton Norbert Hounkonnou \at International Chair in Mathematical Physics and Applications,
ICMPA-UNESCO Chair,  University of Abomey-Calavi, 072 BP 50, Cotonou, Rep. of Benin. \\
\email{norbert.hounkonnou@cipma.uac.bj, hounkonnou@yahoo.fr}
\and Gb\^ev\`ewou Damien  Houndedji \at International Chair in Mathematical Physics and Applications,
ICMPA-UNESCO Chair,   University of Abomey-Calavi, 072 BP 50, Cotonou, Rep. of Benin. \\ \email{houndedjid@gmail.com}
\and
Sergei Silvestrov \at Division of Applied Mathematics, School of Education, Culture and Communication, M\"alardalen University, Box 883, 72123 V\"aster{\aa}s, Sweden. \email{sergei.silvestrov@mdh.se} }

%
%
%
%
%


\maketitle
\label{Chap:HoukonnouHoundedjiSilvestrov:DoubconstrbiHomFrobalgs}

\abstract*{This paper addresses a Hom-associative algebra built as a direct sum of a given Hom-associative
algebra $(\mathcal{A}, \cdot, \alpha)$ and its dual $(\mathcal{A}^{\ast}, \circ, \alpha^{\ast}),$ endowed with a
non-degenerate symmetric bilinear form $\mathcal{B},$ where $\cdot$ and $\circ$ are the products defined
on $\mathcal{A}$ and $\mathcal{A}^{\ast},$ respectively, $\alpha$ and $\alpha^{\ast}$ stand for the corresponding algebra
homomorphisms. Such a double construction, also called Hom-Frobenius algebra, is interpreted in terms of
an  infinitesimal Hom-bialgebra. The same procedure is applied to characterize the double construction of
biHom-associative algebras, also called biHom-Frobenius algebra. Finally, a double construction of Hom-
dendriform algebras, also called double construction of Connes cocycle or symplectic Hom-associative
algebra, is performed.
Besides, the concept of biHom-dendriform algebras is introduced and discussed.
Their bimodules and matched pairs are also constructed, and related relevant properties are given.
}

\abstract{This paper addresses a Hom-associative algebra built as a direct sum of a given Hom-associative
algebra $(\mathcal{A}, \cdot, \alpha)$ and its dual $(\mathcal{A}^{\ast}, \circ, \alpha^{\ast}),$ endowed with a
non-degenerate symmetric bilinear form $\mathcal{B},$ where $\cdot$ and $\circ$ are the products defined
on $\mathcal{A}$ and $\mathcal{A}^{\ast},$ respectively,  and $  \alpha$ and $\alpha^{\ast}$ stand for the corresponding algebra homomorphisms. Such a double construction, also called Hom-Frobenius algebra, is interpreted in terms of an infinitesimal Hom-bialgebra. The same procedure is applied to characterize the double construction of biHom-associative algebras, also called biHom-Frobenius algebra. Finally, a double construction of Hom-dendriform algebras, also called double construction of Connes cocycle or symplectic Hom-associative
algebra, is performed. Besides, the concept of biHom-dendriform algebras is introduced and discussed.
Their bimodules and matched pairs are also constructed, and related relevant properties are given.
}

\section{Introduction}

The Hom-algebraic structures originated from  quasi-deformations of Lie algebras of
vector fields which gave rise to quasi-Lie algebras, defined as  generalized
Lie structures in which the skew-symmetry and Jacobi conditions are twisted.
Hom-Lie algebras and more general quasi-Hom-Lie algebras where introduced first by Silvestrov and his students Hartwig and Larsson in \cite{HnkHndSilvdcbihomfrobalg:HLS}, where the general quasi-deformations and discretizations of Lie algebras of vector fields using general twisted derivations, $\sigma$-derivations, and a general method for construction of deformations of Witt and Virasoro type algebras based on twisted derivations has been developed. The initial motivation came from examples of $q$-deformed Jacobi identities discovered in $q$-deformed versions and other discrete modifications of differential calculi and homological algebra, $q$-deformed Lie algebras and other algebras important in string theory, vertex models in conformal field theory, quantum mechanics and quantum field theory, such as the $q$-deformed Heisenberg algebras, $q$-deformed oscillator algebras, $q$-deformed Witt, $q$-deformed Virasoro algebras and related $q$-deformations of infinite-dimensional algebras  \cite{HnkHndSilvdcbihomfrobalg:AizawaSaito,HnkHndSilvdcbihomfrobalg:ChaiElinPop,HnkHndSilvdcbihomfrobalg:ChaiIsLukPopPresn,HnkHndSilvdcbihomfrobalg:ChaiKuLuk,HnkHndSilvdcbihomfrobalg:ChaiPopPres,HnkHndSilvdcbihomfrobalg:CurtrZachos1,HnkHndSilvdcbihomfrobalg:DamKu,HnkHndSilvdcbihomfrobalg:DaskaloyannisGendefVir,HnkHndSilvdcbihomfrobalg:Hu,HnkHndSilvdcbihomfrobalg:Kassel92,
HnkHndSilvdcbihomfrobalg:LiuKQuantumCentExt,HnkHndSilvdcbihomfrobalg:LiuKQCharQuantWittAlg,HnkHndSilvdcbihomfrobalg:LiuKQPhDthesis}.

Possibility of studying, within the same framework, $q$-deformations of Lie algebras and such well-known generalizations of Lie algebras as the color and super Lie algebras provided further general motivation for development of quasi-Lie algebras and subclasses of quasi-Hom-Lie algebras and Hom-Lie algebras. The general abstract quasi-Lie algebras and the subclasses of quasi-Hom-Lie algebras and Hom-Lie algebras, as well as their color (graded) counterparts, color (graded) quasi-Lie algebras, color (graded) quasi-Hom-Lie algebras and color (graded) Hom-Lie algebras, including in particular the super quasi-Lie algebras, super quasi-Hom-Lie algebras, and super Hom-Lie algebras, have been introduced in
\cite{HnkHndSilvdcbihomfrobalg:HLS,
HnkHndSilvdcbihomfrobalg:LarssonSilvJA2005:QuasiHomLieCentExt2cocyid,
HnkHndSilvdcbihomfrobalg:LarssonSilvCM2005:QuasiLieAl,
HnkHndSilvdcbihomfrobalg:LSGradedquasiLiealg,
HnkHndSilvdcbihomfrobalg:SigSilGrquasiLieWitttype:CzJPhys2006,HnkHndSilvdcbihomfrobalg:SigSilv:LiecolHomLieWittcex:GLTbdSpr2009}.
In \cite{HnkHndSilvdcbihomfrobalg:MS:homstructure}, Hom-associative algebras have been introduced.
Hom-associative algebras is a generalization of  the  associative algebras with the associativity law twisted by
a linear map. In \cite{HnkHndSilvdcbihomfrobalg:MS:homstructure}, Hom-Lie admissible algebras generalizing Lie-admissible algebras, were introduced as Hom-algebras such that the commutator product, defined using the multiplication in a Hom-algebra, yields a Hom-Lie algebra, and Hom-associative algebras were shown to be Hom-Lie admissible. Moreover, in \cite{HnkHndSilvdcbihomfrobalg:MS:homstructure}, more general $G$-Hom-associative algebras including Hom-associative algebras, Hom-Vinberg algebras (Hom-left symmetric algebras), Hom-pre-Lie algebras (Hom-right symmetric algebras), and some other Hom-algebra structures, generalizing $G$-associative algebras, Vinberg and pre-Lie algebras respectively, have been introduced
and shown to be Hom-Lie admissible, meaning that for these classes of Hom-algebras, the operation of taking commutator leads to Hom-Lie algebras as well. Also, flexible Hom-algebras have been introduced, connections to Hom-algebra generalizations of derivations and of adjoint maps have been noticed, and some low-dimensional Hom-Lie algebras have been described.
The enveloping algebras of Hom-Lie algebras were considered in \cite{HnkHndSilvdcbihomfrobalg:Yau:HomEnv} using combinatorial objects of weighted binary trees. In \cite{HnkHndSilvdcbihomfrobalg:HeMaSiUnAlHomAss}, for Hom-associative algebras and Hom-Lie algebras, the envelopment problem, operads, and the Diamond Lemma and Hilbert series for the Hom-associative operad and free algebra have been studied. Strong Hom-associativity yielding a confluent rewrite system and a basis for the free strongly hom-associative algebra has been considered in \cite{HnkHndSilvdcbihomfrobalg:He:stronghomassociativity}. An explicit constructive way, based on free Hom-associative algebras with involutive twisting, was developed in \cite{HnkHndSilvdcbihomfrobalg:GuoZhZheUEPBWHLieA} to obtain the universal enveloping algebras and Poincar{\'e}-Birkhoff-Witt type theorem for Hom-Lie algebras with involutive twisting map. Free
Hom-associative color algebra on a Hom-module and enveloping algebra of color Hom-Lie algebras with involutive twisting and also with more general conditions on the powers of twisting map was constructed, and Poincar{\'e}–Birkhoff–Witt type theorem was obtained in \cite{HnkHndSilvdcbihomfrobalg:ArmakanSilvFarh:envalcolhomLieal,HnkHndSilvdcbihomfrobalg:ArmakanSilvFarh:envalcerttypecolhomLieal}.
It is worth noticing here that, in the subclass of Hom-Lie algebras, the  skew-symmetry is untwisted, whereas the Jacobi identity is twisted by a single linear map and contains three terms as in Lie algebras, reducing to ordinary Lie algebras when the twisting linear map is the identity map.

Hom-algebra structures include their classical counterparts and open new broad possibilities for deformations, extensions to Hom-algebra structures of representations, homology, cohomology and formal deformations, Hom-modules and hom-bimodules, Hom-Lie admissible Hom-coalgebras, Hom-coalgebras, Hom-Hopf algebras, Hom-bialgebras, $L$-modules, $L$-comodules and Hom-Lie quasi-bialgebras, $n$-ary generalizations of biHom-Lie algebras and biHom-associative algebras and generalized derivations, Rota-Baxter operators, Hom-dendriform color algebras, Rota-Baxter bisystems and covariant bialgebras, Rota-Baxter cosystems, coquasitriangular mixed bialgebras, coassociative Yang-Baxter pairs, coassociative Yang-Baxter equation and generalizations of Rota-Baxter systems and algebras, curved $\mathcal{O}$-operator systems and their connections with tridendriform systems and pre-Lie algebras  \cite{HnkHndSilvdcbihomfrobalg:AmmarEjbehiMakhlouf:homdeformation,
HnkHndSilvdcbihomfrobalg:Bakayoko:LaplacehomLiequasibialg,
HnkHndSilvdcbihomfrobalg:Bakayoko:LmodcomodhomLiequasibialg,
HnkHndSilvdcbihomfrobalg:BakBan:bimodrotbaxt,
HnkHndSilvdcbihomfrobalg:BakyokoSilvestrov:HomleftsymHomdendicolorYauTwi,
HnkHndSilvdcbihomfrobalg:BakyokoSilvestrov:MultiplicnHomLiecoloralg,
HnkHndSilvdcbihomfrobalg:BenMakh:Hombiliform,
HnkHndSilvdcbihomfrobalg:BenAbdeljElhamdKaygorMakhl201920GenDernBiHomLiealg,
HnkHndSilvdcbihomfrobalg:CaenGoyv:MonHomHopf,
HnkHndSilvdcbihomfrobalg:HassanzadehShapiroSutlu:CyclichomolHomasal,
HnkHndSilvdcbihomfrobalg:kms:narygenBiHomLieBiHomassalgebras2020,
HnkHndSilvdcbihomfrobalg:LarssonSigSilvJGLTA2008,
HnkHndSilvdcbihomfrobalg:LarssonSilvJA2005:QuasiHomLieCentExt2cocyid,
HnkHndSilvdcbihomfrobalg:MaMakhSil:CurvedOoperatorSyst,
HnkHndSilvdcbihomfrobalg:MaMakhSil:RotaBaxbisyscovbialg,
HnkHndSilvdcbihomfrobalg:MaMakhSil:RotaBaxCosyCoquasitriMixBial,
HnkHndSilvdcbihomfrobalg:MakhSil:HomHopf,HnkHndSilvdcbihomfrobalg:MakhSilv:HomDeform,
HnkHndSilvdcbihomfrobalg:MakhSilv:HomAlgHomCoalg,
HnkHndSilvdcbihomfrobalg:RichardSilvestrovJA2008,
HnkHndSilvdcbihomfrobalg:RichardSilvestrovGLTbnd20092009,
HnkHndSilvdcbihomfrobalg:Sheng:homrep,
HnkHndSilvdcbihomfrobalg:SilvestrovParadigmQLieQhomLie2007,
HnkHndSilvdcbihomfrobalg:Yau:ModuleHomalg,
HnkHndSilvdcbihomfrobalg:Yau:HomEnv,
HnkHndSilvdcbihomfrobalg:Yau:HomHom,
HnkHndSilvdcbihomfrobalg:Yau:HombialgcomoduleHomalg}.

The notion of biHom-associative algebras was introduced in \cite{HnkHndSilvdcbihomfrobalg:GrMakMenPan:Bihom1}.
In fact, a biHom-associative algebra is a (nonassociative) algebra $\mathcal{A}$ endowed with two
commuting multiplicative linear maps $\alpha, \beta: \mathcal{A} \rightarrow \mathcal{A}$ such that $\alpha(a)(bc) =(ab)\beta(c)$, for all $a, b, c \in \mathcal{A}$. This concept arose in the study of algebras in so-called group Hom-categories. In \cite{HnkHndSilvdcbihomfrobalg:GrMakMenPan:Bihom1}, the authors introduced biHom-Lie
algebras (also by using the categorical approach) and biHom-bialgebras. They discussed these
new structures by presenting some basic properties and constructions (representations,
twisted tensor products, smash products, etc.).

A  Frobenius algebra  is an associative algebra equipped with a  non-degenerate invariant bilinear form. This type of algebras also plays an important role in different areas of   mathematics and
physics, such as statistical models over two-dimensional graphs \cite{HnkHndSilvdcbihomfrobalg:BordFilkNowak:alclact} and topological quantum field theory
\cite{HnkHndSilvdcbihomfrobalg:Kock:frobalg2Dtopqft}. In \cite{HnkHndSilvdcbihomfrobalg:BaiC:doublconstrfrobalg},  Bai described associative analogs of Drinfeld's double constructions for Frobenius algebras and for associative algebras equipped with  non-degenerate Connes cocycles.
We note that two different types of constructions are  involved:
\begin{enumerate}[label=\upshape{(\roman*)}]
\item the Drinfeld's double type constructions, from a Frobenius algebra or
from an associative algebra equipped with a Connes cocyle; and
\item the Frobenius algebra obtained from  anti-symmetric solution of
associative Yang-Baxter equation and  non-degenerate Connes cocycle
obtained from a symmetric solution of a $D$-equation.
\end{enumerate}

The aim of the present work is to establish the double constructions of biHom-Frobenius algebras and  Hom-associative algebra equipped with a Connes cocyle, generalizing the double constructions of  Frobenius algebras and Connes cocycle described in \cite{HnkHndSilvdcbihomfrobalg:BaiC:doublconstrfrobalg} by
twisting the defining axioms by a certain twisting map. When the twisting map happens to
be the identity map, one gets the ordinary algebraic structures. Furthermore, the
bialgebras of  related  double constructions are built. We define the antisymmetric infinitesimal biHom-bialgebras and Hom-dendriform $D$-bialgebras.

The paper is organized as follows. In Section \ref{sec:bimodmatchedpairshomassalg}, we introduce the concepts of matched pairs
of Hom-associative algebras and establish some relevant properties. In Section \ref{sec:dblcnstrinvhomfrobal}, we perform the double constructions of multiplicative Hom-Frobenius algebras and antisymmetric infinitesimal Hom-bialgebras. In Section \ref{sec:dblcnstrinvbihomfrobal}, we define the  bimodule of biHom-associative
algebras,  and achieve the double constructions of multiplicative biHom-Frobenius algebras and antisymmetric
infinitesimal biHom-bialgebras. Section \ref{sec:dblcnstrinvsymhomassal} deals with the double constructions of involutive symplectic Hom-associative algebras. Section \ref{sec:matchedpbihomassal}
is devoted to the matched pairs of biHom-associative algebras and related important  characteristics. In Section \ref{sec:conclremarks}, we end with some concluding remarks.

\section{Bimodules and matched pairs of Hom-associative algebras}
\label{sec:bimodmatchedpairshomassalg}
\subsection{Bimodules of Hom-associative algebras}
\label{subsec:bimodmatchedpairshomassalg:bimodhomassalg}
\begin{definition}[\cite{HnkHndSilvdcbihomfrobalg:MS:homstructure}]
A Hom-associative algebra is a triple $(\mathcal{A}, \cdot, \alpha)$ consisting of a linear space $\mathcal{A}$ over a field $\mathcal{K}$, $
\mathcal{K}$-bilinear map $\cdot$: $\mathcal{A}\otimes\mathcal{A}\rightarrow \mathcal{A}$ and a linear space map $\alpha: \mathcal{A}
\rightarrow\mathcal{A}$ satisfying the Hom-associativity property:
\begin{eqnarray}
\alpha(x)\cdot(y\cdot z)=(x\cdot y)\cdot\alpha(z).
\end{eqnarray}
If, in addition, $\alpha$ satisfies the multiplicativity property
\begin{eqnarray}
\alpha(x\cdot y)=\alpha(x)\cdot\alpha(y),
\end{eqnarray}
then $(\mathcal{A}, \cdot, \alpha)$ is said to be multiplicative.
\end{definition}
\begin{remark}
If $\alpha=\id_{\mathcal{A}},$ $(\mathcal{A}, \cdot, \id_{A}),$ simply denoted $(\mathcal{A}, \cdot),$ is an associative algebra.
\end{remark}
\begin{example}
Let $\lbrace e_{1}, e_{2}, e_{3}\rbrace$ be a basis of a $3$-dimensional vector space $\mathcal{A}$ over $\mathcal{K}$. The following multiplication $\cdot$ and map $\alpha$ on $\mathcal{A}$ define a Hom-associative algebra:
\begin{eqnarray}
&&e_{1}\cdot e_{1}= e_{1},\ e_{1}\cdot e_{2}= e_{2}\cdot e_{1}= e_{3},\cr
&&\alpha(e_{1})= a_{1}e_{2} + a_{2}e_{3},\ \alpha(e_{2})= b_{1}e_{2} + b_{2}e_{3}, \alpha(e_{3})= c_{1}e_{2} + c_{2}e_{3},
\end{eqnarray}
where $a_{1}, a_{2}, b_{1}, b_{2}, c_{1}, c_{2} \in \mathcal{K}.$
\end{example}
\begin{definition}
A Hom-module is a pair $(V, \beta),$ where $V$ is a $\mathcal{K}$-vector space, and $\beta: V\rightarrow V$ is a linear map.
\end{definition}

We will use in this article a definition of bimodule of a Hom-associative algebras
including Hom-modules maps conditions \eqref{Hombimodule:bl:eq1}, \eqref{Hombimodule:br:eq2}, while we note that there are also other definitions of Hom-modules and Hom-bimodules of Hom-associative algebras, for example the more general notions requiring only \eqref{Cond:HomBimod:lpb},  \eqref{Cond:HomBimod:rpb} and \eqref{Cond:HomBimod:lar}, \cite{HnkHndSilvdcbihomfrobalg:Bakayoko:LmodcomodhomLiequasibialg,
HnkHndSilvdcbihomfrobalg:BakBan:bimodrotbaxt, HnkHndSilvdcbihomfrobalg:MakhSil:HomHopf,
HnkHndSilvdcbihomfrobalg:MakhSilv:HomAlgHomCoalg,
HnkHndSilvdcbihomfrobalg:MakhSilv:HomDeform,
HnkHndSilvdcbihomfrobalg:HassanzadehShapiroSutlu:CyclichomolHomasal,
HnkHndSilvdcbihomfrobalg:Yau:ModuleHomalg,
HnkHndSilvdcbihomfrobalg:Yau:HombialgcomoduleHomalg}.

In order to avoid, when necessary, the ambiguity of the general category endomorphisms notation $End(L)$ for endomorphisms of $L$ as linear space, algebra or other structure, throughout this paper, we will use the notation $gl(L)$ for the set of all linear transformations on a linear space $L$, and viewing it context-dependent, as a linear space, as a associative algebra with usual associative composition product, as a Lie algebra of all linear transformations on $L$ with the usual commutator product of the associative composition product (usual notation for Lie algebras), or as other structure type on the set $gl(L)$.

\begin{definition}
Let $(\mathcal{A}, \cdot, \alpha)$ be a Hom-associative algebra and let $(V, \beta)$ be a Hom-module. Let $ l, r: \mathcal{A} \rightarrow gl(V) $ be two linear maps. The quadruple  $(l, r, \beta, V)$ is called a bimodule of $\mathcal{A}$ if for all $ x, y \in  \mathcal{A}, v \in V $:
\begin{align} \label{Cond:HomBimod:lpb}
 l(x\cdot y)\beta(v) &= l(\alpha(x))l(y)v,\\
 \label{Cond:HomBimod:rpb}
 r(x\cdot y)\beta(v) &= r(\alpha(y))r(x)v, \\
 \label{Cond:HomBimod:lar}
 l(\alpha(x))r(y)v &= r(\alpha(y))l(x)v,
\\
\label{Hombimodule:bl:eq1}
\beta(l(x)v) &= l(\alpha(x))\beta(v), \\
\label{Hombimodule:br:eq2}
\beta(r(x)v) &= r(\alpha(x))\beta(v).
\end{align}
\end{definition}
\begin{proposition}
Let $(\mathcal{A}, \cdot, \alpha)$ be a Hom-associative algebra and let $(V, \beta)$ be a Hom-module. Let $ l, r: \mathcal{A} \rightarrow gl(V) $ be two linear maps. The quadruple  $(l, r, \beta, V)$  satisfies a Hom-bimodule properties \eqref{Cond:HomBimod:lpb}, \eqref{Cond:HomBimod:rpb}, \eqref{Cond:HomBimod:lar} of a Hom-associative algebra $(\mathcal{A}, \cdot, \alpha )$ if and only if the direct sum  of vector spaces, $\mathcal{A} \oplus V$,  is turned into a Hom-associative algebra  by defining multiplication in $ \mathcal{A} \oplus V $ by
\begin{eqnarray}
(x_{1} + v_{1}) \ast (x_{2} + v_{2}) &=& x_{1} \cdot x_{2} + (l(x_{1})v_{2} + r(x_{2})v_{1}),\cr
(\alpha\oplus\beta)(x_{1} + v_{1})&=&\alpha(x_{1}) + \beta(v_{1})
\end{eqnarray}
for all $ x_{1}, x_{2} \in  \mathcal{A}, v_{1}, v_{2} \in V$.
\end{proposition}
\begin{proof}
Let $v_{1}, v_{2}, v_{3}\in V$ and $x_{1}, x_{2},x_{3}\in \mathcal{A}.$
The left-hand side and right-hand side of Hom-associativity of
$( \mathcal{A} \oplus V, \ast, \alpha\oplus\beta )$  are expended as follows:
\begin{align*}\label{condit. du Bimod.}
& ((x_{1} + v_{1}) \ast (x_{2} + v_{2}))\ast(\alpha\oplus\beta)(x_{3}+v_{3}) \\
& \quad =((x_{1} + v_{1}) \ast (x_{2} + v_{2}))\ast(\alpha(x_{3})+ \beta(v_{3})) \\
& \quad =(x_{1} \cdot x_{2} + (l(x_{1})v_{2} + r(x_{2})v_{1}))\ast(\alpha(x_{3})+ \beta(v_{3})) \\
& \quad =(x_{1} \cdot x_{2})\cdot \alpha(x_{3}) + (l(x_{1} \cdot x_{2})\beta(v_{3}) + r(\alpha(x_{3}))(l(x_{1})v_{2} + r(x_{2})v_{1}))) \\
& \quad =(x_{1} \cdot x_{2})\cdot \alpha(x_{3}) + (l(x_{1} \cdot x_{2})\beta(v_{3}) + r(\alpha(x_{3}))l(x_{1})v_{2} + r(\alpha(x_{3})) r(x_{2})v_{1}) \\
& (\alpha\oplus\beta)(x_{1}+v_{1}) \ast ((x_{2} + v_{2}) \ast (x_{3} + v_{3}))\\
& \quad =(\alpha(x_{1})+ \beta(v_{1})) \ast ((x_{2} + v_{2}) \ast (x_{3} + v_{3}))\\
& \quad =(\alpha(x_{1})+ \beta(v_{1}))\ast (x_{2} \cdot x_{3} + (l(x_{2})v_{3} + r(x_{3})v_{2})) \\
& \quad =\alpha(x_{1}) \cdot (x_{2} \cdot x_{3}) + l(\alpha(x_{1}))(l(x_{2})v_{3} + r(x_{3})v_{2})) + r(x_{2} \cdot x_{3})\beta(v_{1}) \\
&\quad = \alpha(x_{1}) \cdot (x_{2} \cdot x_{3}) + (l(\alpha(x_{1}))l(x_{2})v_{3}
+ l(\alpha(x_{1}))r(x_{3})v_{2} + r(x_{2} \cdot x_{3})\beta(v_{1}))
\end{align*}
These elements of $\mathcal{A} \oplus V$ are equal if and only if
\begin{align*}
& \alpha(x_{1}) \cdot (x_{2} \cdot x_{3})  = (x_{1} \cdot x_{2})\cdot \alpha(x_{3}) \\
& l(x_{1} \cdot x_{2})\beta(v_{3})
+ r(\alpha(x_{3}))l(x_{1})v_{2} + r(\alpha(x_{3})) r(x_{2})v_{1} \\
& \quad \quad \quad \quad \quad \quad \quad \quad =
l(\alpha(x_{1}))l(x_{2})v_{3} + l(\alpha(x_{1}))r(x_{3})v_{2} + r(x_{2} \cdot x_{3})\beta(v_{1})
\end{align*}
for all $x_{1}, x_{2},x_{3}\in \mathcal{A}, v_{1}, v_{2}, v_{3} \in V$.
This holds if and only if the hom-associativity holds, and for each $j=1,2,3$ the respective $V$ terms
involving $v_{j} \in V$ are equal. If the terms are equal then the sums are equal. If the summs are equal then the terms should be equal if one specifies all or two of $v_1, v_2, v_3$ to zero element of $V$ and using that linear transformations map zero to zero.
Since,
\begin{align*}
\text{Hom-associativity} & \Leftrightarrow &
\alpha(x_{1}) \cdot (x_{2} \cdot x_{3}) &= (x_{1} \cdot x_{2})\cdot \alpha(x_{3}),
\ \forall x_{1}, x_{2},x_{3}\in \mathcal{A} \\
\eqref{Cond:HomBimod:lpb} & \Leftrightarrow &
l(x_{1} \cdot x_{2})\beta(v_{3}) &= l(\alpha(x_{1}))l(x_{2})v_{3}, \ \forall x_{1}, x_{2},x_{3}\in \mathcal{A} \\
\eqref{Cond:HomBimod:lar} & \Leftrightarrow & l(\alpha(x_{1}))r(x_{3})v_{2} &= r(\alpha(x_{3}))l(x_{1})v_{2}, \ \forall x_{1}, x_{2},x_{3}\in \mathcal{A} \\
\eqref{Cond:HomBimod:rpb} & \Leftrightarrow &
r(\alpha(x_{3})) r(x_{2})v_{1} &= r(x_{2} \cdot x_{3})\beta(v_{1}), \ \forall x_{1}, x_{2},x_{3}\in \mathcal{A}.
\end{align*}
the proof is complete.
\qed
\end{proof}
We denote such a Hom-associative algebra $(\mathcal{A}\oplus V, \ast, \alpha + \beta),$
 or $\mathcal{A}\times_{l, r, \alpha, \beta}V.$
\begin{example}
Let $(\mathcal{A}, \cdot, \alpha)$ be a multiplicative Hom-associative algebra. Let $L_{\cdot x}$ and $R_{\cdot x}$ denote the left and right multiplication operators, respectively, i. e.
$L_{\cdot x}(y)=x\cdot y, R_{\cdot x}(y) = y\cdot x$ for any $ x, y \in  \mathcal{A}$. Let $L_{\cdot} :  \mathcal{A} \rightarrow gl( \mathcal{A})$ with $ x \mapsto L_{\cdot x}$ and
$ R_{\cdot} : \mathcal{A} \rightarrow gl(\mathcal{A})$ with $ x \mapsto R_{\cdot x}$ (for every $ x \in  \mathcal{A} $) be two linear maps.
Then, the triples $(L_{\cdot}, 0, \alpha), (0, R_{\cdot}, \alpha)$ and $ (L_{\cdot}, R_{\cdot}, \alpha)$ are bimodules of $(\mathcal{A}, \cdot, \alpha)$.
\end{example}
\begin{proposition}
Let $(l, r, \beta, V)$ be bimodule of a multiplicative Hom-associative algebra $(\mathcal{A}, \cdot, \alpha)$. Then, $(l\circ\alpha^{n}, r\circ\alpha^{n}, \beta, V)$ is a bimodule of $\mathcal{A}$ for any integer $n.$
\end{proposition}
\begin{proof}
We have
\begin{eqnarray*}
l\circ\alpha^{n}(x\cdot y)\beta(v)= l(\alpha^{n}(x)\cdot\alpha^{n}(y))\beta(v)= l(\alpha(\alpha^{n}(x)))l(\alpha^{n}(y))v\cr
=l(\alpha^{n+1}(x))l(\alpha^{n}(y))v= l\circ\alpha^{n}(\alpha(x))l\circ\alpha^{n}(y)v.
\end{eqnarray*}
Similarly, the other relations are established.
\qed
\end{proof}
\begin{example}
Let $(\mathcal{A}, \cdot, \alpha)$ be a multiplicative Hom-associative algebra. Then, the quadruple $(L_{\cdot}\circ\alpha^{n}, R_{\cdot}\circ\alpha^{n}, \alpha, \mathcal{A})$ is a bimodule of $\mathcal{A}$ for any finite integer $n.$
\end{example}
\begin{example}
Let $(\mathcal{A}, \cdot, \alpha)$ be a multiplicative associative algebra, and $\beta: \mathcal{A}\rightarrow \mathcal{A}$ be a morphism. Then, $\mathcal{A}_{\beta}=(\mathcal{A}, \cdot_{\beta}=\beta\circ\cdot, \alpha_{\beta}=\beta\circ\alpha)$ is a multiplicative Hom-associative algebra.
 Hence $(L_{\cdot_{\beta}}\circ\alpha^{n}_{\beta}, R_{\cdot_{\beta}}\circ\alpha^{n}_{\beta}, \alpha_{\beta}, \mathcal{A})$ is a bimodule of $\mathcal{A}$ for any finite integer $n.$
\end{example}
\subsection{Matched pairs of Hom-associative algebras}
\label{subsec:bimodmatchedpairshomassalg:matchedpairshomassalg}
\begin{theorem}\label{theo. of matched pairs}
Let $(\mathcal{A}, \cdot, \alpha)$ and $(\mathcal{B}, \circ, \beta)$ be two
Hom-associative algebras. Suppose there are linear maps $l_{ \mathcal{A}}, r_{
\mathcal{A}}: \mathcal{A}\rightarrow gl(\mathcal{B})$ and $l_{ \mathcal{B}}, r_{
\mathcal{B}} :  \mathcal{B} \rightarrow gl( \mathcal{A})$ such that the quadruple
 $(l_{ \mathcal{A}}, r_{
\mathcal{A}},  \beta, \mathcal{B})$ is
a bimodule of $\mathcal{A},$ and $(l_{ \mathcal{B}}, r_{ \mathcal{B}}, \alpha,
\mathcal{A})$ is a bimodule of $\mathcal{B},$
satisfying, for any $ x, y \in  \mathcal{A}, a,b \in  \mathcal{B}$, the following {\it  conditions}:
\begin{eqnarray} \label{match. pair1}
l_{ \mathcal{A}}(\alpha(x))(a\circ b) = l_{ \mathcal{A}}(r_{ \mathcal{B}}(a)x)\beta(b) + (l_{ \mathcal{A}}(x)a) \circ \beta(b),
\end{eqnarray}
\begin{eqnarray} \label{match. pair2}
r_{ \mathcal{A}}(\alpha(x))(a\circ b)= r_{ \mathcal{A}}(l_{ \mathcal{B}}(b)x)\beta(a) + \beta(a) \circ(r_{ \mathcal{A}}(x)b),
\end{eqnarray}
\begin{eqnarray} \label{match. pair3}
l_{ \mathcal{B}}(\beta(a))(x\cdot y)= l_{ \mathcal{B}}(r_{ \mathcal{A}}(x)a)\alpha(y) + (l_{ \mathcal{B}}(a)x) \cdot \alpha(y),
\end{eqnarray}
\begin{eqnarray} \label{match. pair4}
r_{\mathcal{B}}(\beta(a))(x\cdot y) = r_{ \mathcal{B}}(l_{ \mathcal{A}}(y)a)\alpha(x) + \alpha(x) \cdot (r_{ \mathcal{B}}(a)y),
\end{eqnarray}
{\small{
\begin{eqnarray} \label{match. pair5}
l_{ \mathcal{A}}(l_{ \mathcal{B}}(a)x)\beta(b) + (r_{ \mathcal{A}}(x)a)\circ \beta(b) - r_{ \mathcal{A}}(r_{ \mathcal{B}}(b)x)\beta(a) - \beta(a)\circ (l_{ \mathcal{A}}(x)b) = 0,
\end{eqnarray}}}
\begin{eqnarray}\label{match. pair6}
l_{ \mathcal{B}}(l_{ \mathcal{A}}(x)a)\alpha(y) + (r_{ \mathcal{B}}(a)x)\cdot \alpha(y) - r_{ \mathcal{B}}(r_{ \mathcal{A}}(y)a)\alpha(x) - \alpha(x)\cdot (l_{ \mathcal{B}}(a)y) = 0.
\end{eqnarray}
 Then, there is a Hom-associative algebra
 structure on the direct sum $\mathcal{A}\oplus\mathcal{B}$ of
the underlying vector spaces of $\mathcal{A}$ and $\mathcal{B}$ given by
\begin{eqnarray}\label{match. pair product}
(x + a) \ast (y + b) &=& (x\cdot y + l_{ \mathcal{B}}(a)y + r_{ \mathcal{B}}(b)x) + (a\circ b +  l_{ \mathcal{A}}(x)b +  r_{ \mathcal{A}}(y)a),\cr
(\alpha\oplus\beta)(x + a)&=&\alpha(x) + \beta(a)
\end{eqnarray}
for all $ x, y \in  \mathcal{A}, a,b \in  \mathcal{B}$.
\end{theorem}
\begin{proof}
Let $v_{1}, v_{2}, v_{3}\in V$ and $x_{1}, x_{2},x_{3}\in \mathcal{A}.$ Set
\begin{eqnarray}\label{condit. du Bimod. 2}
[(x_{1} + v_{1}) \ast (x_{2} + v_{2})]\ast(\alpha(x_{3})+ \beta(v_{3}))=
 (\alpha(x_{1})+ \beta(v_{1}))\ast[(x_{2} + v_{2}) \ast (x_{3} + v_{3})],\nonumber
\end{eqnarray}
 which is developed to  obtain \eqref{match. pair1}-\eqref{match. pair6}. Then, using the  relations \begin{eqnarray*}
\beta(l_{\mathcal{A}}(x)a)= l_{\mathcal{A}}(\alpha(x))\beta(a),\ \beta(r_{\mathcal{A}}(x)a)= r_{\mathcal{A}}(\alpha(x))\beta(a),
\end{eqnarray*}
\begin{eqnarray*}
\alpha(l_{\mathcal{B}}(a)x)= l_{\mathcal{B}}(\beta(a))\alpha(x),\ \alpha(r_{\mathcal{B}}(a)x)= r_{\mathcal{B}}(\beta(a))\alpha(x),
\end{eqnarray*} we show that $\ast$ is a Hom-associative algebra.
\qed
\end{proof}
We denote this Hom-associative algebra by
$(\mathcal{A}\bowtie \mathcal{B}, \ast, \alpha + \beta)$ or $  \mathcal{A} \bowtie^{l_{ \mathcal{A}}, r_{ \mathcal{A}}, \beta}_{l_{ \mathcal{B}}, r_{ \mathcal{B}}, \alpha}  \mathcal{B}.$
\begin{definition}
Let $ (\mathcal{A},\cdot, \alpha)$ and $(\mathcal{B}, \circ, \beta)$ be two Hom-associative
algebras. Suppose that there are linear maps
 $ l_{ \mathcal{A}}, r_{ \mathcal{A}} :  \mathcal{A} \rightarrow gl( \mathcal{B}) $ and
$ l_{ \mathcal{B}}, r_{ \mathcal{B}} :  \mathcal{B} \rightarrow gl( \mathcal{A}) $ such that
$ (l_{ \mathcal{A}}, r_{ \mathcal{A}}, \beta) $ is a bimodule of $  \mathcal{A} $ and $ (l_{ \mathcal{B}}, r_{ \mathcal{B}}, \alpha) $
is a bimodule of $  \mathcal{B} $. If the conditions  \eqref{match. pair1} - \eqref{match. pair6} are satisfied, then,
 $ ( \mathcal{A},  \mathcal{B}, l_{ \mathcal{A}}, r_{ \mathcal{A}}, \beta, l_{ \mathcal{B}}, r_{ \mathcal{B}}, \alpha) $
 is called a \textbf{matched pair of Hom-associative algebras}.
\end{definition}
\section{Double constructions of involutive Hom-Frobenius algebras and antisymmetric infinitesimal Hom-bialgebras}
\label{sec:dblcnstrinvhomfrobal}
In this section, we consider the multiplicative Hom-associative algebra $(\mathcal{A}, \cdot, \alpha)$ such that $\alpha$ is involutive, i.e, $\alpha^{2}= \id_{\mathcal{A}}$.
\subsection{Double constructions of involutive Hom-Frobenius algebras}
\label{subsec:dblcnstrinvhomfrobal:dcifral}
\begin{definition}
Let $V_{1} $, $ V_{2} $ be two vector spaces. For a linear map $ \phi : V_{1} \rightarrow V_{2} $, we denote the dual (linear) map by $ \phi^{\ast} : V^{\ast}_{2} \rightarrow V^{\ast}_{1} $ given by
\begin{eqnarray*}
\langle v, \phi^{\ast}(u^{\ast})\rangle = \langle \phi(v), u^{\ast} \rangle \mbox{ for all } v \in V_{1} ,  u^{\ast} \in V^{\ast}_{2}.
\end{eqnarray*}
\end{definition}
\begin{lemma}
Let $(l, r, \beta, V)$ be a bimodule of a multiplicative Hom-associative algebra $(\mathcal{A}, \cdot, \alpha)$, and
 let $ l^{\ast}, r^{\ast} $ : $  \mathcal{A} \rightarrow gl(V^{\ast}) $ be the linear maps given for all $ x \in  \mathcal{A} $, $ u^{\ast} \in V^{\ast}$, $v \in V ,$ by
\begin{eqnarray}
\langle l^{\ast}(x)u^{\ast}, v \rangle := \langle l(x)v, u^{\ast} \rangle, \langle r^{\ast}(x)u^{\ast}, v \rangle := \langle r(x)v, u^{\ast} \rangle.
\end{eqnarray}
Then
\begin{enumerate} 
\item $(r^{\ast}, l^{\ast}, \beta^{\ast}, V^{\ast})$ is a bimodule of $(\mathcal{A}, \cdot, \alpha)$;
\item $(r^{\ast}, 0, \beta^{\ast}, V^{\ast})$ and $ (0, l^{\ast}, \beta^{\ast}, V^{\ast})$ are also bimodules of $\mathcal{A}$.
\end{enumerate}
\end{lemma}
\begin{proof} \rm{(i)} Let $ (l, r, \beta, V)$ be a bimodule of a multiplicative Hom-associative algebra $  (\mathcal{A}, \cdot, \alpha)$. We show that
 $ (r^{\ast}, l^{\ast}, \beta^{\ast}, V^{\ast}) $ is a bimodule of $\mathcal{A}$. For $ x, y \in  \mathcal{A}, u^{\ast} \in V^{\ast}, v \in V $,
\begin{enumerate}
\item[(i-1)] the computation
\begin{eqnarray*}
 \langle r^{\ast}(x\cdot y)\beta^{\ast}(u^{\ast}), v \rangle  &=& \langle \beta(r(x\cdot y)v), u^{\ast} \rangle
                                        =\langle r(\alpha(x\cdot y))\beta(v), u^{\ast} \rangle
                                   \cr &=&\langle r(\alpha(x)\cdot \alpha(y))\beta(v), u^{\ast} \rangle
                                   =\langle r(\alpha^{2}(y))r(\alpha(x))v, u^{\ast} \rangle
                                        \cr &=& \langle (r(y)r(\alpha(x)))^{\ast}u^{\ast},v \rangle
                                        = \langle r^{\ast}(\alpha(x))r^{\ast}(y)u^{\ast}, v \rangle
\end{eqnarray*}
 leads to  $ r^{\ast}(x\cdot y)\beta^{\ast}(u^{\ast})= r^{\ast}(\alpha(x))r^{\ast}(y)u^{\ast} $;

\item[(i-2)] the computation
\begin{eqnarray*}
 \langle l^{\ast}(x\cdot y)\beta^{\ast}(u^{\ast}), v \rangle   &=& \langle \beta(l(x\cdot y)(v)), u^{\ast} \rangle
                                     = \langle l(\alpha(x\cdot y))\beta(v), u^{\ast}\rangle
                                   \cr &=& \langle l(\alpha(x)\cdot \alpha(y))\beta(v), u^{\ast} \rangle
                             = \langle l(\alpha^{2}(x))l(\alpha(y))\beta(v), u^{\ast} \rangle
                                        \cr &=& \langle(l(x)l(\alpha(y)))^{\ast}u^{\ast},v \rangle
                                        = \langle l^{\ast}(\alpha(y))l^{\ast}(x)u^{\ast}, v \rangle
\end{eqnarray*}
 gives  $ l^{\ast}(x\cdot y)\beta^{\ast}(u^{\ast})= l^{\ast}(\alpha(y))l^{\ast}(x)u^{\ast} $;

\item[(i-3)] the computation
\begin{multline*}
\langle r^{\ast}(\alpha(x))l^{\ast}(y)u^{\ast}, v\rangle  =                                           \langle l(y)r(\alpha(x))v, u^{\ast} \rangle\cr
=\langle l(\alpha^{2}(y))r(\alpha(x))v, u^{\ast} \rangle =\langle (l\circ\alpha)(\alpha(y))(r\circ\alpha)(x))v, u^{\ast}\rangle \cr 
= \langle r(\alpha^{2}(x))l(\alpha(y))v, u^{\ast}\rangle
=\langle r(x)l(\alpha(y))v, u^{\ast}\rangle\cr
= \langle l^{\ast}(\alpha(y))r^{\ast}(x)u^{\ast},v \rangle
\end{multline*}
 yields $r^{\ast}(\alpha(x))l^{\ast}(y)u^{\ast}= l^{\ast}(\alpha(y))r^{\ast}(x)u^{\ast}$.
 \begin{eqnarray*}
 \langle \beta^{\ast}(r^{\ast}(x))u^{\ast}, v\rangle &=&\langle  r(x)(\beta(v)), u^{\ast}\rangle =\langle  r(\alpha^{2}(x))(\beta(v)), u^{\ast}\rangle\cr
 &=&\langle (r\circ\alpha)(\alpha(x))(\beta(v)), u^{\ast}\rangle
 =\langle \beta(r(\alpha(x)))v, u^{\ast}\rangle\cr
 &=&\langle r^{\ast}(\alpha(x))\beta^{\ast}(u^{\ast}), v\rangle.
 \end{eqnarray*}
 Then $\beta^{\ast}(r^{\ast}(x))u^{\ast}= r^{\ast}(\alpha(x))\beta^{\ast}(u^{\ast}).$
Similarly, one can show that\\ $\beta^{\ast}(l^{\ast}(x))u^{\ast}= l^{\ast}(\alpha(x))\beta^{\ast}(u^{\ast}).$
Hence,  $ (r^{\ast}, l^{\ast}, \beta^{\ast}, V^{\ast}) $ is a bimodule of $  \mathcal{A}.$
\end{enumerate}
{\rm (ii)} Analogously, $(r^{\ast}, 0, \beta^{\ast}, V^{\ast})$ and $ (0, l^{\ast}, \beta^{\ast}, V^{\ast}) $ are shown to be
bimodules of $  \mathcal{A}.$
\qed
\end{proof}
\begin{definition}
Let $(\mathcal{A}, \cdot, \alpha)$ be a Hom-associative algebra, and $B: \mathcal{A}\times\mathcal{A}\rightarrow K$ be a bilinear form on $\mathcal{A}.$ Then,
\begin{enumerate}
\item $B$ is said  to be nondegenerate if
\begin{eqnarray}
\mathcal{A}^{\bot}=\left\lbrace x\in\mathcal{A}/ B(y, x)=0, \forall y\in A \right\rbrace =0;
\end{eqnarray}
\item  $B$ is said to be symmetric if
\begin{eqnarray}
B(x, y)=B(y, x);
\end{eqnarray}
\item  $B$ is said to be  $\alpha$-invariant if
\begin{eqnarray}
B(\alpha(x)\cdot\alpha(y), \alpha(z))= B(\alpha(x), \alpha(y)\cdot\alpha(z)).
\end{eqnarray}
\end{enumerate}
\end{definition}
\begin{definition}
 A Hom-Frobenius algebra is a Hom-associative algebra with a non-degenerate invariant bilinear form.
\end{definition}
\begin{definition}
We call $(\mathcal{A}, \alpha, B)$ a   \textbf{double construction of an
involutive Hom-Frobenius algebra} associated to
$(\mathcal{A}_1, \alpha_{1})$ and $({\mathcal A}_1^*, \alpha^{\ast}_{1})$ if it satisfies the following conditions:
\begin{enumerate}[label=\upshape{(\roman*)}]
\item $ \mathcal{A} = \mathcal{A}_{1}
\oplus \mathcal{A}^{\ast}_{1} $ as the direct sum of vector
spaces;
\item $(\mathcal{A}_1, \alpha_{1})$ and $({\mathcal A}_1^*, \alpha^{\ast}_{1})$ are Hom-associative subalgebras of $
(\mathcal{A}, \alpha)$ with $\alpha=\alpha_{1}\oplus\alpha^{\ast}_{1}$;
\item $B$ is the natural non-degerenate ($\alpha_{1}\oplus\alpha^{\ast}_{1}$)-invariant symmetric
bilinear form on $\mathcal{A}_{1} \oplus \mathcal{A}^{\ast}_{1} $
given, for all $x, y \in \mathcal{A}_{1}, a^{\ast}, b^{\ast} \in \mathcal{A}^{\ast}_{1},$ by
\begin{eqnarray} \label{quadratic form}
 B(x + a^{\ast}, y + b^{\ast}) &=& \langle x, b^{\ast} \rangle +  \langle a^{\ast}, y \rangle,\cr B((\alpha + \alpha^{\ast})(x + a^{\ast}), y + b^{\ast})&=&  B(x + a^{\ast}, (\alpha + \alpha^{\ast})(y + b^{\ast})),
\end{eqnarray}
where $ \langle  , \rangle $ is the natural pair between the vector space $ \mathcal{A}_{1} $ and dual space $ \mathcal{A}^{\ast}_{1}$.
\end{enumerate}
\end{definition}
Let $(\mathcal{A}, \cdot, \alpha)$ be an involutive Hom-associative algebra. Suppose that there is an involutive Hom-associative algebra structure $"\circ"$ on its dual space $\mathcal{A}^{\ast}.$ We construct an involutive Hom-associative algebra structure on the direct sum $\mathcal{A}\oplus\mathcal{A}^{\ast}$ of the underlying vector spaces of $\mathcal{A}$ and $\mathcal{A}^{\ast}$ such that $(\mathcal{A}, \cdot, \alpha)$ and
$(\mathcal{A}^{\ast}, \circ, \alpha^{\ast})$ are Hom-subalgebras,  equipped with the non-degenerate ($\alpha_{1}\oplus\alpha^{\ast}_{1}$)-invariant symmetric bilinear form on
$\mathcal{A}\oplus\mathcal{A}^{\ast}$ given by the equation \eqref{quadratic form}. In other words,
$(\mathcal{A}\oplus\mathcal{A}^{\ast}, \alpha\oplus\alpha^{\ast}, B)$ is an involutive symmetric Hom-associative algebra. Such a construction is called a double construction of an
involutive Hom-Frobenius algebra associated to $(\mathcal{A}, \cdot, \alpha)$ and
$(\mathcal{A}^{\ast}, \circ, \alpha^{\ast})$.
\begin{theorem}\label{Frobenius theorem}
Let $(\mathcal{A}, \cdot, \alpha) $ be an involutive Hom-associative algebra. Suppose that there is an involutive Hom-associative algebra structure $"\circ"$ on its
dual space $ \mathcal{A}^{\ast} $. Then, there is a double construction of an involutive Hom-Frobenius algebra associated to $(\mathcal{A}, \cdot, \alpha)$
and $(\mathcal{A}^{\ast}, \circ, \alpha^{\ast})$ if and only if $(\mathcal{A}, \mathcal{A}^{\ast}, R^{\ast}_{\cdot}, L^{\ast}_{\cdot}, \alpha^{\ast}, R^{\ast}_{\circ},  L^{\ast}_{\circ}, \alpha)$
is a matched pair of involutive Hom-associative algebras.
\end{theorem}
\begin{proof}
Let us consider the following four maps:
\begin{eqnarray*}
&&L^{\ast}_{\cdot}: \mathcal{A} \rightarrow gl(\mathcal{A}^{\ast}), \langle L^{\ast}_{\cdot}(x)u^{\ast}, v \rangle = \langle
L_{\cdot}(x)v, u^{\ast} \rangle = \langle x\cdot v, u^{\ast} \rangle,\cr
&&R^{\ast}_{\cdot} : \mathcal{A} \rightarrow gl(\mathcal{A}^{\ast}), \langle R^{\ast}_{\cdot}(x)u^{\ast}, v \rangle = \langle
R_{\cdot}(x)v, u^{\ast} \rangle = \langle v\cdot x, u^{\ast} \rangle,\cr
&&R^{\ast}_{\circ} : \mathcal{A}^{\ast} \rightarrow gl(\mathcal{A}), \langle R^{\ast}_{\circ}(x^{\ast})u,
v^{\ast} \rangle = \langle R_{\circ}(x^{\ast})v^{\ast}, u \rangle = \langle v^{\ast} \circ x^{\ast}, u\rangle, \cr
&&L^{\ast}_{\circ} : \mathcal{A}^{\ast} \rightarrow gl(\mathcal{A}), \langle L^{\ast}_{\circ}(x^{\ast})u,
 v^{\ast} \rangle = \langle L_{\circ}(x^{\ast})v^{\ast}, u \rangle = \langle x^{\ast} \circ v^{\ast},
  u\rangle,
\end{eqnarray*}
for all $ x, v, u \in \mathcal{A} $, $ x^{\ast}, v^{\ast}, u^{\ast} \in \mathcal{A}^{\ast} $.
If $ (\mathcal{A}, \mathcal{A}^{\ast}, R^{\ast}_{\cdot}, L^{\ast}_{\cdot}, \alpha^{\ast}, R^{\ast}_{\circ},  L^{\ast}_{\circ}, \alpha)$ is a matched pair of multiplicative Hom-associative algebras, then $(\mathcal{A}\bowtie \mathcal{A}^{\ast}, \ast, \alpha + \alpha^{\ast})$ is a multiplicative Hom-associative algebra with the product $ \ast $  given by the equation \eqref{match. pair product}, and the bilinear form $ {B}(\cdot, \cdot) $ defined by the
equation \eqref{quadratic form} is $(\alpha\oplus\alpha^{\ast})$-invariant, that is  $ {B}[(\alpha(x) + \alpha^{\ast}(a^{\ast}))\ast (\alpha(y) + \alpha^{\ast}(b^{\ast})), (\alpha(z) + \alpha^{\ast}(c^{\ast}))] = {B}[\alpha(x) + \alpha^{\ast}(a^{\ast}), (\alpha(y) + \alpha^{\ast}(b^{\ast}))\ast(\alpha(z) + \alpha^{\ast}(c^{\ast}))]$
for all $ x, y \in \mathcal{A}^{\ast},  a^{\ast}, b^{\ast} \in \mathcal{A}^{\ast},$ and $ (x + a^{\ast})\ast
(y + b^{\ast}) = (x\cdot y + l_{{B}}(a)y + r_{{B}}(b)x ) + (a\circ b +
l_{\mathcal{A}}(x)b + r_{\mathcal{A}}(y)a ), $ with $l_{\mathcal{A}} = R^{\ast}_{\cdot}, r_{\mathcal{A}} = L^{\ast}_{\cdot}, l_{{B}} = R^{\ast}_{\circ}, r_{{B}} = L^{\ast} _{\circ} $. Indeed, we
have
\begin{eqnarray*}
&&{B}[(\alpha(x) + \alpha^{\ast}(a^{\ast}))\ast (\alpha(y) + \alpha^{\ast}(b^{\ast})), (\alpha(z) + \alpha^{\ast}(c^{\ast}))]\cr
&&= {B}[(\alpha(x)\cdot\alpha(y) + l_{A^{\ast}}
(\alpha^{\ast}(a^{\ast}))\alpha(y) + r_{A^{\ast}}(\alpha^{\ast}(b^{\ast}))\alpha(x)) + (\alpha^{\ast}(a^{\ast})\circ\alpha^{\ast}(b^{\ast})\cr
 &&+ l_{\mathcal{A}}(\alpha(x))\alpha^{\ast}(b^{\ast}) + r_{\mathcal{A}}(\alpha(y))\alpha^{\ast}(a^{\ast})), \alpha(z) + \alpha^{\ast}(c^{\ast})]
                                                        \cr &&= \langle \alpha(x)\cdot\alpha(y), \alpha^{\ast}(c^{\ast})\rangle +
                                                     \langle \alpha^{\ast}(c^{\ast}) \circ \alpha^{\ast}(a^{\ast}), \alpha(y)\rangle + \langle \alpha^{\ast}(b^{\ast})\circ\alpha^{\ast}(c^{\ast}), \alpha(x)\rangle \cr
&& + \langle\alpha^{\ast}(a^{\ast})\circ\alpha^{\ast}(b^{\ast}), \alpha(z)\rangle + \langle \alpha(z)\cdot\alpha(x), \alpha^{\ast}(b^{\ast})\rangle +
                                                              \langle \alpha(y)\cdot\alpha(z), \alpha^{\ast}(a^{\ast})\rangle
\end{eqnarray*}
and
\begin{eqnarray*}
&&{B}[\alpha(x) + \alpha^{\ast}(a^{\ast}), (\alpha(y) + \alpha^{\ast}(b^{\ast}))\ast(\alpha(z) + \alpha^{\ast}(c^{\ast}))]\cr
&&= {B}[\alpha(x) + \alpha^{\ast}(a^{\ast}), (\alpha(y)\cdot\alpha(z) + l_{\mathcal{A}^{\ast}}(\alpha^{\ast}(b^{\ast}))\alpha(z) + r_{\mathcal{A}^{\ast}}(\alpha^{\ast}(c^{\ast}))\alpha(y))\cr
&& +(\alpha^{\ast}(b^{\ast})\circ\alpha^{\ast}(c^{\ast}) + l_{A}(\alpha(y))\alpha^{\ast}(c^{\ast}) + r_{\mathcal{A}}(\alpha(z))\alpha^{\ast}(b^{\ast}))]\cr
 &&= \langle \alpha(x), \alpha^{\ast}(b^{\ast})\circ\alpha^{\ast}(c^{\ast})\rangle
                                                            + \langle \alpha^{\ast}(c^{\ast}), \alpha(x)\cdot\alpha(y)\rangle + \langle \alpha^{\ast}(b^{\ast}), \alpha(z)\cdot\alpha(x)\rangle \cr
 &&  + \langle \alpha(y)\cdot\alpha(z), \alpha^{\ast}(a^{\ast})\rangle
                     + \langle \alpha^{\ast}(a^{\ast})\circ\alpha^{\ast}(b^{\ast}), \alpha(z)\rangle + \langle \alpha(c^{\ast})\circ\alpha^{\ast}(a^{\ast}), \alpha(y)\rangle.
\end{eqnarray*}
Thus,  $ {B} $ is well $(\alpha\oplus\alpha^{\ast})$-invariant.
 Conversely, set
\begin{eqnarray*}
x \ast a^{\ast} = l_{\mathcal{A}}(x)a^{\ast} + r_{\mathcal{A}^{\ast}}(a^{\ast})x, a^{\ast} \ast x = l_{\mathcal{A}^{\ast}}(a^{\ast})x + r_{\mathcal{A}}(x)a^{\star},
\end{eqnarray*}
for $ x \in \mathcal{A}, a^{\ast} \in \mathcal{A}^{\ast} $.
 Then, $ (\mathcal{A}, \mathcal{A}^{\ast}, R^{\ast}_{\cdot}, L^{\ast}_{\cdot},
 \alpha^{\ast}, R^{\ast}_{\circ},  L^{\ast}_{\circ}, \alpha)$ is a matched pair of
 multiplicative Hom-associative algebras, since the  double construction of the
involutive Hom-Frobenius algebra associated to $ (\mathcal{A}, \cdot, \alpha) $ and $
 (\mathcal{A}^{\ast}, \circ, \alpha^{\ast}) $ produces the equations \eqref{match. pair1}
 - \eqref{match. pair6}.
\qed
\end{proof}
\begin{theorem}\label{homMathched pair's theorem}
Let $ (\mathcal{A}, \cdot, \alpha) $ be an involutive Hom-associative algebra. Suppose that there is an involutive Hom-associative algebra
 structure $ "\circ" $ on its dual space $(\mathcal{A}^{\ast}, \alpha^{\ast})$. Then,  $ (\mathcal{A}, \mathcal{A}^{\ast},
  R^{\ast}_{\cdot}, L^{\ast}_{\cdot}, \alpha^{\ast}, R^{\ast}_{\circ},  L^{\ast}_{\circ}, \alpha)$ is a matched pair of involutive Hom-associative algebras
  if and only if, for any $ x \in \mathcal{A}$ and $ a^{\ast}, b^{\ast} \in \mathcal{A}^{\ast} $,
\small{
\begin{eqnarray} \label{infinitesimal cond.}
R^{\ast}_{\cdot}(\alpha(x))(a^{\ast} \circ b^{\ast}) &=& R^{\ast}_{\cdot}(L^{\ast}_{\circ}(a^{\ast})x)\alpha^{\ast}(b^{\ast}) +
(R^{\ast}_{\cdot}(x)a^{\ast})\circ \alpha^{\ast}(b^{\ast}),
\\
\label{antisymmetric cond.}
 R^{\ast}_{\cdot}(R^{\ast}_{\circ}(a^{\ast})x)\alpha^{\ast}(b^{\ast}) + L^{\ast}_{\cdot}(x)a^{\ast}\circ \alpha^{\ast}(b^{\ast}) &=&
 L^{\ast}_{\cdot}(L^{\ast}_{\circ}(b^{\ast})x)\alpha^{\ast}(a^{\ast}) + \alpha^{\ast}(a^{\ast})\circ (R^{\ast}_{\cdot}(x)b^{\ast}).
\end{eqnarray}
}
\end{theorem}
\begin{proof}
Obviously, \eqref{infinitesimal cond.} gives \eqref{match. pair1}, and
 \eqref{antisymmetric cond.} reduces to \eqref{match. pair5}  when $ l_{\mathcal{A}} = R^{\ast}_{\cdot}, r_{\mathcal{A}} = L^{\ast}_{\cdot},\\ l_{{B}} = l_{\mathcal{A}^{\ast}} = R^{\ast}_{\circ}, r_{{B}} = r_{\mathcal{A}^{\ast}} = L^{\ast}_{\circ} $. Now, show that
\begin{eqnarray*}
\eqref{match. pair1} \Longleftrightarrow  \eqref{match. pair2} \Longleftrightarrow \eqref{match. pair3}  \Longleftrightarrow \eqref{match. pair4}  \cr {\mbox{and}}\;
\eqref{match. pair5} \Longleftrightarrow \eqref{match. pair6}.
\end{eqnarray*}
 Suppose  \eqref{infinitesimal cond.}  and \eqref{antisymmetric cond.} are satisfied and show that one has
\begin{eqnarray*}
L^{\ast}_{\cdot}(\alpha(x))(a^{\ast}\circ b^{\ast}) &=& L^{\ast}_{\cdot}(R^{\ast}_{\circ}(b^{\ast})x)\alpha^{\ast}(a^{\ast}) + \alpha^{\ast}(a^{\ast})\circ (L^{\ast}_{\cdot}(x)b^{\ast}) \cr
R^{\ast}_{\circ}(\alpha^{\ast}(a^{\ast}))(x\cdot y) &=& R^{\ast}_{\circ}(L^{\ast}_{\cdot}(x)a^{\ast})\alpha(y) + (R^{\ast}_{\circ}(a)x)\cdot \alpha(y) \cr
L^{\ast}_{\circ}(\alpha^{\ast}(a^{\ast}))(x\cdot y) &=& L^{\ast}_{\circ}(R^{\ast}_{\cdot}(y)a^{\ast})\alpha(x) + \alpha(x)\cdot (L^{\ast}_{\circ}(a^{\ast})y) \cr
R^{\ast}_{\circ}(R^{\ast}_{\cdot}(x)a^{\ast})\alpha(y) &+& (L^{\ast}_{\circ}(a^{\ast})x)\cdot \alpha(y) - L^{\ast}_{\circ}(L_{\cdot}(y)a^{\ast})\alpha(x) - \alpha(x)\cdot (R^{\ast}_{\circ}(a)y) = 0.
\end{eqnarray*}
We have
\begin{eqnarray*}
\langle R^{\ast}_{\cdot}(x)a^{\ast}, y \rangle &=& \langle L^{\ast}_{\cdot}(y)a^{\ast}, x \rangle = \langle    y\cdot x, a^{\ast}\rangle, \cr
\langle R^{\ast}_{\circ}(b^{\ast})x, a^{\ast} \rangle &=& \langle L^{\ast}_{\circ}(a^{\ast})x, b^{\ast} \rangle = \langle a^{\ast}\circ b^{\ast}, x \rangle,
\cr
\alpha^{\ast}(R^{\ast}_{\cdot}(x)a^{\ast}) &=& R^{\ast}_{\cdot}(\alpha(x))\alpha^{\ast}(a^{\ast}),\quad \alpha^{\ast}(L^{\ast}_{\cdot}(x)a^{\ast})= L^{\ast}_{\cdot}(\alpha(x))\alpha^{\ast}(a^{\ast}),
\cr
\alpha(R^{\ast}_{\circ}(a^{\ast})x) &=& R^{\ast}_{\circ}(\alpha^{\ast}(a^{\ast}))\alpha(x),\quad  \alpha(L^{\ast}_{\circ}(a^{\ast})x)= L^{\ast}_{\circ}(\alpha^{\ast}(a^{\ast}))\alpha(x),
\end{eqnarray*}
for all $ x, y \in \mathcal{A}, a^{\ast}, b^{\ast} \in \mathcal{A}^{\ast} $. Set $\alpha(x)= z,\alpha(y)= t,$ $\alpha^{\ast}(a^{\ast})= c^{\ast}$ and $ \alpha^{\ast}(b^{\ast})= d^{\ast}.$
Then
\begin{enumerate}[label=\upshape{(\roman*)}]
\item  the statement \eqref{match. pair1} $ \Longleftrightarrow $ \eqref{match. pair2} follows from
\begin{eqnarray*}
\langle R^{\ast}_{\cdot}(\alpha(x))(a^{\ast}\circ b^{\ast}), y\rangle &=& \langle y\cdot \alpha(x), a^{\ast}\circ b^{\ast} \rangle = \langle (L_{\cdot}(y)\circ\alpha)x, a^{\ast}\circ b^{\ast} \rangle \cr
&=& \langle x, \alpha^{\ast}(L^{\ast}_{\cdot}(y)(a^{\ast}\circ b^{\ast})) \rangle = \langle L^{\ast}_{\cdot}(\alpha(y))\alpha^{\ast}(a^{\ast}\circ b^{\ast}), x \rangle \cr
&=& \langle L^{\ast}_{\cdot}(\alpha(y))(\alpha^{\ast}(a^{\ast})\circ \alpha^{\ast}(b^{\ast})), x \rangle\cr
 &=&\langle L^{\ast}_{\cdot}(\alpha(y))(c^{\ast}\circ d^{\ast}), x \rangle;\cr
\langle R^{\ast}_{\cdot}(L^{\ast}_{\circ}(a^{\ast})x)\alpha(b^{\ast}), y \rangle &=& \langle y\cdot L^{\ast}_{\circ}(a^{\ast})x, \alpha^{\ast}(b^{\ast}) \rangle = \langle L^{\ast}_{\cdot}(y)(\alpha^{\ast}(b^{\ast})), L^{\ast}_{\circ}(a^{\ast})x \rangle \cr
&=& \langle L^{\ast}_{\circ}(a^{\ast})x, L^{\ast}_{\cdot}(y)(\alpha^{\ast}(b^{\ast})) \rangle\cr
&=& \langle a^{\ast}\circ(L^{\ast}_{\cdot}(y)(\alpha^{\ast}(b^{\ast}))), x\rangle \cr
  &=& \langle \alpha^{\ast}(c^{\ast})\circ(L^{\ast}_{\cdot}(y)(d^{\ast})), x\rangle; \cr
\langle (R^{\ast}_{\cdot}(x)a^{\ast})\circ \alpha^{\ast}(b^{\ast}), y\rangle &=& \langle R^{\ast}_{\circ}(\alpha^{\ast}(b^{\ast}))y, R^{\ast}_{\cdot}(x)a^{\ast}\rangle = \langle a^{\ast},  (R^{\ast}_{\circ}(\alpha^{\ast}(b^{\ast}))y)\cdot x \rangle \cr
&=& \langle L^{\ast}_{\cdot}[R^{\ast}_{\circ}(\alpha^{\ast}(b^{\ast}))y]a^{\ast},  x \rangle = \langle L^{\ast}_{\cdot}( R^{\ast}_{\circ}(d^{\ast})y)\alpha^{\ast}(c^{\ast}), x \rangle;
\end{eqnarray*}
\item the statement \eqref{match. pair2} $ \Longleftrightarrow $ \eqref{match. pair3} follows from
\begin{eqnarray*}
\langle L^{\ast}(\alpha(x))(a^{\ast} \circ b^{\ast}), y\rangle
&=& \langle a^{\ast}\circ b^{\ast}, \alpha(x)\cdot y \rangle
=\langle R^{\ast}_{\circ}(b^{\ast})(\alpha(x)\cdot y), a^{\ast}\rangle \cr
&=& \langle R^{\ast}_{\circ}(\alpha^{\ast}(d^{\ast}))(z\cdot y), a^{\ast}\rangle; \cr
\langle \alpha^{\ast}(a^{\ast})\circ(L^{\ast}_{\cdot}(x)b^{\ast}), y\rangle
&=& \langle \alpha^{\ast}(a^{\ast}), R^{\ast}_{\circ}(L^{\ast}_{\cdot}(x)b^{\ast})y\rangle =\langle a^{\ast}, \alpha[R^{\ast}_{\circ}(L^{\ast}_{\cdot}(x)b^{\ast})y]\rangle\cr
&=&\langle a^{\ast}, R^{\ast}_{\circ}[\alpha^{\ast}(L^{\ast}_{\cdot}(x)b^{\ast})]\alpha(y)\rangle\cr
&=&\langle a^{\ast}, R^{\ast}_{\circ}[L^{\ast}_{\cdot}(\alpha(x))\alpha^{\ast}(b^{\ast})]\alpha(y)\rangle\cr
&=&\langle a^{\ast}, R^{\ast}_{\circ}(L^{\ast}_{\cdot}(z)d^{\ast})\alpha(y)\rangle; \cr
\langle L^{\ast}_{\cdot}(R^{\ast}_{\circ}(b^{\ast})x)\alpha^{\ast}(a^{\ast}), y\rangle
&=&\langle (R^{\ast}_{\circ}(b^{\ast})x)\circ y, \alpha^{\ast}(a^{\ast})\rangle =\langle \alpha[(R^{\ast}_{\circ}(b^{\ast})x)\circ y], a^{\ast}\rangle \cr
&=&\langle (R^{\ast}_{\circ}(\alpha^{\ast}(b^{\ast}))\alpha(x))\circ \alpha(y), a^{\ast}\rangle \cr
&=&\langle R^{\ast}_{\circ}(d^{\ast})z\cdot \alpha(y), a^{\ast}\rangle;
\end{eqnarray*}
\item the statement \eqref{match. pair1} $  \Longleftrightarrow $ \eqref{match. pair4} follows from
\begin{eqnarray*}
\langle R^{\ast}(\alpha(x))(a^{\ast}\circ b^{\ast}), y \rangle &=& \langle a^{\ast}\circ b^{\ast}, y\cdot \alpha(x)\rangle =\langle L_{\cdot}(a^{\ast})b^{\ast}, y\cdot z \rangle \cr
 &=& \langle L^{\ast}_{\circ}(a^{\ast})(y\cdot z)\rangle =\langle L^{\ast}_{\circ}(\alpha^{\ast}(c^{\ast}))(y\cdot z)\rangle;
\cr
\langle (R^{\ast}_{\cdot}(x)a^{\ast})\circ\alpha^{\ast}(b^{\ast}), y\rangle &=& \langle \alpha^{\ast}(b^{\ast}), L^{\ast}_{\cdot}(R^{\ast}_{\cdot}(x)a^{\ast})y\rangle =\langle b^{\ast}, \alpha^{\ast}[L^{\ast}_{\cdot}(R^{\ast}_{\cdot}(x)a^{\ast})y]\rangle \cr
&=& \langle b^{\ast}, L^{\ast}_{\cdot}(R^{\ast}_{\cdot}(\alpha(x))\alpha^{\ast}(a^{\ast}))\alpha(y)\rangle\cr &=&\langle b^{\ast}, L^{\ast}_{\cdot}(R^{\ast}_{\cdot}(z)c^{\ast})\alpha(y)\rangle;
\cr
\langle R^{\ast}_{\cdot}(L^{\ast}_{\circ}(a^{\ast})x)\alpha^{\ast}(b^{\ast}), y\rangle &=&\langle y\cdot L^{\ast}_{\circ}(a^{\ast})x, \alpha^{\ast}(b^{\ast})\rangle =\langle \alpha(y)\cdot \alpha(L^{\ast}_{\circ}(a^{\ast})x), b^{\ast}\rangle \cr
&=&\langle \alpha(y)\cdot L^{\ast}_{\circ}(\alpha^{\ast}(a^{\ast}))\alpha(x), b^{\ast}\rangle =\langle \alpha(y)\cdot L^{\ast}_{\circ}(c^{\ast})z, b^{\ast}\rangle;
\end{eqnarray*}
\item the statement \eqref{match. pair5} $  \Longleftrightarrow $ \eqref{match. pair6} follows from
\begin{eqnarray*}
\langle L^{\ast}_{\cdot}(L^{\ast}_{\circ}(b^{\ast})x)\alpha^{\ast}(a^{\ast}), y\rangle &=& \langle (L^{\ast}_{\circ}(b^{\ast})x)\cdot y, \alpha^{\ast}(a^{\ast})\rangle =\langle a^{\ast}, \alpha(L^{\ast}_{\circ}(b^{\ast})x)\cdot \alpha(y) \rangle \cr
&=& \langle a^{\ast}, L^{\ast}_{\circ}(\alpha^{\ast}(b^{\ast}))\alpha(x)\cdot \alpha(y)\rangle =\langle a^{\ast}, L^{\ast}_{\circ}(d^{\ast})z\cdot \alpha(y)\rangle;\cr
\langle \alpha^{\ast}(a^{\ast})\circ (R^{\ast}_{\cdot}(x)b^{\ast}), y\rangle &=& \langle R^{\ast}_{\circ}(R^{\ast}_{\circ}(x)b^{\ast})y, \alpha^{\ast}(a^{\ast})\rangle = \langle \alpha^{\ast}(a^{\ast})\circ (R^{\ast}_{\cdot}(x)b^{\ast}), y\rangle\cr
&=& \langle \alpha[R^{\ast}_{\circ}(R^{\ast}_{\circ}(x)b^{\ast})y], a^{\ast}\rangle\cr
 &=&\langle R^{\ast}_{\circ}[R^{\ast}_{\circ}(\alpha(x))\alpha^{\ast}(b^{\ast})]\alpha(y), a^{\ast}\rangle\cr
&=&\langle R^{\ast}_{\circ}(R^{\ast}_{\cdot}(z)d^{\ast})\alpha(y), a^{\ast}\rangle; \cr
\langle (L^{\ast}_{\cdot}(x)a^{\ast})\circ \alpha^{\ast}(b^{\ast}), y\rangle &=& \langle R^{\ast}_{\circ}(\alpha^{\ast}(b^{\ast}))y, L^{\ast}_{\cdot}(x)a^{\ast}\rangle =\langle x\cdot (R^{\ast}_{\circ}(d^{\ast})y), a^{\ast}\rangle  \cr
&=& \langle \alpha(z)\cdot (R^{\ast}_{\circ}(d^{\ast})y), a^{\ast}\rangle;                                           \cr
\langle R^{\ast}_{\cdot}(R^{\ast}_{\circ}(a^{\ast})x)\alpha^{\ast}(b^{\ast}), y\rangle &=& \langle y\cdot R^{\ast}_{\circ}(a^{\ast})x, \alpha^{\ast}(b^{\ast})\rangle =\langle \alpha^{\ast}(b^{\ast}), L_{\cdot}(y)(R^{\ast}_{\circ}(a^{\ast})x) \rangle \cr
&=& \langle (L^{\ast}_{\cdot}(y)(d^{\ast}), R^{\ast}_{\circ}(a^{\ast})x \rangle = \langle L^{\ast}_{\cdot}(y)d^{\ast}\circ a^{\ast}, x\rangle \cr
&=& \langle L^{\ast}_{\circ}(L^{\ast}_{\cdot}(y)d^{\ast})x, a^{\ast} \rangle = \langle L^{\ast}_{\circ}(L^{\ast}_{\cdot}(y)d^{\ast})\alpha(z), a^{\ast}\rangle
 \end{eqnarray*}
\end{enumerate}
which completes the proof.  \qed
\end{proof}
\subsection{Antisymmetric infinitesimal Hom-bialgebras}
\label{subsec:dblcnstrinvhomfrobal:asyminfhombial}
Let $ \mathcal{A} $ be a multiplicative Hom-associative algebra.
 Let $ \sigma : \mathcal{A}\otimes \mathcal{A} \rightarrow \mathcal{A}\otimes \mathcal{A}  $ be the exchange operator defined as
\begin{eqnarray*}
\sigma(x \otimes y) = y \otimes x,
\end{eqnarray*} for all $ x, y \in \mathcal{A} $.
\begin{proposition}
Let $(\mathcal{A}, \cdot, \alpha)$ be a multiplicative Hom-associative algebra. Then,\\ $(\alpha\otimes L_{\cdot}, R_{\cdot}\otimes\alpha, \alpha\otimes \alpha, \mathcal{A}\otimes\mathcal{A})$ given, for any $x, a, b\in \mathcal{A}:$ by
\begin{eqnarray*}
(\alpha\otimes L)(x)(a\otimes b)  = (\alpha\otimes L(x))(a\otimes b) = \alpha(a)\otimes x\cdot b, \cr
(R_{\cdot}\otimes \alpha)(x)(a\otimes b)  = (R(x)\otimes \alpha)(a\otimes b) = a\cdot x\otimes \alpha(b),
\end{eqnarray*}
is a bimodule of $\mathcal{A}.$
\end{proposition}
\begin{proof}
Let $x, y, v_{1}, v_{2}\in \mathcal{A}$.
\begin{enumerate}[label=\upshape{(\roman*)}]
\item By formulas for the maps and hom-associativity,
\begin{eqnarray*}
(\alpha\otimes L_{\cdot})(x\cdot y)(\alpha\otimes \alpha)(v_{1}\otimes v_{2})&=& [\alpha\otimes L_{\cdot}(x\cdot y)](\alpha(v_{1})\otimes \alpha(v_{2}))\cr &=& v_{1}\otimes(x\cdot y)\cdot\alpha(v_{2});\cr
(\alpha\otimes L_{\cdot}(\alpha(x)))(\alpha\otimes L_{\cdot}(y))(v_{1}\otimes v_{2})&=& (\alpha\otimes L_{\cdot}(\alpha(x)))(\alpha(v_{1}))\otimes (y\cdot v_{2})\cr &=& v_{1}\otimes\alpha(x)\cdot(y\cdot v_{2})
\end{eqnarray*}
give $(\alpha\otimes L_{\cdot})(x\cdot y)(\alpha\otimes \alpha)(v_{1}\otimes v_{2})= (\alpha\otimes L(\alpha(x)))(\alpha\otimes L(y))(v_{1}\otimes v_{2}).$
\item By formulas for the maps and hom-associativity,
\begin{eqnarray*}
(R_{\cdot}\otimes \alpha)(x\cdot y)(\alpha\otimes \alpha)(v_{1}\otimes v_{2})&=& (R_{\cdot}(x\cdot y)\otimes \alpha)(\alpha(v_{1})\otimes \alpha(v_{2}))\cr &=& \alpha(v_{1})\cdot(x\cdot y)\otimes v_{2}\cr
(R_{\cdot}(\alpha(y))\otimes \alpha)(R_{\cdot}(x)\otimes \alpha)(v_{1}\otimes v_{2})&=& (R_{\cdot}(\alpha(y))((v_{1}\cdot x)\otimes\alpha(v_{2}))\cr &=& (v_{1}\cdot x)\cdot\alpha(y)\otimes v_{2}
\end{eqnarray*}
yield $(R_{\cdot}\otimes \alpha)(x\cdot y)(\alpha\otimes \alpha)(v_{1}\otimes v_{2})= (R_{\cdot}(\alpha(y))\otimes \alpha)(R_{\cdot}(x)\otimes \alpha)(v_{1}\otimes v_{2}).$
\item By formulas for the maps,
\begin{eqnarray*}
(\alpha\otimes L_{\cdot}(\alpha(x)))(R_{\cdot}(y)\otimes\alpha)(v_{1}\otimes v_{2})&=&(\alpha\otimes L_{\cdot}(\alpha(x)))(v_{1}\cdot y\otimes \alpha(v_{2}))\cr &=&\alpha(v_{1})\cdot\alpha(y)\otimes\alpha(x) \cdot \alpha(v_{2});\cr
(R_{\cdot}(\alpha(y))\otimes \alpha)(\alpha\otimes L_{\cdot}(x))(v_{1}\otimes v_{2})&=& (R_{\cdot}(\alpha(y))\otimes \alpha)(\alpha(v_{1})\otimes x\cdot v_{2})\cr &=& \alpha(v_{1})\cdot\alpha(y)\otimes\alpha(x)\cdot\alpha(v_{2})
\end{eqnarray*}
give
$$(\alpha\otimes L_{\cdot}(\alpha(x)))(R_{\cdot}(y)\otimes\alpha)(v_{1}\otimes v_{2})= (R_{\cdot}(\alpha(y))\otimes \alpha)(\alpha\otimes L_{\cdot}(x))(v_{1}\otimes v_{2}).$$
\item By formulas for the maps,
\begin{eqnarray*}
(\alpha\otimes \alpha)(\alpha\otimes L_{\cdot})(v_{1}\otimes v_{2})= (\alpha\otimes \alpha)(\alpha(v_{1})\otimes x\cdot v_{2})= v_{1}\otimes \alpha(x)\cdot \alpha(v_{2});\cr
(\alpha\otimes L_{\cdot}(\alpha(x)))(v_{1}\otimes v_{2})= (\alpha\otimes L_{\cdot}\alpha(x))(\alpha(v_{1})\otimes \alpha(v_{2}))= v_{1}\otimes\alpha(x)\cdot\alpha(v_{2})
\end{eqnarray*}
imply $(\alpha\otimes \alpha)(\alpha\otimes L_{\cdot})(v_{1}\otimes v_{2})= \alpha\otimes L_{\cdot}(\alpha(x)))(v_{1}\otimes v_{2}).$
\item By formulas for the maps,
\begin{eqnarray*}
(\alpha\otimes\alpha)(R_{\cdot}\alpha(x))(v_{1}\otimes v_{2})&=& (\alpha\otimes \alpha)(v_{1}\cdot x\otimes \alpha(v_{2}))\cr &=& \alpha(v_{1})\cdot\alpha(x)\otimes v_{2};\cr
(R_{\cdot}(\alpha(x))\otimes \alpha)(\alpha\otimes \alpha)(v_{1}\otimes v_{2})&=& (R_{\cdot}(\alpha(x))\otimes\alpha)(\alpha(v_{1})\otimes \alpha(v_{2}))\cr &=& \alpha(v_{1})\cdot\alpha(x)\otimes v_{2}
\end{eqnarray*}
yield $(\alpha\otimes\alpha)(R_{\cdot}\alpha(x))(v_{1}\otimes v_{2})= (R_{\cdot}(\alpha(x))\otimes \alpha)(\alpha\otimes \alpha)(v_{1}\otimes v_{2}).$
\end{enumerate}
Hence, the  proof is achieved.
\qed
\end{proof}
\begin{remark}
The quadruple $(L_{\cdot}\otimes\alpha, \alpha\otimes R_{\cdot}, \alpha\otimes \alpha, \mathcal{A}\otimes\mathcal{A})$ is also a bimodule of $\mathcal{A}.$
\end{remark}

\begin{theorem}\label{hombialgebra theorem}
Let $(\mathcal{A},\cdot, \alpha)$ be an involutive Hom-associative algebra.
Suppose there is an involutive Hom-associative algebra structure $ "\circ" $ on its dual space $ \mathcal{A}^{\ast} $
 given by a linear map $ \Delta^{\ast} : \mathcal{A}^{\ast} \otimes \mathcal{A}^{\ast} \rightarrow \mathcal{A}^{\ast} $.
Then, $(\mathcal{A}, \mathcal{A}^{\ast}, R^{\ast}_{\cdot}, L^{\ast}_{\cdot}, \alpha^{\ast}, R^{\ast}_{\circ},  L^{\ast}_{\circ}, \alpha)$ is a matched pair of involutive Hom-associative algebras if and only if $ \Delta: \mathcal{A} \rightarrow \mathcal{A} \otimes \mathcal{A} $ satisfies the
 following conditions:
\begin{eqnarray} \label{infinitesimal identity}
\Delta\circ\alpha(x\cdot y) = (\alpha \otimes L_{\cdot}(x))\bigtriangleup(y) + (R_{\cdot}(y) \otimes \alpha)\bigtriangleup(x),
\\ \label{antisymmetric identity}
(L_{\cdot}(y) \otimes \alpha - \alpha \otimes R_{\cdot}(y))\Delta(x) + \sigma [(L_{\cdot}(x) \otimes \alpha - \alpha \otimes R_{\cdot}(x))\Delta(y)]= 0
\end{eqnarray}
for all $ x, y \in \mathcal{A} $.
\end{theorem}
\begin{proof}
Let $ \lbrace e_{1},... , e_{n} \rbrace $ be a basis of $ \mathcal{A},$ and $  \lbrace e^{\ast}_{1},... , e^{\ast}_{n} \rbrace $ be its dual basis. Set $ e_{i} \cdot e_{j} = \sum^{n}_{k=1}c^{k}_{ij}e_{k} $ and $  e^{\ast}_{i} \circ e^{\ast}_{j} = \sum^{n}_{k=1}f^{k}_{ij}e^{\ast}_{k} $. Therefore, we have $ \Delta(e_{k}) = \sum^{n}_{i,j=1}f^{k}_{ij} e_{i} \otimes e_{j}, $ and
\begin{eqnarray*}
R^{\ast}_{\cdot}(e_{i})e^{\ast}_{j} = \sum^{n}_{k=1}c^{j}_{ki}e^{\ast}_{k}, \  L^{\ast}_{\cdot}(e_{i})e^{\ast}_{j} = \sum^{n}_{k=1}c^{j}_{ik}e^{\ast}_{k}, \ \alpha(e_{i})=\sum^{n}_{q=1}b^{i}_{q}e_{q}, \cr
R^{\ast}_{\circ}(e_{i}^{\ast})e_{j} = \sum^{n}_{k=1}f^{j}_{ki}e_{k}, \   L^{\ast}_{\circ}(e_{i}^{\ast})e_{j} = \sum^{n}_{k=1}f^{j}_{ik}e_{k}, \ \alpha^{\ast}(e^{\ast}_{i})=\sum^{n}_{q=1}b^{\ast i}_{q}e^{\ast}_{q}.
\end{eqnarray*}
We have $\langle \alpha^{\ast}(e^{\ast}_{i}), e_{j}\rangle=b^{\ast j}_{i}=\langle e^{\ast}_{i}, \alpha(e_{j})\rangle=b^{i}_{j}$ which implies $b^{\ast j}_{i}= b^{i}_{j}.$ Also, from the identity $\alpha^{2}=\id$ and $\alpha(e_{i})=\sum^{n}_{k=1}b^{i}_{k}e_{k},$ we get $\sum^{n}_{k=1}\sum^{n}_{l=1}b^{i}_{k}b^{k}_{l}e_{l}=\sum^{n}_{l=1}\delta^{i}_{l}e_{l}=e_{i},$ with $b^{i}_{k}b^{k}_{l}=\delta^{i}_{l}.$
Hence, collecting  the coefficient of $ e_{u} \otimes e_{v} $ (for any i, j, k, m) yields
\begin{eqnarray*}
\Delta(\alpha(e_{m})\cdot \alpha(e_{i}))&=& (\alpha \otimes L_{\cdot}(e_{m}))\Delta(e_{i}) + (R_{\cdot}(e_{i}) \otimes \alpha)\Delta(e_{m})\cr
                                 &=& (\alpha\otimes L(e_{m}))( \sum^{n}_{u,v=1}f^{i}_{uv}e_{u}\otimes e_{v}) + (R(e_{i}) \otimes \alpha)(\sum^{n}_{u,v=1}f^{m}_{uv}e_{u}\otimes e_{v}) \cr
                                  &=&  \sum^{n}_{u,v=1}f^{i}_{uv}\alpha(e_{u})\otimes e_{m}\cdot e_{v} + \sum^{n}_{u,v=1}f^{m}_{uv}e_{u}\cdot e_{i}\otimes \alpha(e_{v})\cr
                                  =\sum^{n}_{u,v=1}f^{i}_{uv}(&\displaystyle{\sum^{n}_{j=1}}& b^{u}_{j}e_{j}) \otimes  (\sum^{n}_{k=1}c^{k}_{mv}e_{k}) +\sum^{n}_{u,v=1}f^{m}_{uv}(\sum^{n}_{j=1}c^{j}_{ui}e_{u})\otimes (\sum^{n}_{k=1}b^{v}_{k}e_{k})  \cr
                                  &=& \sum^{n}_{j, k, u, v=1}(f^{i}_{uv}b^{u}_{j}c^{k}_{mv} + f^{m}_{uv}c^{j}_{ui}b^{v}_{k})e_{j}\otimes e_{k};
                                  \cr
\Delta\circ\alpha(e_{m}\cdot e_{i}) &=& \Delta\circ\alpha(\sum^{n}_{l=1}c^{l}_{mi}e_{l})
                                      = \sum^{n}_{l=1}c^{l}_{mi}\Delta(\alpha(e_{l}))\cr
                                      &=& \sum^{n}_{l=1}c^{l}_{mi}\Delta(\sum^{n}_{q=1}b^{l}_{q}e_{q})
                                    = \sum^{n}_{l, q, j, k=1}c^{l}_{mi}b^{l}_{q} f^{q}_{jk}e_{j}\otimes e_{k},\
\end{eqnarray*}
since $\Delta\circ\alpha(e_{m}\cdot e_{i})=(\alpha\otimes\alpha)\circ\Delta(e_{m}\cdot e_{i}).$
Then, $$ \sum^{n}_{l, u, v, j, k=1}c^{l}_{mi}f^{l}_{uv}b^{u}_{j}b^{v}_{k}e_{j}\otimes e_{k}=\sum^{n}_{l, q, j, k=1}c^{l}_{mi}b^{l}_{q} f^{q}_{jk}e_{j}\otimes e_{k}.$$
We obtain the relation
\begin{eqnarray*}
\sum^{n}_{q=1} c^{l}_{mi}b^{l}_{q} f^{q}_{jk} &=& \sum^{n}_{u, v=1}(f^{i}_{uv}b^{u}_{j}c^{k}_{mv} + f^{m}_{uv}c^{j}_{ui}b^{v}_{k})\cr
&\Updownarrow & \cr
\sum^{n}_{u, v=1}c^{l}_{mi}f^{l}_{uv}b^{u}_{j}b^{v}_{k} &=& \sum^{n}_{u, v=1}(f^{i}_{uv}b^{u}_{j}c^{k}_{mv} + f^{m}_{uv}c^{j}_{ui}b^{v}_{k}),
\end{eqnarray*}
and  the identity given by the coefficient of $ e^{m} $ in
\begin{eqnarray*}
R^{\ast}_{\cdot}(\alpha(e_{i}))(e^{\ast}_{j} \circ e^{\ast}_{k}) &=& R^{\ast}_{\cdot}(L^{\ast}_{\circ}(e^{\ast}_{j})e_{i})\alpha^{\ast}(e^{\ast}_{k}) + (R^{\ast}_{\cdot}(e_{i})e^{\ast}_{j})\circ \alpha^{\ast}(e^{\ast}_{k}). \cr
  &=& R^{\ast}_{\cdot}(\sum^{n}_{u=1} f^{i}_{ju} e_{u})\alpha^{\ast}(e^{\ast}_{v}) + (\sum^{n}_{u=1}
  c^{j}_{ui}e^{\ast}_{u})\circ \alpha^{\ast}(e^{\ast}_{k})\cr
 &=& \sum^{n}_{u=1} f^{i}_{ju}R^{\ast}_{\cdot}(e_{u})(\sum^{n}_{v=1}b^{v}_{k}e^{\ast}_{v}) +  \sum^{u}_{u=1}c^{j}_{ui}(e^{\ast}_{u}\circ (\sum^{n}_{v=1}b^{v}_{k}e^{\ast}_{v})) \cr
  &=& \sum^{n}_{u, v=1} f^{i}_{ju}b^{v}_{k}R^{\ast}_{\cdot}(e_{u})e^{\ast}_{v} +  \sum^{n}_{u, v=1}c^{j}_{ui}b^{v}_{k}(e^{\ast}_{u}\circ e^{\ast}_{v}) \cr
 &=& \sum^{n}_{u, v=1} f^{i}_{ju}b^{v}_{k}(\sum^{n}_{m=1} c^{v}_{mu}e^{\ast}_{m}) + \sum^{n}_{u, v, m=1}c^{j}_{ui}b^{v}_{k}f^{m}_{uvp}e^{\ast}_{m}\cr
  &=& \sum^{n}_{u, v, m=1}(f^{i}_{ju}b^{v}_{k}c^{v}_{mu} + c^{j}_{ui}b^{v}_{k}f^{m}_{uv} )e^{\ast}_{m}; \cr
R^{\ast}_{\cdot}(\alpha(e_{i}))(e^{\ast}_{j} \circ e^{\ast}_{k}) &=& R^{\ast}_{\cdot}(\alpha(e_{i}))(\sum^{n}_{l=1}f^{l}_{jk}e^{\ast}_{l})
  = \sum^{n}_{l=1}f^{l}_{jk} R^{\ast}_{\cdot}(\sum^{n}_{q=1}b^{i}_{q}e_{q})(e^{\ast}_{l}) = \cr
  \sum^{n}_{l, q=1}f^{l}_{jv}b^{u}_{q} R^{\ast}_{\cdot}(e_{q})e^{\ast}_{l}
  &=& \sum^{n}_{l, q=1}f^{l}_{jk}b^{i}_{q}(\sum^{n}_{m=1}c^{l}_{mq}e^{\ast}_{m})
  = \sum^{n}_{l, q, m=1}f^{l}_{jk}b^{i}_{q}c^{l}_{mq}e^{\ast}_{m}.
\end{eqnarray*}
Then, we arrive at
\begin{eqnarray*}
\sum^{n}_{q=1}f^{l}_{jk}b^{i}_{q}c^{l}_{mq} &=& \sum^{n}_{u, v=1}(f^{i}_{ju}b^{v}_{k}c^{v}_{mu} + c^{j}_{ui}b^{v}_{k}f^{m}_{uv})
\cr
&\Updownarrow&
\cr
\sum^{n}_{u, v=1}f^{l}_{uv}b^{u}_{j}b^{v}_{k}c^{l}_{mq} &=& \sum^{n}_{u, v=1}(f^{i}_{ju}b^{v}_{k}c^{v}_{mu} + c^{j}_{ui}b^{v}_{k}f^{m}_{uv}).
\end{eqnarray*}
 Thus, taking $c^{l}_{mi}=b^{i}_{q}c^{l}_{mq},$ $b^{l}_{q}f^{q}_{jk}=f^{l}_{jk},$ $f^{i}_{uv}b^{u}_{j}c^{k}_{mv}=f^{i}_{ju}c^{v}_{mu}b^{v}_{k},$
 we obtain that \eqref{infinitesimal identity} corresponds to \eqref{infinitesimal cond.}.
Similarly, we have
\begin{eqnarray*}
&&(L_{\cdot}(e_{i}) \otimes\alpha - \alpha\otimes R_{\cdot}(e_{i}))\Delta(e_{m}) + \sigma [(L_{\cdot}(e_{m})\otimes \alpha
- \alpha\otimes R_{\cdot}(e_{m}))\Delta(e_{i})] = 0 \Leftrightarrow \cr
&&(L_{\cdot}(e_{i})\otimes\alpha - \alpha\otimes R_{\cdot}(e_{i}))(\sum^{n}_{k,l=1}f^{m}_{lk}e_{l}\otimes e_{k}) \cr
&&+ \sigma [(L_{\cdot}(e_{m})\otimes\alpha - \alpha\otimes R_{\cdot}(e_{m}))
(\sum^{n}_{k,l=1}f^{i}_{lk}e_{l}\otimes e_{k})] = 0
\Leftrightarrow
 \cr
&&
\sum^{n}_{k,l=1}f^{m}_{lk}(e_{i}\cdot e_{l}\otimes \alpha(e_{k}) - \alpha(e_{l})\otimes
e_{k}\cdot e_{i})\cr &&
+ \sigma [\sum^{n}_{k,l=1}f^{i}_{kl}(e_{m}\cdot
e_{l}\otimes\alpha(e_{k})
 - \alpha(e_{l})\otimes e_{k}\cdot e_{m})] = 0
 \Leftrightarrow
\cr
&&
\sum^{n}_{k,l=1}f^{m}_{lk}
((\sum^{n}_{j=1}c^{j}_{il}e_{j})\otimes(\sum^{n}_{p=1}b^{k}_{p}e_{p}) -
(\sum^{n}_{p=1}b^{l}_{p}e_{p})\otimes(\sum^{n}_{j=1}c^{j}_{ki}e_{j}))
\cr &&
+ \sigma
[\sum^{n}_{j,l=1}f^{i}_{kl}((\sum^{n}_{j=1}c^{j}_{ml} e_{j})
\otimes(\sum^{n}_{p=1}b^{k}_{p}e_{p})
 - (\sum^{n}_{p=1}b^{l}_{p}e_{p})\otimes(\sum^{n}_{j=1}c^{j}_{km}e_{j}))] = 0
\Leftrightarrow
 \cr
&&
\sum^{n}_{k,l,j, p=1}(f^{m}_{lk}c^{j}_{il}d^{k}_{p}e_{j}\otimes e_{p} -
f^{m}_{lk}c^{j}_{ki}b^{l}_{p}e_{p}\otimes e_{j})\cr && + \sigma[\sum^{n}_{k,l,j, p=1}
(f^{i}_{kl}c^{j}_{ml}b^{k}_{p}e_{j}\otimes e_{p}
-f^{i}_{kl}c^{j}_{km}b^{l}_{p}e_{p}\otimes e_{j})] = 0;
\cr
 && R^{\ast}_{\cdot}(R^{\ast}_{\circ}(e^{\ast}_{j})e_{i})\alpha^{\ast}
(e^{\ast}_{k}) + (L^{\ast}_{\cdot}(e_{i})e^{\ast}_{j})
\circ\alpha^{\ast}(e^{\ast}_{k})\cr &&= L^{\ast}_{\cdot}(L^{\ast}_{\circ}
(e^{\ast}_{k})e_{i})\alpha^{\ast}(e^{\ast}_{j}) + \alpha^{\ast}(e^{\ast}_{j})\circ
(R^{\ast}_{\cdot}(e_{i})e^{\ast}_{k}) \Leftrightarrow \cr
&&R^{\ast}_{\cdot}(\sum^{n}_{l=1}f^{i}_{lj}e_{l})(\sum^{n}_{p=1}d^{k}_{p}e^{\ast}_{p}) +
(\sum^{n}_{l=1}c^{j}_{il} e^{\ast}_{l})\circ
(\sum^{n}_{p=1}d^{k}_{p}e^{\ast}_{p})\cr &&= L^{\ast}_{\cdot}(\sum^{n}_{l=1}f^{i}_{kl}e_{l})
(\sum^{n}_{q=1}d^{j}_{q}e^{\ast}_{q}) +
(\sum^{n}_{q=1}d^{j}_{q}e^{\ast}_{q})\circ(\sum^{n}_{l=1}c^{k}_{li}e^{\ast}_{l})
\Leftrightarrow \cr
&&\sum^{n}_{l, p=1}f^{i}_{lj}d^{k}_{p}R^{\ast}_{\cdot}(e_{l})e^{\ast}_{p} +  \sum^{n}_{l,
p=1}c^{j}_{il}d^{k}_{p}e^{\ast}_{l}\circ
e^{\ast}_{p}= \sum^{n}_{l, q=1}f^{i}_{kl}d^{j}_{q}L^{\ast}_{\cdot}(e_{l})e^{\ast}_{q} + \sum^{n}_{q, l=1}d^{j}_{q}c^{k}_{li}
e^{\ast}_{q}\circ e^{\ast}_{l} \Leftrightarrow \cr
&&\sum^{n}_{l, p=1}f^{i}_{lj}d^{k}_{p}(\sum^{n}_{m=1}c^{p}_{ml}e^{\ast}_{m}) +  \sum^{n}_{l, p, m=1}c^{j}_{il}d^{k}_{p}f^{m}_{lp}e^{\ast}_{m}= \sum^{n}_{l, q, m=1}f^{i}_{kl}d^{j}_{q}c^{q}_{lm}e^{\ast}_{m} + \sum^{n}_{q, l, m=1}d^{j}_{q}c^{k}_{li}f^{m}_{ql}
e^{\ast}_{m}  \cr
&&\Leftrightarrow \sum^{n}_{l,m, p=1}(f^{i}_{lj}d^{k}_{p}c^{p}_{ml} + f^{m}_{lp}d^{k}_{p}c^{j}_{il})e^{\ast}_{m} =   \sum^{n}_{l,m, q=1}(f^{i}_{kl}d^{j}_{q}c^{q}_{lm} + f^{m}_{ql}d^{j}_{q}c^{k}_{li})e^{\ast}_{m}.
\end{eqnarray*}
Thus, we conclude that
\eqref{antisymmetric identity} corresponds to  \eqref{antisymmetric cond.}.
\qed
\end{proof}
\begin{definition}
Let $(\mathcal{A}, \cdot, \alpha)$ be an involutive Hom-associative algebra.
An \textbf{antisymmetric infinitesimal Hom-bialgebra} structure on $\mathcal{A}$ is a linear map $ \Delta: \mathcal{A} \rightarrow \mathcal{A} \otimes \mathcal{A} $ such that
\begin{enumerate}[label=\upshape{(\roman*)}]
\item $  \Delta^{\ast}: \mathcal{A}^{\ast}\otimes \mathcal{A}^{\ast}  \rightarrow
\mathcal{A}^{\ast}$ defines an involutive Hom-associative algebra structure on $ \mathcal{A}^{\ast} $;
\item $\Delta $ satisfies \eqref{infinitesimal identity} and \eqref{antisymmetric identity}.
\end{enumerate}
We denote such an antisymmetric infinitesimal Hom-bialgebra by $ (\mathcal{A}, \bigtriangleup, \alpha)$ or $(\mathcal{A}, \mathcal{A}^{\ast}, \alpha, \alpha^{\ast})$.
\end{definition}

\begin{corollary}
Let $(\mathcal{A}, \cdot, \alpha)$ and $(\mathcal{A}^{\ast}, \circ, \alpha^{\ast})$ be two involutive associative algebras. Then, the following conditions are equivalent:
\begin{enumerate}[label=\upshape{(\roman*)}]
\item There is a double construction of an involutive Hom-Frobenius algebra associated to $(\mathcal{A}, \cdot, \alpha) $ and $(\mathcal{A}^{\ast}, \circ, \alpha^{\ast})$;
\item $(\mathcal{A}, \mathcal{A}^{\ast}, R^{\ast}_{\cdot}, L^{\ast}_{\cdot}, \alpha^{\ast} R^{\ast}_{\circ},  L^{\ast}_{\circ}, \alpha) $ is a matched pair of involutive associative algebras;
\item $(\mathcal{A}, \mathcal{A}^{\ast}, \alpha, \alpha^{\ast})$ is an antisymmetric infinitesimal Hom-bialgebra.
\end{enumerate}
\end{corollary}
\begin{proof}
From  Theorems \ref{Frobenius theorem} and \ref{hombialgebra theorem}, we have the equivalences.
\qed
\end{proof}
\section{Double constructions of involutive biHom-Frobenius algebras}
\label{sec:dblcnstrinvbihomfrobal}
\subsection{Bimodule and matched pair of biHom-associative algebras}
\label{subsec:dblcnstrinvbihomfrobal:bimodmatchedpbihomassal}
\begin{definition}[\cite{HnkHndSilvdcbihomfrobalg:GrMakMenPan:Bihom1}]
A biHom-associative algebra is a quadruple $(\mathcal{A}, \cdot, \alpha, \beta)$ consisting of a linear space $\mathcal{A}$, $\mathcal{K}$-bilinear map $\cdot$: $\mathcal{A}\otimes\mathcal{A}\rightarrow \mathcal{A}$, linear maps $\alpha, \beta: \mathcal{A}\rightarrow\mathcal{A}$ satisfying, for all $x, y, z\in \mathcal{A}$, the following conditions:
\begin{eqnarray*}
&&\alpha\circ\beta= \beta\circ\alpha,  \mbox { (commutativity)};\cr
&&\alpha(x\cdot y)=\alpha(x)\cdot\alpha(y),\ \beta(x\cdot y)=\beta(x)\cdot\beta(y), \mbox { (multiplicativity)};\cr
&&\alpha(x)\cdot(y\cdot z)=(x\cdot y)\cdot\beta(z), \mbox{ (biHom-associativity)}.
\end{eqnarray*}
\end{definition}
\begin{remark}
If $\alpha=\beta,$ $(\mathcal{A}, \cdot, \alpha, \alpha)$  is a Hom-associative algebra.
\end{remark}
\begin{definition}
A biHom-module is a triple $(M, \alpha, \beta),$ where $M$ is a $\mathcal{K}$-vector space, and $\alpha, \beta: M\rightarrow M$ are two  linear maps.
\end{definition}
\begin{definition}[\cite{HnkHndSilvdcbihomfrobalg:GrMakMenPan:Bihom1}]
 Let $(A, \mu_{\mathcal{A}}, \alpha_{\mathcal{A}}, \beta_{\mathcal{A}})$ be a biHom-associative algebra. A left $\mathcal{A}$-module is a triple $(M, \alpha_{M} , \beta_{M})$, where $M$ is a linear space, $\alpha_{M}$ , $\beta_{M}$ : $M \rightarrow M$ are linear maps, with, in addition,  another linear map: $\mathcal{A}\otimes M \rightarrow M , a\otimes m  \mapsto a\cdot m,$ such that, for all $a, a'\in \mathcal{A}, m \in M:$
 \begin{eqnarray*}
&&\alpha_{M}\circ\beta_{M}= \beta_{M}\circ\alpha_{M},\
\alpha_{M}(a\cdot m)= \alpha_{\mathcal{A}}(a)\cdot \alpha_{M}(m),\cr
&&\beta _{M}(a\cdot m)= \beta_{\mathcal{A}}(a)\cdot \beta_{M}(m),\
\alpha_{\mathcal{A}}(a)\cdot(a'\cdot m)= (aa')\cdot\beta_{M}(m).
\end{eqnarray*}
\end{definition}
Let us give now the definition of bimodule of a biHom-associative algebra.

\begin{definition}
Let $(\mathcal{A}, \cdot, \alpha_{1}, \alpha_{2})$ be a biHom-associative algebra, and let $(V, \beta_{1}, \beta_{2})$ be a biHom-module. Let $ l, r: \mathcal{A} \rightarrow gl(V) $ be two linear maps. The quintuple $(l, r, \beta_{1}, \beta_{2}, V)$ is called a bimodule of $\mathcal{A}$ if
\begin{eqnarray}
\label{Cond.biHom-bimod}
 l(x\cdot y)\beta_{1}(v)&=& l(\alpha_{2}(x))l(y)v,\quad r(x\cdot y)\beta_{2}(v)= r(\alpha_{1}(y))r(x)v, \\
 l(\alpha_{2}(x))r(y)v &=& r(\alpha_{1}(y))l(x)v, \\
\label{biHom-bimodule eq1}
\beta_{1}(l(x)v)&=& l(\alpha_{1}(x))\beta_{1}(v),\quad \beta_{1}(r(x)v)= r(\alpha_{1}(x))\beta_{1}(v),
\\
\label{biHom-bimodule eq2}
\beta_{2}(l(x)v) &=& l(\alpha_{2}(x))\beta_{2}(v),\quad \beta_{2}(r(x)v)= r(\alpha_{2}(x))\beta_{2}(v)
\end{eqnarray}
for all $ x, y \in  \mathcal{A}, v \in V $.
\end{definition}
\begin{proposition}
Let $(l, r, \beta_{1}, \beta_{2}, V)$ be a bimodule of a biHom-associative algebra $(\mathcal{A}, \cdot, \alpha_{1}, \alpha_{2})$. Then, the direct sum $\mathcal{A} \oplus V$ of vector spaces is turned into a biHom-associative algebra  by defining multiplication in $\mathcal{A} \oplus V $ by
\begin{eqnarray*}
&&(x_{1} + v_{1}) \ast (x_{2} + v_{2})=x_{1} \cdot x_{2} + (l(x_{1})v_{2} + r(x_{2})v_{1}),\cr
&&(\alpha_{1}\oplus\beta_{1})(x_{1} + v_{1})=\alpha_{1}(x_{1}) + \beta_{1}(v_{1}), (\alpha_{2}\oplus\beta_{2})(x_{1} + v_{1})=\alpha_{2}(x_{1}) + \beta_{2}(v_{1}),
\end{eqnarray*}
for all $ x_{1}, x_{2} \in  \mathcal{A}, v_{1}, v_{2} \in V$.
\end{proposition}
\begin{proof}
Let $v_{1}, v_{2}, v_{3}\in V$ and $x_{1}, x_{2},x_{3}\in \mathcal{A}.$ Setting and computing
\begin{eqnarray*}\label{bicondit. du Bimod.}
[(x_{1} + v_{1}) \ast (x_{2} + v_{2})]\ast(\alpha_{1}(x_{3})+ \beta_{1}(v_{3}))=\cr(\alpha_{2}(x_{1})+ \beta_{2}(v_{1}))\ast[(x_{2} + v_{2}) \ast (x_{3} + v_{3})],
\end{eqnarray*}
and similarly for the other relations,
give the required conditions.
\qed
\end{proof}
We denote such a biHom-associative algebra by $(\mathcal{A} \oplus V, \ast, \alpha_{1} + \beta_{1}, \alpha_{2} + \beta_{2}),$
or $\mathcal{A} \times_{l, r, \alpha_{1}, \alpha_{2}, \beta_{1}, \beta_{2}} V.$

\begin{example}
Let $(\mathcal{A}, \cdot, \alpha, \beta)$ be a multiplicative biHom-associative algebra.
Then,\\ $ (L_{\cdot}, 0, \alpha, \beta), (0, R_{\cdot}, \alpha, \beta) $ and $ (L_{\cdot}, R_{\cdot}, \alpha, \beta) $ are bimodules of $(\mathcal{A}, \cdot, \alpha, \beta)$.
\end{example}
\begin{proposition}
Let $(l, r, \beta_{1}, \beta_{2}, V)$ be bimodule of a multiplicative biHom-associative algebra  $(\mathcal{A}, \cdot, \alpha_{1}, \alpha_{2})$. Then, $(l\circ\alpha^{n}_{1}, r\circ\alpha^{n}_{2}, \beta_{1}, \beta_{2}, V)$ is a bimodule of $\mathcal{A}$ for any finite integer $n.$
\end{proposition}
\begin{proof}
We have
\begin{eqnarray*}
(l\circ\alpha^{n}_{1})(x\cdot y)\beta_{1}(v)&=& l(\alpha^{n}_{1}(x)\cdot\alpha^{n}_{1}(y))\beta_{1}(v)= l(\alpha_{2}(\alpha^{n}_{1}(x)))l(\alpha^{n}_{1}(y))v\cr
&=& l(\alpha^{n}_{1}(\alpha_{2}(x)))l(\alpha^{n}_{1}(y))v= (l\circ\alpha^{n}_{1})(\alpha_{2}(x))(l\circ\alpha^{n}_{1})(y)v.
\end{eqnarray*}
Similarly, the other relations are established.
\qed
\end{proof}
\begin{example}
Let $(\mathcal{A}, \cdot, \alpha_{1}, \alpha_{2})$ be a multiplicative biHom-associative algebra. Then, $(L_{\cdot}\circ\alpha^{n}_{1}, R_{\cdot}\circ\alpha^{n}_{2}, \alpha_{1}, \alpha_{2}, \mathcal{A})$ is a bimodule of $\mathcal{A}$ for any finite integer $n.$
\end{example}

\begin{theorem}
Let $(\mathcal{A}, \cdot, \alpha_{1}, \alpha_{2}) $ and $(\mathcal{B}, \circ, \beta_{1}, \beta_{2})$ be two biHom-associative
algebras. Suppose there exist linear maps $l_{ \mathcal{A}}, r_{ \mathcal{A}} : \mathcal{A}
\rightarrow gl(\mathcal{B}),$ and $l_{ \mathcal{B}}, r_{ \mathcal{B}} :  \mathcal{B} \rightarrow
gl( \mathcal{A}) $ such that $ (l_{ \mathcal{A}}, r_{ \mathcal{A}},  \beta_{1}, \beta_{2}, \mathcal{B})$ is a bimodule of $\mathcal{A}, $ and $ (l_{ \mathcal{B}}, r_{ \mathcal{B}}, \alpha_{1}, \alpha_{2}, \mathcal{A})$ is a bimodule of $\mathcal{B},$ satisfying, for any $ x, y \in  \mathcal{A}, a,b \in  \mathcal{B}$ the following conditions:
\begin{eqnarray} \label{bimatch. pair1}
l_{ \mathcal{A}}(\alpha_{2}(x))(a\circ b) &=& l_{ \mathcal{A}}(r_{ \mathcal{B}}(a)x)\beta_{1}(b) + (l_{ \mathcal{A}}(x)a) \circ \beta_{1}(b),
\\
\label{bimatch. pair2}
r_{ \mathcal{A}}(\alpha_{1}(x))(a\circ b) &=& r_{ \mathcal{A}}(l_{ \mathcal{B}}(b)x)\beta_{2}(a) + \beta_{2}(a) \circ(r_{ \mathcal{A}}(x)b),
\\
\label{bimatch. pair3}
l_{ \mathcal{B}}(\beta_{2}(a))(x\cdot y)&=& l_{\mathcal{B}}(r_{ \mathcal{A}}(x)a)\alpha_{1}(y) + (l_{ \mathcal{B}}(a)x) \cdot \alpha_{1}(y),
\\
\label{bimatch. pair4}
r_{\mathcal{B}}(\beta_{1}(a))(x\cdot y) &=& r_{ \mathcal{B}}(l_{\mathcal{A}}(y)a)\alpha_{2}(x) + \alpha_{2}(x)\cdot(r_{ \mathcal{B}}(a)y),
\\
\label{bimatch. pair5}
l_{ \mathcal{A}}(l_{\mathcal{B}}(a)x)\beta_{1}(b) &+& (r_{ \mathcal{A}}(x)a)\circ \beta_{1}(b)- \nonumber \\
&& r_{ \mathcal{A}}(r_{ \mathcal{B}}(b)x)\beta_{2}(a) - \beta_{2}(a)\circ(l_{ \mathcal{A}}(x)b) = 0,
\\
\label{bimatch. pair6}
l_{ \mathcal{B}}(l_{ \mathcal{A}}(x)a)\alpha_{1}(y) &+& (r_{ \mathcal{B}}(a)x)\cdot\alpha_{1}(y)- \nonumber \\
&& r_{ \mathcal{B}}(r_{ \mathcal{A}}(y)a)\alpha_{2}(x)- \alpha_{2}(x)\cdot (l_{ \mathcal{B}}(a)y) = 0.
\end{eqnarray}
 Then, there is a biHom-associative algebra
 structure on the direct sum $\mathcal{A}\oplus\mathcal{B}$ of
the underlying vector spaces of $\mathcal{A}$ and $\mathcal{B}$ given by
\begin{equation}\label{bimatch. pair product}
\begin{array}{lll}
(x + a) \ast (y + b) &=& (x\cdot y + l_{ \mathcal{B}}(a)y + r_{ \mathcal{B}}(b)x) + (a\circ b +  l_{ \mathcal{A}}(x)b +  r_{ \mathcal{A}}(y)a)\cr
(\alpha_{1}\oplus\beta_{1})(x + a) &=& \alpha_{1}(x) + \beta_{1}(a), (\alpha_{2}\oplus\beta_{2})(x + a)=\alpha_{2}(x) + \beta_{2}(a)
\end{array}
\end{equation}
for all $ x, y \in  \mathcal{A}, a,b \in  \mathcal{B}$.
\end{theorem}
\begin{proof}
Let $v_{1}, v_{2}, v_{3}\in V$ and $x_{1}, x_{2},x_{3}\in \mathcal{A}.$ Setting and computing
\begin{eqnarray*}
[(x_{1} + v_{1}) \ast (x_{2} + v_{2})]\ast(\alpha_{1}(x_{3})+ \beta_{1}(v_{3}))=\cr(\alpha_{2}(x_{1})+ \beta_{2}(v_{1}))\ast[(x_{2} + v_{2}) \ast (x_{3} + v_{3})],
\end{eqnarray*}
we obtain \eqref{bimatch. pair1}-\eqref{bimatch. pair6}. Then, using the following relations:
\begin{eqnarray*}
\beta_{1}(l_{\mathcal{A}}(x)a)&=& l_{\mathcal{A}}(\alpha_{1}(x))\beta_{1}(a),\ \beta_{1}(r_{\mathcal{A}}(x)a)= r_{\mathcal{A}}(\alpha_{1}(x))\beta_{1}(a),
\\
\beta_{2}(l_{\mathcal{A}}(x)a)&=& l_{\mathcal{A}}(\alpha_{2}(x))\beta_{2}(a),\ \beta_{2}(r_{\mathcal{A}}(x)a)= r_{\mathcal{A}}(\alpha_{2}(x))\beta_{2}(a),
\\
\alpha_{1}(l_{\mathcal{B}}(a)x)&=& l_{\mathcal{B}}(\beta_{1}(a))\alpha_{1}(x),\ \alpha_{1}(r_{\mathcal{B}}(a)x)= r_{\mathcal{B}}(\beta_{1}(a))\alpha_{1}(x),
\\
\alpha_{2}(l_{\mathcal{B}}(a)x)&=& l_{\mathcal{B}}(\beta_{2}(a))\alpha_{2}(x),\ \alpha_{2}(r_{\mathcal{B}}(a)x)= r_{\mathcal{B}}(\beta_{2}(a))\alpha_{2}(x),
\end{eqnarray*}
we show that $\ast$ is a biHom-associative algebra structure.
\qed
\end{proof}
We denote this biHom-associative algebra by
$(\mathcal{A}\bowtie \mathcal{B}, \ast, \alpha_{1} + \beta_{1}, \alpha_{2} + \beta_{2})$ or $  \mathcal{A} \bowtie^{l_{ \mathcal{A}}, r_{ \mathcal{A}}, \beta_{1}, \beta_{2}}_{l_{ \mathcal{B}}, r_{ \mathcal{B}}, \alpha_{1}, \alpha_{2}}  \mathcal{B}.$

\begin{definition}
Let $ (\mathcal{A},\cdot, \alpha_{1}, \alpha_{2})$ and $(\mathcal{B}, \circ, \beta_{1}, \beta_{2})$ be two biHom-associative
algebras. Suppose  there exist linear maps
 $ l_{ \mathcal{A}}, r_{ \mathcal{A}} :  \mathcal{A} \rightarrow gl( \mathcal{B}),$ and
$ l_{ \mathcal{B}}, r_{ \mathcal{B}} :  \mathcal{B} \rightarrow gl( \mathcal{A}) $ such that
$ (l_{ \mathcal{A}}, r_{ \mathcal{A}}, \beta_{1}, \beta_{2}) $ is a bimodule of $\mathcal{A},$ and $ (l_{ \mathcal{B}}, r_{ \mathcal{B}}, \alpha_{1}, \alpha_{2}) $
is a bimodule of $  \mathcal{B} $. Then,
 $ ( \mathcal{A},  \mathcal{B}, l_{ \mathcal{A}}, r_{ \mathcal{A}}, \beta_{1}, \beta_{2}, l_{ \mathcal{B}}, r_{ \mathcal{B}}, \alpha_{1}, \alpha_{2}) $
 is called a \textbf{matched pair of biHom-associative algebras},  if the conditions \eqref{bimatch. pair1} - \eqref{bimatch. pair6} are satisfied.
\end{definition}
\subsection{Double constructions of involutive biHom-Frobenius algebras}
\label{sec:dblcnstrinvbihomfrobal:invbihomfrobal}
Now, we consider the multiplicative biHom-associative algebra $(\mathcal{A}, \cdot, \alpha_{1}, \alpha_{2})$
such that $\alpha_{1}\circ\alpha_{2}=\alpha_{2}\circ\alpha_{1}=\id_{\mathcal{A}}$, i.e, $\alpha^{-1}_{1}=
\alpha_{2},$ and $\alpha^{2}_{1}=\alpha^{2}_{2}=\id_{A}$.
\begin{lemma}
Let $(l, r, \beta_{1}, \beta_{2}, V)$ be a bimodule of $(\mathcal{A}, \cdot,
\alpha_{1}, \alpha_{2})$. Then
\begin{enumerate}[label=\upshape{(\roman*)}]
\item $(r^{\ast}, l^{\ast}, \beta^{\ast}_{2}, \beta^{\ast}_{1}, V^{\ast})$ is a bimodule of $(\mathcal{A}, \cdot, \alpha_{1}, \alpha_{2});$
\item $(r^{\ast}, 0, \beta^{\ast}_{2}, \beta^{\ast}_{1}, V^{\ast})$ and $ (0, l^{\ast}, \beta^{\ast}_{2}, \beta^{\ast}_{1}, V^{\ast})$ are also bimodules of $\mathcal{A}$.
\end{enumerate}
\end{lemma}
\begin{proof} (i) Let $(l, r, \beta_{1}, \beta_{2}, V)$ be a bimodule of an involutive biHom-associative algebra $  (\mathcal{A}, \cdot, \alpha_{1}, \alpha_{2})$. Show that
 $ (r^{\ast}, l^{\ast}, \beta^{\ast}_{2}, \beta^{\ast}_{1}, V^{\ast}) $ is a bimodule of $\mathcal{A}$. Let $ x, y \in  \mathcal{A}, u^{\ast} \in V^{\ast}, v \in V.$
 Then
\begin{enumerate}
\item[(i-1)] the following computation
\begin{multline*}
 \langle r^{\ast}(x\cdot y)\beta^{\ast}_{2}(u^{\ast}), v \rangle   = \langle \beta_{2}(r(x\cdot y)v), u^{\ast} \rangle
  =\langle r(\alpha_{2}(x\cdot y))\beta_{2}(v), u^{\ast} \rangle
                                   \cr = \langle r(\alpha_{2}(x)\cdot \alpha_{2}(y))\beta_{2}(v), u^{\ast}\rangle
                                   =\langle r[\alpha_{1}(\alpha_{2}(y))]r(\alpha_{2}(x))v, u^{\ast} \rangle
                                        \cr = \langle (r(y)r(\alpha_{2}(x)))^{\ast}u^{\ast},v \rangle
                                        = \langle r^{\ast}(\alpha_{2}(x))r^{\ast}(y)u^{\ast}, v \rangle;
\end{multline*}
 leads to  $ r^{\ast}(x\cdot y)\beta^{\ast}_{2}(u^{\ast})= r^{\ast}(\alpha_{2}(x))r^{\ast}(y)u^{\ast} $;
\item[(i-2)] the following computation
\begin{eqnarray*}
\langle l^{\ast}(x\cdot y)\beta_{1}^{\ast}(u^{\ast}), v\rangle   &=& \langle \beta_{1}(l(x\cdot y)(v)), u^{\ast} \rangle =\langle l(\alpha_{1}(x\cdot y))\beta_{1}(v), u^{\ast}\rangle
                                   \cr &=& \langle l(\alpha_{1}(x)\cdot \alpha_{1}(y))\beta_{1}(v), u^{\ast} \rangle\cr
                             &=& \langle l[\alpha_{2}(\alpha_{1}(x))]l(\alpha_{1}(y))\beta(v), u^{\ast} \rangle
                                        \cr &=& \langle(l(x)l(\alpha_{1}(y)))^{\ast}u^{\ast},v \rangle
                                        = \langle l^{\ast}(\alpha_{1}(y))l^{\ast}(x)u^{\ast}, v \rangle
\end{eqnarray*}
 gives  $ l^{\ast}(x\cdot y)\beta^{\ast}_{1}(u^{\ast})= l^{\ast}(\alpha_{1}(y))l^{\ast}(x)u^{\ast} $;
\item[(i-3)] the following computation
\begin{multline*}
\langle r^{\ast}(\alpha_{2}(x))l^{\ast}(y)u^{\ast}, v\rangle  =                                           
\langle l(y)r(\alpha_{2}(x))v, u^{\ast}\rangle \cr 
=\langle (l\circ\alpha_{1})(\alpha_{2}(y))(r\circ\alpha_{2})(x))v, u^{\ast} \rangle
= \langle (r\circ\alpha_{2})(\alpha_{1}(x))(l\circ\alpha_{1})(y)v, u^{\ast}\rangle \cr
=\langle r(x)l(\alpha_{1}(y))v, u^{\ast}\rangle 
= \langle l^{\ast}(\alpha_{1}(y))r^{\ast}(x)u^{\ast},v \rangle
\end{multline*}
 yields    $r^{\ast}(\alpha_{2}(x))l^{\ast}(y)u^{\ast}= l^{\ast}(\alpha_{1}(y))r^{\ast}(x)u^{\ast}$.

Furthermore,
\begin{eqnarray*}
 \langle \beta^{\ast}_{2}(r^{\ast}(x))u^{\ast}, v\rangle &=&\langle r(x)(\beta_{2}(v)), u^{\ast}\rangle =\langle (r\circ\alpha_{1})(\alpha_{2}(x))(\beta_{2}(v)), u^{\ast}\rangle\cr
 &=&\langle \beta_{2}(r(\alpha_{1}(x)))v, u^{\ast}\rangle =\langle r^{\ast}(\alpha_{1}(x))\beta^{\ast}_{2}(u^{\ast}), v\rangle.
 \end{eqnarray*}
 Hence, $\beta^{\ast}_{2}(r^{\ast}(x))u^{\ast}= r^{\ast}(\alpha_{1}(x))\beta^{\ast}_{2}(u^{\ast}).$ By analogy, we establish the other conditions.
Hence,  $(r^{\ast}, l^{\ast}, \beta^{\ast}_{2}, \beta^{\ast}_{1}, V^{\ast})$ is a bimodule of $  \mathcal{A} $.
\end{enumerate}
(ii) Similarly, one can show  that $(r^{\ast}, 0, \beta^{\ast}_{2}, \beta^{\ast}_{1}, V^{\ast})$ and $ (0, l^{\ast}, \beta^{\ast}_{2}, \beta^{\ast}_{1}, V^{\ast}) $
 are  bimodules of $  \mathcal{A} $ as well.
 \qed
\end{proof}
 \begin{definition}
Let $(\mathcal{A}, \cdot, \alpha, \beta)$ be a biHom-associative algebra, and $B: \mathcal{A}\times\mathcal{A}\rightarrow K$ be a bilinear form on $\mathcal{A}.$
$B$ is said $\alpha\beta$-invariant if
\begin{eqnarray}
B(\beta(x)\cdot\alpha(y), \alpha(z))= B(\alpha(x), \beta(y)\cdot\alpha(z)).
\end{eqnarray}
\end{definition}
\begin{definition}
 A biHom-Frobenius algebra is a biHom-associative algebra with a non-degenerate invariant bilinear form.
\end{definition}
\begin{definition}
We call $(\mathcal{A}, \alpha, \beta, B)$ a \textbf{double construction of an involutive biHom-Frobenius algebra} associated to
$(\mathcal{A}_1, \alpha_{1})$ and $({\mathcal A}_1^*, \alpha^{\ast}_{1})$ if it satisfies the conditions:
\begin{enumerate}[label=\upshape{\arabic*)}]
\item $ \mathcal{A} = \mathcal{A}_{1}
\oplus \mathcal{A}^{\ast}_{1} $ as the direct sum of vector
spaces;
\item $(\mathcal{A}_1, \alpha_{1}, \alpha_{2})$ and $({\mathcal A}_1^*, \alpha^{\ast}_{1}, \alpha^{\ast}_{2})$ are biHom-associative subalgebras of $
(\mathcal{A}, \alpha)$ with $\alpha=\alpha_{1}\oplus\alpha^{\ast}_{1}$ and $\beta=\alpha_{2}\oplus\alpha^{\ast}_{2}$;
\item $B$ is the natural non-degerenate $(\alpha_{1}\oplus\alpha^{\ast}_{1})(\alpha_{2}\oplus\alpha^{\ast}_{2})$-invariant symmetric
bilinear form on $ \mathcal{A}_{1} \oplus \mathcal{A}^{\ast}_{1} $
given by
\begin{eqnarray} \label{biquadratic form}
\begin{array}{rcl}
 B(x + a^{\ast}, y + b^{\ast}) &=& \langle x, b^{\ast} \rangle +  \langle a^{\ast}, y \rangle,\cr B((\alpha_{1} + \alpha^{\ast}_{1})(x + a^{\ast}), y + b^{\ast})&=&  B(x + a^{\ast}, (\alpha_{1} + \alpha^{\ast}_{1})(y + b^{\ast})), \cr B((\alpha_{2} + \alpha^{\ast}_{2})(x + a^{\ast}), y + b^{\ast})&=&  B(x + a^{\ast}, (\alpha_{2} + \alpha^{\ast}_{2})(y + b^{\ast}))
\end{array}
\end{eqnarray}
for all  $x, y \in \mathcal{A}_{1}, a^{\ast}, b^{\ast} \in \mathcal{A}^{\ast}_{1},$
where $ \langle  , \rangle $ is the natural pair between the vector space $ \mathcal{A}_{1} $ and its dual space  $ \mathcal{A}^{\ast}_{1} $.
\end{enumerate}
\end{definition}

Let $(\mathcal{A}, \cdot, \alpha_{1}, \alpha_{2})$ be an involutive biHom-associative algebra. Suppose  there also exists an
involutive biHom-associative algebra structure $"\circ"$ on its dual space $\mathcal{A}^{\ast}.$ We
construct an involutive biHom-associative algebra structure on the direct sum $\mathcal{A}\oplus\mathcal{A}^{\ast}$
of the underlying vector spaces of $\mathcal{A}$ and $\mathcal{A}^{\ast}$ such that $(\mathcal{A}, \cdot, \alpha_{1}, \alpha_{2})$ and
$(\mathcal{A}^{\ast}, \circ, \alpha^{\ast}_{1}, \alpha^{\ast}_{1})$ are biHom-subalgebras, and the non-degenerate invariant symmetric bilinear form on
$\mathcal{A}\oplus\mathcal{A}^{\ast}$ is given by \eqref{biquadratic form}. Hence, $(\mathcal{A}\oplus\mathcal{A}^{\ast}, \alpha_{1}\oplus\alpha^{\ast}_{1}, \alpha_{2}\oplus\alpha^{\ast}_{2}, B)$ is a symmetric multiplicative biHom-associative algebra. Such a construction is called a double construction of an involutive biHom-Frobenius algebra
associated to $(\mathcal{A}, \cdot, \alpha_{1}, \alpha_{2})$ and $(\mathcal{A}^{\ast}, \circ,
\alpha^{\ast}_{1}, \alpha^{\ast}_{2})$.

\begin{theorem}\label{biFrobenius theorem}
Let $(\mathcal{A}, \cdot, \alpha_{1}, \alpha_{2}) $ be an involutive biHom-associative algebra. Suppose  there is an involutive biHom-associative algebra structure $ " \circ " $ on its
dual space $ \mathcal{A}^{\ast} $. Then, there is a double construction of an involutive symmetric  biHom-associative algebra associated to $ (\mathcal{A}, \cdot, \alpha_{1}, \alpha_{2})$
and $(\mathcal{A}^{\ast}, \circ, \alpha^{\ast}_{1}, \alpha^{\ast}_{2}) $ if and only if
$(\mathcal{A}, \mathcal{A}^{\ast}, R^{\ast}_{\cdot}, L^{\ast}_{\cdot}, \alpha^{\ast}_{2},
\alpha^{\ast}_{1}, R^{\ast}_{\circ},  L^{\ast}_{\circ}, \alpha_{2}, \alpha_{1})$
is a matched pair of involutive biHom-associative algebras.
\end{theorem}

\begin{proof}
By a similar proof as for  Theorem \ref{Frobenius theorem}, we obtain the results. Let us show that $B$ is well $(\alpha_{1}\oplus\alpha^{\ast}_{1})(\alpha_{2}\oplus\alpha^{\ast}_{2})$-invariant. Let $x, y, z\in\mathcal{A}$ and $a^{\ast}, b^{\ast}, c^{\ast}\in\mathcal{A}^{\ast}.$ We have
\begin{eqnarray*}
&&\mathcal{B}[(\alpha_{2}(x) + \alpha^{\ast}_{2}(a^{\ast}))\ast (\alpha_{1}(y) + \alpha^{\ast}_{1}(b^{\ast})), (\alpha_{1}(z) + \alpha^{\ast}_{1}(c^{\ast}))]\cr
 && \quad = \langle \alpha_{2}(x)\cdot\alpha_{1}(y), \alpha^{\ast}_{1}(c^{\ast})\rangle +
                                                     \langle \alpha^{\ast}_{1}(c^{\ast}) \circ \alpha^{\ast}_{2}(a^{\ast}), \alpha_{1}(y)\rangle + \langle \alpha^{\ast}_{1}(b^{\ast})\circ\alpha^{\ast}_{1}(c^{\ast}), \alpha_{2}(x)\rangle \cr
&& \quad + \langle\alpha^{\ast}_{2}(a^{\ast})\circ\alpha^{\ast}_{1}(b^{\ast}), \alpha_{1}(z)\rangle + \langle \alpha_{1}(z)\cdot\alpha_{2}(x), \alpha^{\ast}_{1}(b^{\ast})\rangle +
                                                              \langle \alpha_{1}(y)\cdot\alpha_{1}(z), \alpha^{\ast}_{2}(a^{\ast})\rangle;
\\
&&\mathcal{B}[\alpha_{1}(x) + \alpha^{\ast}_{1}(a^{\ast}), (\alpha_{2}(y) + \alpha^{\ast}_{2}(b^{\ast}))\ast(\alpha_{1}(z) + \alpha^{\ast}_{1}(c^{\ast}))]\cr
 && \quad = \langle \alpha_{1}(x), \alpha^{\ast}_{2}(b^{\ast})\circ\alpha^{\ast}_{1}(c^{\ast})\rangle
                                                            + \langle \alpha^{\ast}_{1}(c^{\ast}), \alpha_{1}(x)\cdot\alpha_{2}(y)\rangle + \langle \alpha^{\ast}_{2}(b^{\ast}), \alpha_{1}(z)\cdot\alpha_{1}(x)\rangle \cr
 &&  \quad + \langle \alpha_{2}(y)\cdot\alpha_{1}(z), \alpha^{\ast}_{1}(a^{\ast})\rangle
                     + \langle \alpha^{\ast}_{1}(a^{\ast})\circ\alpha^{\ast}_{2}(b^{\ast}), \alpha_{1}(z)\rangle + \langle \alpha_{1}(c^{\ast})\circ\alpha^{\ast}_{1}(a^{\ast}), \alpha_{2}(y)\rangle.
\end{eqnarray*}
Using $\alpha^{2}_{1}=\alpha^{2}_{2}=\id_{\mathcal{A}},$ $\alpha^{\ast 2}_{1}=\alpha^{\ast 2}_{2}=\id_{\mathcal{A}^{\ast}}, \alpha_{1}=\alpha^{-1}_{2}$ and $\alpha^{\ast}_{1}=\alpha^{\ast -1}_{2},$ we obtain
\begin{eqnarray*}
&&\mathcal{B}[(\alpha_{2}(x) + \alpha^{\ast}_{2}(a^{\ast}))\ast (\alpha_{1}(y) + \alpha^{\ast}_{1}(b^{\ast})), (\alpha_{1}(z) + \alpha^{\ast}_{1}(c^{\ast}))]\cr
&&=\mathcal{B}[\alpha_{1}(x) + \alpha^{\ast}_{1}(a^{\ast}), (\alpha_{2}(y) + \alpha^{\ast}_{2}(b^{\ast}))\ast(\alpha_{1}(z) + \alpha^{\ast}_{1}(c^{\ast}))]\cr
&&= \langle x, b^{\ast}\circ c^{\ast}\rangle + \langle c^{\ast}, x\cdot y\rangle + \langle b^{\ast}, z\cdot x\rangle + \langle y\cdot z, a^{\ast}\rangle + \langle a^{\ast}\circ b^{\ast}, z\rangle + \langle c^{\ast}\circ a^{\ast}, y\rangle.
\end{eqnarray*}
This completes the proof. \qed
\end{proof}
\begin{theorem}
Let $ (\mathcal{A}, \cdot, \alpha_{1}, \alpha_{2})$ be an involutive biHom-associative algebra. Suppose  there exists an involutive biHom-associative algebra
 structure $ "\circ" $ on its dual space $(\mathcal{A}^{\ast}, \alpha^{\ast}_{1}, \alpha^{\ast}_{2})$. Then, $ (\mathcal{A}, \mathcal{A}^{\ast},
  R^{\ast}_{\cdot}, L^{\ast}_{\cdot}, \alpha^{\ast}_{2}, \alpha^{\ast}_{1}, R^{\ast}_{\circ},  L^{\ast}_{\circ}, \alpha_{2}, \alpha_{1})$ is a matched pair of involutive biHom-associative algebras
  if and only if, for any $ x \in \mathcal{A}$ and $ a^{\ast}, b^{\ast} \in \mathcal{A}^{\ast} $,
\begin{eqnarray} \label{infinitesimal bicond.}
R^{\ast}_{\cdot}(\alpha_{2}(x))(a^{\ast} \circ b^{\ast}) = R^{\ast}_{\cdot}(L^{\ast}_{\circ}(a^{\ast})x)\alpha^{\ast}_{2}(b^{\ast}) +
(R^{\ast}_{\cdot}(x)a^{\ast})\circ \alpha^{\ast}_{2}(b^{\ast}), \\
\begin{array}{ll}
\label{antisymmetric bicond.}
 R^{\ast}_{\cdot}(R^{\ast}_{\circ}(a^{\ast})x)\alpha^{\ast}_{2}(b^{\ast}) & + L^{\ast}_{\cdot}(x)a^{\ast}\circ \alpha^{\ast}_{2}(b^{\ast})= \\
 & L^{\ast}_{\cdot}(L^{\ast}_{\circ}(b^{\ast})x)\alpha^{\ast}_{1}(a^{\ast}) + \alpha^{\ast}_{1}(a^{\ast})\circ (R^{\ast}_{\cdot}(x)b^{\ast}).
\end{array}
\end{eqnarray}
\end{theorem}
\begin{proof}
By a similar proof as for  Theorem \ref{homMathched pair's theorem}, and using the following valid relations
\begin{eqnarray*}
&&\alpha^{\ast}_{2}(R^{\ast}_{\cdot}(x)a^{\ast})= R^{\ast}_{\cdot}(\alpha_{1}(x))\alpha^{\ast}_{2}(a^{\ast}),\ \alpha^{\ast}_{2}(L^{\ast}_{\cdot}(x)a^{\ast})= L^{\ast}_{\cdot}(\alpha_{1}(x))\alpha^{\ast}_{2}(a^{\ast})\cr
&&\alpha^{\ast}_{1}(R^{\ast}_{\cdot}(x)a^{\ast})= R^{\ast}_{\cdot}(\alpha_{2}(x))\alpha^{\ast}_{1}(a^{\ast}),\ \alpha^{\ast}_{1}(L^{\ast}_{\cdot}(x)a^{\ast})= L^{\ast}_{\cdot}(\alpha_{2}(x))\alpha^{\ast}_{1}(a^{\ast})\cr
&&\alpha_{2}(R^{\ast}_{\circ}(a^{\ast})x)= R^{\ast}_{\circ}(\alpha^{\ast}_{1}(a^{\ast}))\alpha_{2}(x),\  \alpha_{2}(L^{\ast}_{\circ}(a^{\ast})x)= L^{\ast}_{\circ}(\alpha^{\ast}_{1}(a^{\ast}))\alpha_{2}(x)\cr
&&\alpha_{1}(R^{\ast}_{\circ}(a^{\ast})x)= R^{\ast}_{\circ}(\alpha^{\ast}_{2}(a^{\ast}))\alpha_{1}(x),\  \alpha_{1}(L^{\ast}_{\circ}(a^{\ast})x)= L^{\ast}_{\circ}(\alpha^{\ast}_{2}(a^{\ast}))\alpha_{1}(x),
\end{eqnarray*}
the equivalences
\begin{equation*}
\eqref{bimatch. pair1} \Longleftrightarrow  \eqref{bimatch. pair2} \Longleftrightarrow \eqref{bimatch. pair3}  \Longleftrightarrow \eqref{bimatch. pair4}
 {\mbox{\quad and \quad }}\;
\eqref{bimatch. pair5} \Longleftrightarrow \eqref{bimatch. pair6}.
\end{equation*}
are obtained. \qed
\end{proof}
\begin{theorem}\label{bihombialgebra theorem}
Let $(\mathcal{A},\cdot, \alpha_{1}, \alpha_{2})$ be an involutive biHom-associative algebra.
Suppose there is an involutive biHom-associative algebra structure $ "\circ" $ on its dual space $ \mathcal{A}^{\ast} $
 given by a linear map $ \Delta^{\ast} : \mathcal{A}^{\ast} \otimes \mathcal{A}^{\ast} \rightarrow \mathcal{A}^{\ast} $.
Then, \begin{eqnarray*}(\mathcal{A}, \mathcal{A}^{\ast}, R^{\ast}_{\cdot}, L^{\ast}_{\cdot}, \alpha^{\ast}_{2}, \alpha^{\ast}_{1}, R^{\ast}_{\circ},  L^{\ast}_{\circ}, \alpha_{2}, \alpha_{1})\end{eqnarray*} is a matched pair of involutive
 biHom-associative algebras if and only if $ \Delta: \mathcal{A} \rightarrow \mathcal{A} \otimes \mathcal{A} $ satisfies the
 following two conditions:
\begin{eqnarray} \label{infinitesimal bi-identity}
\Delta\circ\alpha_{2}(x\cdot y) = (\alpha_{2}\otimes L_{\cdot}(x))\bigtriangleup(y) + (R_{\cdot}(y)\otimes\alpha_{2})\bigtriangleup(x),
\\
\label{antisymmetric bi-identity}
(L_{\cdot}(y)\otimes\alpha_{2} - \alpha_{1}\otimes R_{\cdot}(y))\Delta(x) + \sigma [(L_{\cdot}(x)\otimes\alpha_{2} - \alpha_{1}\otimes R_{\cdot}(x))\Delta(y)]= 0
\end{eqnarray}
for all $ x, y \in \mathcal{A} $.
\end{theorem}
\begin{proof}
This proof is simillar to  that of Theorem \ref{hombialgebra theorem}.
\qed
\end{proof}
\begin{definition}
Let $(\mathcal{A}, \cdot, \alpha_{1}, \alpha_{2})$ be an involutive biHom-associative algebra. An \textbf{antisymmetric infinitesimal biHom-bialgebra} structure on
 $\mathcal{A}$ is a linear map $ \Delta: \mathcal{A} \rightarrow \mathcal{A} \otimes \mathcal{A} $
  such that
\begin{enumerate}[label=\upshape{(\alph*)}]
\item $  \Delta^{\ast}: \mathcal{A}^{\ast}\otimes \mathcal{A}^{\ast}  \rightarrow
\mathcal{A}^{\ast}$ defines an involutive biHom-associative algebra structure on $ \mathcal{A}^{\ast} $;
\item $\Delta $ satisfies \eqref{infinitesimal bi-identity} and \eqref{antisymmetric bi-identity}.
\end{enumerate}
We denote it by $ (\mathcal{A}, \bigtriangleup, \alpha_{1}, \alpha_{2})$ or $(\mathcal{A}, \mathcal{A}^{\ast}, \alpha_{1}, \alpha_{2}, \alpha^{\ast}_{1}, \alpha_{2}, \alpha^{\ast}_{1})$.
\end{definition}

\begin{corollary}
Let $(\mathcal{A}, \cdot, \alpha_{1}, \alpha_{2})$ and $(\mathcal{A}^{\ast}, \circ, \alpha^{\ast}_{1}, \alpha^{\ast}_{2})$ be two biHom-associative algebras. Then, the following conditions are equivalent:
\begin{enumerate}[label=\upshape{\arabic*)}]
\item There is a double construction of an involutive biHom-Frobenius algebra associated to $(\mathcal{A}, \cdot, \alpha_{1}, \alpha_{2})$ and $(\mathcal{A}^{\ast}, \circ, \alpha^{\ast}_{1}, \alpha^{\ast}_{2})$;
\item
$(\mathcal{A}, \mathcal{A}^{\ast}, R^{\ast}_{\cdot}, L^{\ast}_{\cdot}, \alpha^{\ast}_{2},
\alpha^{\ast}_{1}, R^{\ast}_{\circ},  L^{\ast}_{\circ}, \alpha_{2}, \alpha_{1}) $ is a
matched pair of multiplicative biHom-associative algebras;
\item $(\mathcal{A}, \mathcal{A}^{\ast}, \alpha_{1}, \alpha_{2}, \alpha^{\ast}_{1},
\alpha^{\ast}_{2})$ is an antisymmetric infinitesimal biHom-bialgebra.
\end{enumerate}
\end{corollary}
\begin{proof}
From  Theorems \ref{biFrobenius theorem} and \ref{bihombialgebra theorem}, we have the equivalences.
\qed
\end{proof}
\section{Double constructions of involutive symplectic Hom-associative algebras}
\label{sec:dblcnstrinvsymhomassal}
\subsection{Hom-dendriform algebras}
\label{subsec:dblcnstrinvsymhomassal:homdendal}
\begin{definition}
A Hom-dendriform algebra is a quadruple $(\mathcal{A}, \prec, \succ, \alpha)$ consisting of a vector space $\mathcal{A}$ on which the operations $\prec, \succ: \mathcal{A}\otimes \mathcal{A}\rightarrow \mathcal{A},$ and $\alpha: \mathcal{A}\rightarrow \mathcal{A}$ are linear maps satisfying
\begin{eqnarray*}
(x \prec y)\prec \alpha(z) &=& \alpha(x)\prec (y \ast z),\cr
(x \succ y) \prec\alpha(z) &=& \alpha(x) \succ (y \prec z),\cr
\alpha(x)\succ (y \succ z ) &=& ( x \ast y) \succ \alpha(z),
\end{eqnarray*}
where
\begin{eqnarray}\label{associative-dendriform}
x \ast y = x \prec y + x \succ y.
\end{eqnarray}
\end{definition}
\begin{definition}
Let $(\mathcal{A}, \prec, \succ, \alpha)$ and $(\mathcal{A}', \prec',  \succ', \alpha')$ be two Hom-dendriform algebras. A linear map $f: \mathcal{A}\rightarrow \mathcal{A}'$ is a Hom-dendriform algebra morphism if
\begin{eqnarray*}
\prec'\circ(f\otimes f)= f\circ\prec,\ \succ'\circ(f\otimes f)= f\circ\succ \mbox{ and } f\circ\alpha= \alpha'\circ f.
\end{eqnarray*}
\end{definition}
\begin{proposition}
Let $(\mathcal{A}, \prec, \succ, \alpha)$ be a Hom-dendriform algebra. Then, $(\mathcal{A}, \ast, \alpha)$ is a Hom-associative algebra.
\end{proposition}
\begin{proof}
For all $x, y, z\in\mathcal{A}$,
\begin{eqnarray*}
(x\ast y)\ast\alpha(z)
&=& (x\prec y)\prec\alpha(z) + (x\prec y)\succ\alpha(z) + (x\succ y)\prec\alpha(z) + (x\succ y)\succ\alpha(z)\cr
&=& (x\prec y)\prec\alpha(z) + (x\succ y)\prec\alpha(z) + (x\ast y)\succ\alpha(z)\cr
&=& \alpha(x)\prec(y\ast z) + \alpha(x)\succ(y\prec z) + \alpha(x)\succ(y\succ z)\cr
&=& \alpha(x)\prec(y\ast z) + \alpha(x)\succ(y\ast z)
= \alpha(x)\ast(y\ast z),
\end{eqnarray*}
which completes the proof.
\qed
\end{proof}
 We call $ (\mathcal{A}, \ast, \alpha)$ the associated Hom-associative algebra of $ (\mathcal{A}, \prec, \succ, \alpha),$ and $ (\mathcal{A}, \succ, \prec, \alpha)$ is called a compatible Hom-dendriform algebra structure on the Hom-associative algebra $ (\mathcal{A}, \ast, \alpha) $.

Let $ (\mathcal{A}, \prec, \succ, \alpha) $ be a Hom-dendriform algebra. For any $ x \in \mathcal{A} $,  let $ L_{\succ}(x),  R_{\succ}(x) $ and $ L_{\prec}(x),$ $ R_{\prec}(x) $ denote the left and right multiplication operators of $(\mathcal{A}, \prec)$ and $(\mathcal{A}, \succ)$, respectively, i. e.,
\begin{eqnarray*}
L_{\succ}(x) y = x \succ y, R_{\succ}(x) y = y \succ x, L_{\prec}(x) y = x \prec y,
L_{\prec}(x) y = y \prec x,
\end{eqnarray*}
for all $ x, y \in \mathcal{A} $. Moreover, let $ L_{\succ}, R_{\succ}, L_{\prec}, R_{\prec} : \mathcal{A} \rightarrow gl(\mathcal{A}) $ be four linear maps with $ x \mapsto L_{\succ}(x),  x \ \mapsto R_{\succ}(x),  x  \mapsto L_{\prec}(x), $ and $  x  \mapsto R_{\prec}(x)  $, respectively.
\begin{proposition}
The quadruple  $ (L_{\succ}, R_{\prec}, \alpha, \mathcal{A}) $ is a bimodule of the associated Hom-associative algebra $(\mathcal{A}, \ast, \alpha)$.
\end{proposition}
\begin{proof}
For all $x, y, v\in\mathcal{A},$
\begin{eqnarray*}
L_{\succ}(x\ast y)\alpha(v) &=& (x\ast y)\succ\alpha(v)=\alpha(x)\succ(y\succ v)=L_{\succ}(\alpha(x))L_{\succ}(y)v ,\\
R_{\prec}(x\ast y)\alpha(v) &=& \alpha(v)\prec(x\ast y)= (v\prec x)\prec\alpha(y)=R_{\prec}(\alpha(y))R_{\prec}(x)v , \\
L_{\succ}(\alpha(x))R_{\prec}(y)v &=& \alpha(x)\succ(v\prec y)=(x\succ v)\prec\alpha(y)= R_{\prec}(\alpha(y))L_{\succ}(x)v ,  \\
\alpha(L_{\succ}(x)v) &=& \alpha(x\succ v)= \alpha(x)\succ\alpha(v)=L_{\succ}(\alpha(x))\alpha(v) , \\
\alpha(R_{\prec}(x)v)&=& \alpha(v\prec x)= \alpha(v)\prec\alpha(x)=R_{\prec}(\alpha(x))\alpha(v),
\end{eqnarray*}
which completes the proof.
\qed
\end{proof}
\subsection{$ \mathcal{O} $-operators and Hom-dendriform algebras}
\label{subsec:dblcnstrinvsymhomassal:oopshomdendal}
\begin{definition}
Let $(\mathcal{A}, \cdot, \alpha)$ be a Hom-associative algebra, and $(l, r, \beta, V)$ be a bimodule. Then, a  linear map $ T : V \rightarrow \mathcal{A} $
is called an \textbf{$ \mathcal{O} $-operator associated} to $(l, r, \beta, V)$,  if $ T $ satisfies
\begin{eqnarray*}
\alpha T= T\beta \mbox{ and } T(u)\cdot T(v) = T(l(T(u))v + r(T(v))u) \mbox { for all } u, v \in V.
\end{eqnarray*}
\end{definition}

\begin{example}
Let $(\mathcal{A}, \cdot, \alpha)$ be a multiplicative Hom-associative algebra. Then,
 the identity map $\id$ is an $ \mathcal{O} $-operator associated to the bimodule $(L,0, \alpha)$ or $(0,R, \alpha)$.
\end{example}

\begin{example}
Let $(\mathcal{A}, \cdot, \alpha)$ be a multiplicative Hom-associative algebra.
A linear map $ f : \mathcal{A} \rightarrow \mathcal{A} $ is called a \textbf{Rota-Baxter operator} on $ \mathcal{A} $ of weight zero if $ f $
satisfies
\begin{eqnarray*}
f(x)\cdot f(y) = f(f(x)\cdot y + x\cdot f(y)) \mbox { for all } x, y \in \mathcal{A}.
\end{eqnarray*}
In fact, a Rota-Baxter operator on $ \mathcal{A} $ is just an $ \mathcal{O} $-operator associated to the bimodule $(L, R, \alpha)$.
\end{example}
\begin{theorem}
Let $(\mathcal{A}, \cdot, \alpha)$ be a Hom-associative algebra, and $(l, r, \beta, V) $ be a bimodule.
Let $ T : V \rightarrow \mathcal{A} $ be an $ \mathcal{O} $-operator associated to $(l, r, \beta, V)$. Then, there exists a Hom-dendriform
algebra structure on $ V $ given by
\begin{eqnarray*}
 u \succ v = l(T(u))v ,  u \prec v = r(T(v))u
\end{eqnarray*}
for all $u, v \in V$. So, there is an associated Hom-associative algebra structure on $ V $ given by the equation \eqref{associative-dendriform},  and $ T $
is a homomorphism of Hom-associative algebras. Moreover, $ T(V) = \lbrace { T(v) \setminus v \in V }  \rbrace  \subset \mathcal{A} $ is a Hom-associative
subalgebra of $ \mathcal{A}, $ and there is an induced Hom-dendriform algebra structure on $ T(V) $ given by
\begin{eqnarray}
T(u) \succ T(v) = T(u \succ v), T(u) \prec T(v) = T(u \prec v)
\end{eqnarray}
for all $ u, v \in V $. Its corresponding associated Hom-associative algebra structure on $ T(V) $ given by the equation \eqref{associative-dendriform} is
just the Hom-associative subalgebra structure of $ \mathcal{A}, $ and $ T $ is a homomorphism of Hom-dendriform algebras.
\end{theorem}
\begin{proof}
For any $x, y, z\in V,$ we have
\begin{eqnarray*}
&&(x\succ y)\prec\beta(z) - \beta(x)\succ(y\prec z)\cr
&&\quad =l(T(x)y)\prec\beta(z)-\beta(x)\succ r(T(z)y)\cr
&&\quad =r(T\beta(z))l(T(x))y-l(T\beta(x)y)r(T(z)y)\cr
&&\quad =r(\alpha(T(z)))l(T(x))y-l(\alpha(T(x)))r(T(z))y = 0.
\end{eqnarray*}
The two other axioms are  similarly checked.
\qed
\end{proof}
\begin{corollary}\label{dendriform-invertible operator}
Let $(\mathcal{A}, \ast, \alpha)$ be a multiplicative Hom-associative algebra. Then, there is a compatible multiplicative Hom-dendriform algebra structure on $ \mathcal{A} $
if and only if there exists an invertible $ \mathcal{O} $-operator of $ (\mathcal{A}, \ast, \alpha)$.
\end{corollary}
\begin{proof}
If $ T $ is an invertible $ \mathcal{O}- $operator associated to a bimodule $(l, r, \beta, V)$, then,
the compatible multiplicative Hom-dendriform algebra structure on $ \mathcal{A} $ is given by
\begin{eqnarray*}
x \succ y = T(l(x)T^{-1}(y)), x \prec y = T(r(y)T^{-1}(x)) \mbox { for all } x, y \in \mathcal{A}.
\end{eqnarray*}
Conversely, let $(\mathcal{A}, \succ, \prec, \alpha)$ be a Hom-dendriform algebra, and $(\mathcal{A}, \ast, \alpha)$ be
 the associated multiplicative Hom-associative algebra. Then, the identity map $\id$ is an $ \mathcal{O}- $operator
associated to the bimodule $(L_{\succ}, R_{\prec}, \alpha)$ of $(\mathcal{A}, \ast, \alpha)$.
\qed
\end{proof}
\subsection{Bimodules and matched pairs of Hom-dendriform algebras}
\label{subsec:dblcnstrinvsymhomassal:bimodmatchedphomdendal}
\begin{definition}
Let $(\mathcal{A}, \succ, \prec, \alpha)$ be a Hom-dendriform algebra, and $ V $ be a vector space.
Let $l_{\succ}, r_{\succ}, l_{\prec}, r_{\prec} : \mathcal{A} \rightarrow gl(V) $ and $\beta: V \rightarrow V$ be  linear maps. Then, the sextuple
( $ l_{\succ}, r_{\succ}, l_{\prec}, r_{\prec}, \beta, V$) is called a \textbf{bimodule} of $ \mathcal{A} $
 if the following equations hold, for any $ x, y \in \mathcal{A} $ and $v\in V$:
\begin{eqnarray*}
&&l_{\prec}(x \prec y)\beta(v)= l_{\prec}(\alpha(x))l_{\ast}(y)v, r_{\prec}(\alpha(x))l_{\prec}(y)v=l_{\prec}(\alpha(y))r_{\ast}(x)v,\cr
&&r_{\prec}(\alpha(y))r_{\prec}(y)v= r_{\prec}(x\ast y)\beta(v), l_{\prec}(x \succ y)\beta(v) = l_{\succ}(\alpha(x))l_{\prec}(y)v,\cr
&& r_{\prec}(\alpha(x))l_{\succ}(y)v= l_{\succ}(\alpha(y))r_{\prec}(x)v, r_{\prec}(\alpha(x))r_{\succ}(y)v = r_{\succ}(y\prec x)\beta(v),
\cr
&&l_{\succ}(x\ast y)\beta(v)= l_{\succ}(\alpha(x))l_{\succ}(y)v, r_{\succ}(\alpha(x))l_{\ast}(y)v= l_{\succ}(\alpha(y))r_{\succ}(x)v,\cr
&&r_{\succ}(\alpha(x))r_{\ast}(y)v= r_{\succ}(y \succ x)\beta(v),\cr
&&\beta(l_{\succ}(x)v)=l_{\succ}(\alpha(x))\beta(v), \beta(l_{\prec}(x)v)=l_{\prec}(\alpha(x))\beta(v),\cr
&&\beta(r_{\succ}(x)v)=r_{\succ}(\alpha(x))\beta(v), \beta(r_{\prec}(x)v)=r_{\prec}(\alpha(x))\beta(v),
\end{eqnarray*}

where $ x \ast y = x \succ y + x \prec y, \;l_{\ast} = l_{\succ} + l_{\prec},\; r_{\ast} = r_{\succ} + r_{\prec} $.
\end{definition}
\begin{proposition}
Let $(l_{\succ}, r_{\succ}, l_{\prec}, r_{\prec}, \beta, V)$ be a bimodule of a Hom-dendriform algebra $(\mathcal{A},\succ, \prec, \alpha)$. Then, there exists a Hom-dendriform algebra structure on the direct sum $\mathcal{A}\oplus V $ of the underlying vector spaces of $ \mathcal{A}$ and $ V $ given by
\begin{eqnarray*}
(x + u) \succ (y + v) &=& x \succ y + l_{\succ}(x)v + r_{\succ}(y)u, \cr
(x + u) \prec (y + v) &=& x \prec y + l_{\prec}(x)v + r_{\prec}(y)u,\cr
(\alpha\oplus\beta)(x + u)&=&\alpha(x) + \beta(u)
\end{eqnarray*}
for all $ x, y \in \mathcal{A}, u, v \in V $.
\end{proposition}
\begin{proof}
By a straightforward calculation, we obtain the  result. \qed
\end{proof}
We denote this algebra by $ \mathcal{A} \times_{l_{\succ},r_{\succ}, l_{\prec}, r_{\prec}, \alpha, \beta} V$.
\begin{proposition}
Let ($l_{\succ}, r_{\succ}, l_{\prec}, r_{\prec}, \beta, V$) be a bimodule of a Hom-dendriform algebra $(\mathcal{A},
\succ, \prec, \alpha)$. Let $(\mathcal{A}, \ast, \alpha)$ be the associated Hom-associative algebra. Then the following statements hold. 
\begin{enumerate}[label=\upshape{\arabic*)}]
\item $ (l_{\succ}, r_{\prec}, \beta, V) $ and $ (l_{\succ} + l_{\prec}, r_{\succ} + r_{\prec}, \beta, V)$ are bimodules of $ (\mathcal{A}, \ast, \beta). $
\item For any bimodule $(l, r, \beta, V)$ of $(\mathcal{A}, \ast, \alpha)$, 
$(l, 0, 0, r, \beta, V)$ is a bimodule \\ of $(\mathcal{A}, \succ, \prec, \alpha). $
\item $(l_{\succ} + l_{\prec}, 0, 0,  r_{\succ} + r_{\prec}, \beta, V)$ and $(l_{\succ}, 0, 0, r_{\prec}, \beta, V)$ are bimodules \\
 of $(\mathcal{A}, \succ, \prec, \alpha).$
\item  The dendriform algebras $ \mathcal{A} \times_{l_{\succ}, r_{\succ}, l_{\prec}, r_{\prec}, \alpha, \beta} V $ and  $ \mathcal{A} \times_{l_{\succ} +  l_{\prec} , 0, 0, r_{\succ} + r_{\prec}, \alpha, \beta} V $ have the same associated
 Hom-associative algebra $\mathcal{A} \times_{l_{\succ} +  l_{\prec}, r_{\succ} + r_{\prec}, \alpha, \beta} V.$
 \end{enumerate}
\end{proposition}
\begin{proof}
It results from a direct computation.
\qed
\end{proof}
\begin{theorem}
Let $(\mathcal{A}, \succ_{\mathcal{A}}, \prec_{\mathcal{A}}, \alpha)$ and $(\mathcal{B}, \succ_{\mathcal{B}}, \prec_{\mathcal{B}}, \beta)$
 be two Hom-dendriform algebras. Suppose  there are linear maps
$ l_{\succ_{\mathcal{A}}},   r_{\succ_{\mathcal{A}}},  l_{\prec_{\mathcal{A}}},  r_{\prec_{\mathcal{A}}} : \mathcal{A} \rightarrow gl(\mathcal{B}),$
and $ l_{\succ_{\mathcal{B}}},   r_{\succ_{\mathcal{B}}},  l_{\prec_{\mathcal{B}}},  r_{\prec_{\mathcal{B}}} : \mathcal{B} \rightarrow gl(\mathcal{A})$
such that ($ l_{\succ_{\mathcal{A}}},   r_{\succ_{\mathcal{A}}},  l_{\prec_{\mathcal{A}}},  r_{\prec_{\mathcal{A}}}, \beta, \mathcal{B}$) is a bimodule of $\mathcal{A},$ and  ($l_{\succ_{\mathcal{B}}},   r_{\succ_{\mathcal{B}}},  l_{\prec_{\mathcal{B}}},  r_{\prec_{\mathcal{B}}}, \alpha, \mathcal{A})$ ) is a bimodule  of $\mathcal{B},$
 satisfying the following  relations:
\begin{eqnarray} \label{eq35}
r_{\prec_{\mathcal{A}}}(\alpha(x))(a \prec_{\mathcal{B}} b) = \beta(a)\prec_{\mathcal{B}}( r_{\mathcal{A}}(x)b) + r_{\prec_{\mathcal{A}}}(l_{\mathcal{B}}(x)\beta(a)),
\\
\label{eq36}
\begin{array}{ll}
l_{\prec_{\mathcal{A}}}(l_{\prec_{\mathcal{B}}}(x))\beta(b) &+ (r_{\prec_{\mathcal{A}}}(x)a) \prec_{\mathcal{B}}\beta(b)= \cr
& \beta(a) \prec_{\mathcal{B}} (l_{\prec_{\mathcal{A}}}(x)b) + r_{\prec_{\mathcal{A}}}(r_{\prec_{\mathcal{B}}}(b)x)\beta(a),
\end{array}
\\
\label{eq37}
l_{\prec_{\mathcal{A}}}(\alpha(x))(a \ast_{\mathcal{B}} b) = (l_{\prec_{\mathcal{A}}}(x)a) \ast_{\mathcal{B}} \beta(b) +
 l_{\prec_{\mathcal{A}}}(r_{\prec_{\mathcal{A}}}(a)x)\beta(b),
\\
\label{eq38}
r_{\prec_{\mathcal{A}}}(\alpha(x))(a \succ_{\mathcal{B}} b) = r_{\succ_{\mathcal{A}}}(l_{\prec_{\mathcal{B}}}(b)x)\beta(a) +
\beta(a)\succ_{\mathcal{B}} (r_{\prec_{\mathcal{A}}}(x)b),
\\
\label{eq39}
\begin{array}{ll}
l_{\prec_{\mathcal{A}}}(l_{\succ_{\mathcal{B}}}(a)x)\beta(b) &+ (r_{\succ_{\mathcal{A}}}(x)a) \prec_{\mathcal{B}}\beta(b)= \cr
& \beta(a)\succ_{B} (l_{\prec_{\mathcal{A}}}(x)b) + r_{\succ_{\mathcal{A}}}(r_{\prec_{\mathcal{B}}}(b)x)\beta(a)
\end{array}
\\
\label{eq40}
l_{\succ_{\mathcal{A}}}(\alpha(x))(a \prec_{\mathcal{B}} b) = ( l_{\succ_{\mathcal{A}}}(x)a) \prec_{\mathcal{B}}\beta(b) +
 l_{\prec_{\mathcal{A}}}(r_{\succ_{\mathcal{B}}}(a)x)\beta(b),
\\
\label{eq41}
r_{\succ_{\mathcal{A}}}(\alpha(x))(a \ast_{\mathcal{B}} b)= \beta(a)\succ_{\mathcal{B}} (r_{\succ_{\mathcal{A}}}(x)b) +
 r_{\succ_{\mathcal{A}}}(l_{\succ_{\mathcal{B}}}(b)x)\beta(a),
\\
\label{eq42}
\begin{array}{ll}
\beta(a)\succ_{\mathcal{B}} (l_{\succ_{\mathcal{A}}}(x)b) &+ r_{\succ_{\mathcal{A}}}(r_{\succ_{\mathcal{B}}}(b)x)\beta(a) =\cr
& l_{\succ_{\mathcal{A}}}(l_{\mathcal{B}}(a)x)\beta(b) + (r_{\mathcal{A}}(x)a) \succ_{\mathcal{B}}\beta(b),
\end{array}
\\
\label{eq43}
l_{\succ_{\mathcal{A}}}(\alpha(x))(a \succ_{\mathcal{B}} b) = (l_{\mathcal{A}}(x)a) \succ_{\mathcal{B}}\beta(b) + l_{\succ_{\mathcal{A}}}(r_{\mathcal{B}}(a)x)\beta(b),
\\
\label{eq44}
r_{\prec_{\mathcal{B}}}(\beta(a))(x \prec_{\mathcal{A}} y) = \alpha(x)\prec_{\mathcal{A}} (r_{\mathcal{B}}(a)y) + r_{\prec_{\mathcal{B}}}(l_{\mathcal{A}}(y)a)\alpha(x),
\\
\label{eq45}
\begin{array}{ll}
l_{\prec_{B}}(l_{\prec_{\mathcal{A}}}(x)a)\alpha(y) &+ (r_{\prec_{\mathcal{B}}}(a)x) \prec_{\mathcal{A}}\alpha(y)= \cr
& \alpha(x)\prec_{\mathcal{A}} (l_{\mathcal{B}}(a)y) + r_{\prec_{\mathcal{B}}}(r_{\mathcal{A}}(y)a)\alpha(x),
\end{array} \\
\label{eq46}
l_{\prec_{\mathcal{B}}}(\beta(a))(x \ast_{\mathcal{A}} y) = (l_{\prec_{\mathcal{B}}}(a)x) \prec_{\mathcal{A}}\alpha(y) +
l_{\prec_{\mathcal{B}}}(r_{\prec_{\mathcal{A}}}(x)a)\alpha(y),
\\
\label{eq47}
r_{\prec_{\mathcal{B}}}(\beta(a))(x \succ_{\mathcal{A}} y) = r_{\succ_{\mathcal{B}}}(l_{\prec_{\mathcal{B}}}(y)a)\alpha(x) +
 \alpha(x)\succ_{\mathcal{A}} (r_{\prec_{\mathcal{B}}}(a)y),
\\
\label{eq48}
\begin{array}{ll}
l_{\prec_{\mathcal{B}}}(l_{\succ_{\mathcal{A}}}(x)a)\alpha(y) &+ (r_{\succ_{\mathcal{B}}}(a)x) \prec_{\mathcal{A}}\alpha(y)=\cr
 & \alpha(x)\succ_{\mathcal{A}} (l_{\prec_{\mathcal{B}}}(a)y) + r_{\succ_{\mathcal{B}}}(r_{\prec_{\mathcal{A}}}(y)a)\alpha(x),
\end{array}
\\
\label{eq49}
l_{\succ_{\mathcal{B}}}(\beta(a))(x \prec_{\mathcal{A}} y) = (l_{\succ_{\mathcal{B}}}(a)x) \prec_{\mathcal{A}}\alpha(y) +
l_{\prec_{\mathcal{B}}}(r_{\succ_{\mathcal{A}}}(x)a)\alpha(y),
\\
\label{eq50}
 r_{\succ_{\mathcal{B}}}(\beta(a))(x \ast_{\mathcal{A}} y)= \alpha(x)\succ_{\mathcal{A}} (r_{\succ_{\mathcal{B}}}(a)y) +
r_{\succ_{\mathcal{B}}}(l_{\succ_{\mathcal{A}}}(y)a)\alpha(x),
 \\
\label{eq51}
\begin{array}{ll}
 \alpha(x)\succ_{\mathcal{A}} (l_{\succ_{\mathcal{B}}}(a)y) &+ r_{\succ_{\mathcal{B}}}(r_{\succ_{\mathcal{A}}}(y)a)\alpha(x)=\cr
& l_{\succ_{\mathcal{B}}}(l_{\mathcal{A}}(x)a)\alpha(y) + (r_{B}(a)x) \succ_{\mathcal{A}} \alpha(y),
\end{array}
\\
\label{eq52}
l_{\succ_{\mathcal{B}}}(\beta(a))(x \succ_{\mathcal{A}} y) = (l_{\mathcal{B}}(a)x) \succ_{\mathcal{A}}\alpha(y) +
l_{\succ_{\mathcal{B}}}(r_{\mathcal{A}}(x)a)\alpha(y)
\end{eqnarray}
for any $ x, y \in \mathcal{A}, a, b \in \mathcal{B} $ and $ l_{\mathcal{A}} = l_{\succ_{\mathcal{A}}} +
 l_{\prec_{\mathcal{A}}}, r_{\mathcal{A}} =  r_{\succ_{\mathcal{A}}} +  r_{\prec_{\mathcal{A}}}, l_{\mathcal{B}} =
 l_{\succ_{\mathcal{B}}} +  l_{\prec_{\mathcal{B}}} , r_{\mathcal{B}} =  r_{\succ_{\mathcal{B}}} +  r_{\prec_{\mathcal{B}}} $.
 Then, there is a Hom-dendriform algebra structure on the direct sum $ \mathcal{A} \oplus \mathcal{B} $ of the underlying vector spaces of
 $ \mathcal{A} $ and $ \mathcal{B} $ given by
\begin{eqnarray*}
(x + a) \succ ( y + b )&=&(x \succ_{\mathcal{A}} y + r_{\succ_{\mathcal{B}}}(b)x + l_{\succ_{\mathcal{B}}}(a)y) + \cr &&
 (l_{\succ_{\mathcal{A}}}(x)b + r_{\succ_{\mathcal{A}}}(y)a + a \succ_{\mathcal{B}} b), \cr
(x + a) \prec ( y + b )&=&(x \prec_{\mathcal{A}} y + r_{\prec_{\mathcal{B}}}(b)x + l_{\prec_{\mathcal{B}}}(a)y) + \cr &&
(l_{\prec_{\mathcal{A}}}(x)b + r_{\prec_{\mathcal{A}}}(y)a + a \prec_{\mathcal{B}} b),\cr
(\alpha\oplus\beta)(x + a)&=&\alpha(x) + \beta(a)
\end{eqnarray*}
for any $ x, y \in \mathcal{A}, a, b \in \mathcal{B} $.
\end{theorem}
\begin{proof}
It  is obtained in a similar way as for Theorem \ref{theo. of matched pairs}.
\qed
\end{proof}
 Let $ \mathcal{A} \bowtie^{l_{\succ_{\mathcal{A}}}, r_{\succ_{\mathcal{A}}},
l_{\prec_{\mathcal{A}}}, r_{\prec_{\mathcal{A}}}, \beta}_{l_{\succ_{\mathcal{B}}}, r_{\succ_{\mathcal{B}}}, l_{\prec_{\mathcal{B}}},
r_{\prec_{\mathcal{B}}}, \alpha} \mathcal{B} $ or simply $ \mathcal{A} \bowtie \mathcal{B} $  denote this Hom-dendriform algebra.
\begin{definition}
Let $ (\mathcal{A}, \succ_{\mathcal{A}}, \prec_{\mathcal{A}}, \alpha) $ and $  (\mathcal{B}, \succ_{\mathcal{B}}, \prec_{\mathcal{B}}, \beta) $
be two Hom-dendriform algebras. Suppose there are linear maps
$ l_{\succ_{\mathcal{A}}}, r_{\succ_{\mathcal{A}}}, l_{\prec_{\mathcal{A}}}, r_{\prec_{\mathcal{A}}} : \mathcal{A} \rightarrow gl(\mathcal{B}) $
and $ l_{\succ_{\mathcal{B}}}, r_{\succ_{\mathcal{B}}}, l_{\prec_{\mathcal{B}}}, r_{\prec_{\mathcal{B}}} : \mathcal{B} \rightarrow gl(\mathcal{A}) $
 such that $(l_{\succ_{\mathcal{A}}}, r_{\succ_{\mathcal{A}}}, l_{\prec_{\mathcal{A}}}, r_{\prec_{\mathcal{A}}}, \beta)$ is a bimodule of $ \mathcal{A}, $
and $(l_{\succ_{\mathcal{B}}}, r_{\succ_{\mathcal{B}}}, l_{\prec_{\mathcal{B}}}, r_{\prec_{\mathcal{B}}}, \alpha)$ is a bimodule of $ \mathcal{B} $.
If \eqref{eq35} - \eqref{eq52} are satisfied, then $(\mathcal{A}, \mathcal{B}, l_{\succ_{\mathcal{A}}},
r_{\succ_{\mathcal{A}}}, l_{\prec_{\mathcal{A}}}, r_{\prec_{\mathcal{A}}}, \beta,   l_{\succ_{\mathcal{B}}}, r_{\succ_{\mathcal{B}}}, l_{\prec_{\mathcal{B}}},
 r_{\prec_{\mathcal{B}}}, \alpha)$ is called a \textbf{matched pair of Hom-dendriform algebras}.
\end{definition}

\begin{corollary}\label{match dendriform-associative algebras}
If $(\mathcal{A}, \mathcal{B}, l_{\succ_{\mathcal{A}}}, r_{\succ_{\mathcal{A}}}, l_{\prec_{\mathcal{A}}},
r_{\prec_{\mathcal{A}}}, \beta,
 l_{\succ_{\mathcal{B}}}, r_{\succ_{\mathcal{B}}}, l_{\prec_{\mathcal{B}}}, r_{\prec_{\mathcal{B}}},
 \alpha) $ is a matched pair of Hom-dendriform algebras, then
 \begin{eqnarray*}(\mathcal{A}, \mathcal{B}, l_{\succ_{\mathcal{A}}} + l_{\prec_{\mathcal{A}}},
r_{\succ_{\mathcal{A}}} + r_{\prec_{\mathcal{A}}},
l_{\succ_{\mathcal{B}}} + l_{\prec_{\mathcal{B}}},  r_{\succ_{\mathcal{B}}} + r_{\prec_{\mathcal{B}}}, \alpha + \beta)\end{eqnarray*} is a matched pair of the associated
Hom-associative algebras $ (\mathcal{A}, \ast_{\mathcal{A}}, \alpha) $ and  $(\mathcal{B}, \ast_{\mathcal{B}}, \beta)$.
\end{corollary}
\begin{proof}
In fact, the associated Hom-associative algebra $(\mathcal{A} \bowtie \mathcal{B}, \ast, \alpha + \beta)$ is exactly the Hom-associative algebra obtained
from the matched pair  of Hom-associative algebras $(\mathcal{A}, \mathcal{B}, l_{\mathcal{A}}, r_{\mathcal{A}}, \beta,  l_{\mathcal{B}}, r_{\mathcal{B}}, \alpha)$ with
\begin{eqnarray*}
(x + a)\ast (y + b)&=& x\ast_{\mathcal{A}} y + l_{\mathcal{B}}(a)y + r_{\mathcal{B}}(b)x + a \ast_{\mathcal{B}} b +  l_{\mathcal{A}}(x)b +
r_{\mathcal{A}}(y)a,\cr
(\alpha\oplus\beta)(x + a)&=&\alpha(x) + \beta(a)
\end{eqnarray*}
for all $ x, y \in \mathcal{A}, a, b \in \mathcal{B} $, where $ l_{\mathcal{A}} = l_{\succ_{\mathcal{A}}} + l_{\prec_{\mathcal{A}}}, r_{\mathcal{A}}
= r_{\succ_{\mathcal{A}}} + r_{\prec_{\mathcal{A}}}, l_{\mathcal{B}} = l_{\succ_{\mathcal{B}}} + l_{\prec_{\mathcal{B}}}, r_{\mathcal{B}}=
r_{\succ_{\mathcal{B}}} + r_{\prec_{\mathcal{B}}}  $.
\qed
\end{proof}
\subsection{Double constructions of involutive symplectic Hom-associative algebras}
\label{subsec:dblcnstrinvsymhomassal:dblcnstrinvsymhomassal}
In this sequel, we suppose that $\alpha$ is involutive.
\begin{proposition}
Let ($l_{\succ}, r_{\succ}, l_{\prec}, r_{\prec}, \beta, V$) be a bimodule of a  Hom-dendriform algebra $(\mathcal{A},
\succ, \prec, \alpha)$, and let $(\mathcal{A}, \ast, \alpha)$ be the associated involutive Hom-associative algebra.
\begin{enumerate}[label=\upshape{\arabic*)}]
\item Let $ l^{\ast}_{\succ},r^{\ast}_{\succ},l^{\ast}_{\prec}, r^{\ast}_{\prec} : \mathcal{A}
\rightarrow gl(V^{\ast}) $ be the linear maps given by
\begin{eqnarray*}
\langle l^{\ast}_{\succ}(x)a^{\ast}, y \rangle &=& \langle l_{\succ}(x)y, a^{\ast} \rangle,
\langle r^{\ast}_{\succ}(x)a^{\ast}, y \rangle = \langle r_{\succ}(x)y, a^{\ast} \rangle, \cr
\langle l^{\ast}_{\prec}(x)a^{\ast}, y \rangle &=& \langle l_{\prec}(x)y, a^{\ast} \rangle,
\langle r^{\ast}_{\prec}(x)a^{\ast}, y \rangle = \langle r_{\prec}(x)y, a^{\ast} \rangle.
\end{eqnarray*}
Then,  $(r^{\ast}_{\succ} +  r^{\ast}_{\prec}, -l^{\ast}_{\prec}, -r^{\ast}_{\succ}, l^{\ast}_{\succ} + l^{\ast}_{\prec}, \beta^{\ast}, V^{\ast})$ is a bimodule of $(\mathcal{A},\succ,\prec, \alpha);$
\item $(r^{\ast}_{\succ} +  r^{\ast}_{\prec}, 0, 0, l^{\ast}_{\succ} + l^{\ast}_{\prec}, \beta^{\ast}, V^{\ast})$
 and $(r^{\ast}_{\prec}, 0, 0, l^{\ast}_{\succ}, \beta^{\ast}, V^{\ast})$ are bimodules \\ 
 of $ (\mathcal{A}, \ast, \alpha); $
\item $(r^{\ast}_{\succ} +  r^{\ast}_{\prec}, l^{\ast}_{\succ} + l^{\ast}_{\prec}, \beta^{\ast}, V^{\ast})$ and
$(r^{\ast}_{\prec}, l^{\ast}_{\succ}, \beta^{\ast}, V^{\ast})$ are bimodules of $(\mathcal{A}, \ast, \alpha); $
\item The Hom-dendriform algebras  $ \mathcal{A} \times_{r^{\ast}_{\succ} + r^{\ast}_{\prec} ,
-l^{\ast}_{\prec}, -r^{\ast}_{\succ}, l^{\ast}_{\succ} + l^{\ast}_{\prec}, \alpha, \beta^{\ast}} V^{\ast} $ and \\
$ \mathcal{A} \times_{r^{\ast}_{\prec}, 0, 0, l^{\ast}_{\succ}, \alpha, \beta^{\ast}} V^{\ast} $ have the same Hom-associative algebra \\
$\mathcal{A} \times_{r^{\ast}_{\prec}, l^{\ast}_{\succ}, \alpha, \beta^{\ast}} V^{\ast} $.
\end{enumerate}
\end{proposition}
\begin{example}
Let $(\mathcal{A}, \prec, \succ, \alpha)$ be an involutive Hom-dendriform algebra. Then,
\begin{eqnarray*}
(L_{\succ}, R_{\succ}, L_{\prec}, R_{\prec}, \alpha, \mathcal{A}), (L_{\succ}, 0, 0, R_{\prec}, \alpha, \mathcal{A}),
 (L_{\succ} + L_{\prec}, 0, 0,  R_{\succ} + R_{\prec}, \alpha, \mathcal{A})
\end{eqnarray*}
are bimodules of $(\mathcal{A}, \prec, \succ, \alpha)$. On the other hand,
\begin{eqnarray*}
 (R^{\ast}_{\succ} + R^{\ast}_{\prec}, -L^{\ast}_{\prec}, -R^{\ast}_{\succ},  L^{\ast}_{\succ} +
 L^{\ast}_{\prec}, \alpha^{\ast}, \mathcal{A}^{\ast}),  (R^{\ast}_{\prec}, 0, 0,  L^{\ast}_{\succ}, \alpha^{\ast}, \mathcal{A}^{\ast}),\cr (R^{\ast}_{\succ}
+ R^{\ast}_{\prec}, 0, 0,  L^{\ast}_{\succ} + L^{\ast}_{\prec}, \alpha^{\ast}, \mathcal{A}^{\ast})
\end{eqnarray*}
are bimodules of $(\mathcal{A}, \succ, \prec, \alpha)$  too. There are two compatible Hom-dendriform algebra structures,
\begin{eqnarray*}
\mathcal{A} \times_{R^{\ast}_{\succ} + R^{\ast}_{\prec}, -L^{\ast}_{\prec}, -R^{\ast}_{\succ},
L^{\ast}_{\succ} + L^{\ast}_{\prec}, \alpha, \alpha^{\ast}} \mathcal{A}^{\ast} \mbox { and } \mathcal{A} \times_{R^{\ast}_{\succ} + R^{\ast}_{\prec}, 0, 0,  L^{\ast}_{\succ} + L^{\ast}_{\prec}, \alpha, \alpha^{\ast}}\mathcal{A}^{\ast},
\end{eqnarray*}
on the same Hom-associative algebra $\mathcal{A} \times_{ R^{\ast}_{\prec},  L^{\ast}_{\succ}, \alpha, \alpha^{\ast}}\mathcal{A}^{\ast} $.
\end{example}
\begin{definition}
Let $(\mathcal{A}, \alpha)$ be a Hom-associative algebra. We say that $(\mathcal{A}, \alpha, \omega)$ is a symplectic Hom-associative algebra if $\omega$ is a non-degenerate skew-symmetric bilinear form on $\mathcal{A}$ such that the following identity (invariance condition) is  satisfied for all $x, y, z\in\mathcal{A}$:
\begin{eqnarray}
\omega(\alpha(x)\alpha(y), \alpha(z)) + \omega(\alpha(y)\alpha(z), \alpha(x)) + \omega(\alpha(z)\alpha(x), \alpha(y))=0.
\end{eqnarray}
\end{definition}
\begin{theorem}\label{dendriform-symplectic theorem}
Let $(\mathcal{A}, \ast, \alpha)$ be an involutive Hom-associative algebra, and let $\omega $ be an $\alpha$-invariant non-degenerate skew-symmetric bilinear form on $\mathcal{A}$.
Then,  there exists a compatible Hom-dendriform algebra structure $ \succ, \prec $ on $( \mathcal{A}, \alpha)$ given by
\begin{eqnarray} \label{dendriform-symplectic equation}
\omega(x \succ y, z) = \omega(y, z \ast x), \ \ \ \omega(x \succ y, z) = \omega(x, y \ast z) \mbox { for all } x, y \in \mathcal{A}.
\end{eqnarray}
\end{theorem}
\begin{proof}
Define a linear map $ T : \mathcal{A} \rightarrow \mathcal{A}^{\ast} $ by $ \langle T(x), y \rangle = \omega(x, y) $ for all $ x, y \in \mathcal{A} $.
Then, $ T $ is invertible and $ T^{-1} $ is an $  \mathcal{O}-$operator of the involutive Hom-associative algebra $(\mathcal{A}, \ast, \alpha)$ associated to the
bimodule $(R^{\ast}_{\ast}, L^{\ast}_{\ast}, \alpha^{\ast})$. By Corollary \ref{dendriform-invertible operator}, there is a compatible Hom-dendriform algebra structure
$ \succ, \prec $ on $(\mathcal{A}, \ast)$ given by
\begin{eqnarray*}
x \succ y = T^{-1}R^{\ast}_{\ast}(x)T(y), \ \ x \prec y =  T^{-1}L^{\ast}_{\ast}(y)T(x)
\end{eqnarray*}
for all $ x, y \in \mathcal{A} $, which gives exactly the equation \eqref{dendriform-symplectic equation}.
\qed
\end{proof}
\begin{definition}
We call $(\mathcal{A}, \alpha, \mathcal{B})$ a \textbf{double construction of involutive symplectic Hom-associative algebra} associated to
$(\mathcal{A}_1, \alpha_{1})$ and $({\mathcal A}_1^*, \alpha^{\ast}_{1})$ if it satisfies the following conditions:
\begin{enumerate}[label=\upshape{\arabic*)}]
\item $ \mathcal{A} = \mathcal{A}_{1} \oplus \mathcal{A}^{\ast}_{1} $ as the direct sum of vector
spaces;
\item $(\mathcal{A}_1, \alpha_{1})$ and $({\mathcal A}_1^*, \alpha^{\ast}_{1})$ are Hom-associative subalgebras of $
(\mathcal{A}, \alpha)$ with $\alpha=\alpha_{1}\oplus\alpha^{\ast}_{1}$;
\item $ \omega $ is the natural non-degenerate antisymmetric ($\alpha_{1}\oplus\alpha^{\ast}_{1}$)-invariant
bilinear form on $ \mathcal{A}_{1} \oplus \mathcal{A}^{\ast}_{1} $
given by
\begin{eqnarray} \label{symplectic form}
 \omega(x + a^{\ast}, y +  b^{\ast})&=& -\langle x, b^{\ast} \rangle + \langle a^{\ast}, y \rangle, \cr
 \omega((\alpha + \alpha^{\ast})(x + a^{\ast}), y + b^{\ast})&=&  \omega(x + a^{\ast}, (\alpha + \alpha^{\ast})(y + b^{\ast}))
\end{eqnarray}
for all  $x, y \in \mathcal{A}_{1}, a^{\ast}, b^{\ast} \in \mathcal{A}^{\ast}_{1},$
where $ \langle  , \rangle $ is the natural pair between the vector space $ \mathcal{A}_{1} $ and its dual space  $ \mathcal{A}^{\ast}_{1} $.
\end{enumerate}
\end{definition}

 Let $(\mathcal{A}, \ast_{\mathcal{A}}, \alpha)$ be an involutive
 Hom-associative algebra, and suppose that there is an involutive Hom-associative algebra structure $ \ast_{\mathcal{A}^{\ast}} $ on its dual
space $ \mathcal{A}^{\ast} $. We construct an involutive symplectic Hom-associative algebra structure on the direct sum
 $ \mathcal{A} \oplus \mathcal{A}^{\ast} $ of the underlying vector spaces of $ \mathcal{A} $ and $ \mathcal{A}^{\ast} $ such
 that both $ \mathcal{A} $ and $ \mathcal{A}^{\ast} $ are Hom-subalgebras,  equipped with the natural non-degenerate antisymmetric ($\alpha_{1}\oplus\alpha^{\ast}_{1}$)-invariant bilinear form on $\mathcal{A} \oplus \mathcal{A}^{\ast}$ given by
 the equation \eqref{symplectic form}. Such a construction is called
double construction of involutive symplectic Hom-associative algebras associated to $(\mathcal{A}, \ast_{\mathcal{A}}, \alpha)$ and
$ (\mathcal{A}^{\ast}, \ast_{\mathcal{A}^{\ast}}, \alpha^{\ast}),$  denoted by $ (T(\mathcal{A}) = \mathcal{A}
\bowtie^{\alpha^{\ast}}_{\alpha} \mathcal{A}^{\ast}, \omega)$.

\begin{corollary}
Let $ (T(\mathcal{A}) = \mathcal{A}
\bowtie^{\alpha^{\ast}}_{\alpha} \mathcal{A}^{\ast}, \omega)$
 be a double construction of involutive symplectic Hom-associative algebras. Then,
 there exists a compatible involutive Hom-dendriform algebra structure $ \succ, \prec $ on $(T(\mathcal{A}), \alpha\oplus\alpha^{\ast})$
defined by the equation \eqref{symplectic form}.
  Moreover, $ \mathcal{A} $ and $ \mathcal{A}^{\ast} $, endowed with this product,
 are Hom-dendriform subalgebras.
\end{corollary}
\begin{proof}
The first  part follows from Theorem \ref{dendriform-symplectic theorem}.
Let $ x, y \in \mathcal{A} $. Set $ x \succ y = a + b^{\ast},$ where
$ a \in \mathcal{A}, b^{\ast} \in \mathcal{A}^{\ast} $. Since $ \mathcal{A} $ is a Hom-associative subalgebra of $(T(\mathcal{A}), \alpha\oplus\alpha^{\ast})$
and $ \omega(\mathcal{A}, \mathcal{A}) = \omega(\mathcal{A}^{\ast}, \mathcal{A}^{\ast}) = 0 $, we have
\begin{eqnarray*}
\omega(b^{\ast}, \mathcal{A}^{\ast}) = \omega(b^{\ast}, \mathcal{A}) = \omega(x \succ y, \mathcal{A}) = \omega(y, \mathcal{A} \ast x) = 0.
\end{eqnarray*}
Therefore,
 $ b^{\ast} = 0 $ due to the non-dependence of $ \omega $. Hence, $ x \succ y = a \in \mathcal{A}$. Similarly, $ x \prec y  \in \mathcal{A} $.
 Thus, $ \mathcal{A} $ is a Hom-dendriform subalgebra of $ T(\mathcal{A}) $ with the product $ \prec, \succ $.
 By symmetry of $ \mathcal{A} $, $ \mathcal{A}^{\ast} $ is also a Hom-dendriform subalgebra.
 \qed
\end{proof}
\begin{theorem}
Let $(\mathcal{A}, \succ_{\mathcal{A}}, \prec_{\mathcal{A}}, \alpha)$ be an involutive Hom-dendriform algebra, and $(\mathcal{A}, \ast_{\mathcal{A}}, \alpha)$ be
the associated involutive Hom-associative algebra. Suppose  there is an involutive Hom-dendriform algebra  structure
$"\succ_{\mathcal{A}^{\ast}}, \prec_{\mathcal{A}^{\ast}}, \alpha^{\ast}"$ on its dual space $\mathcal{A}^{\ast},$ and
 $ (\mathcal{A}^{\ast}, \ast_{\mathcal{A}^{\ast}}, \alpha^{\ast})$ is the associated involutive Hom-associative algebra. Then, there exists a double construction
of involutive symplectic Hom-associative algebras associated to $(\mathcal{A}, \ast_{\mathcal{A}}, \alpha)$ and $(\mathcal{A}, \ast_{\mathcal{A}^{\ast}}, \alpha^{\ast})$
if and only if the octuple
$(\mathcal{A}, \mathcal{A}^{\ast}, R^{\ast}_{\prec_{\mathcal{A}}}, L^{\ast}_{\succ_{\mathcal{A}}}, \alpha^{\ast}, R^{\ast}_{\prec_{\mathcal{A}^{\ast}}},
L^{\ast}_{\succ_{\mathcal{A}^{\ast}}}, \alpha)$ is a matched pair of Hom-associative algebras.
\end{theorem}
\begin{proof}
The conclusion can be obtained by a similar proof as in Theorem \ref{Frobenius theorem}.
Then, if  $(\mathcal{A}, \mathcal{A}^{\ast}, R^{\ast}_{\prec_{\mathcal{A}}}, L^{\ast}_{\succ_{\mathcal{A}}}, \alpha^{\ast},
R^{\ast}_{\prec_{\mathcal{A}^{\ast}}},L^{\ast}_{\succ_{\mathcal{A}^{\ast}}}, \alpha)$ is a matched pair of the involutive Hom-associative algebras $(\mathcal{A}, \ast_{\mathcal{A}}, \alpha)$ and $(\mathcal{A}, \ast_{\mathcal{A}^{\ast}}, \alpha^{\ast})$,
it is straightforward to show that the bilinear form \eqref{symplectic form} is an  ($\alpha_{1}\oplus\alpha^{\ast}_{1}$)-invariant on the Hom-associative algebra
$ \mathcal{A} \bowtie^{R^{\ast}_{\prec_{\mathcal{A}}}, L^{\ast}_{\succ_{\mathcal{A}}}, \alpha^{\ast}}_{R^{\ast}_{\prec_{\mathcal{A}^{\ast}}},
L^{\ast}_{\succ_{\mathcal{A}^{\ast}}}, \alpha} \mathcal{A}^{\ast} $ given by
\begin{eqnarray*}
(x + a^{\ast})\ast_{\mathcal{A} \oplus \mathcal{A}^{\ast}} (y + b^{\ast}) &=&(x \ast_{\mathcal{A}} y + R^{\ast}_{\prec_{\mathcal{A}^{\ast}}}(a^{\ast})y
+ L^{\ast}_{\succ_{\mathcal{A}^{\ast}}}(b^{\ast})x)\cr
&&+ (a^{\ast} \ast_{\mathcal{A}^{\ast}} b^{\ast} + R^{\ast}_{\prec_{\mathcal{A}}}(x)b^{\ast} + L^{\ast}_{\succ_{\mathcal{A}}}(y)a^{\ast}).
\end{eqnarray*}
In fact, we have
\begin{eqnarray*}
&&\omega[(\alpha(x) + \alpha^{\ast}(a^{\ast}))\ast_{\mathcal{A} \oplus \mathcal{A}^{\ast}} (\alpha(y) + \alpha^{\ast}(b^{\ast})), \alpha(z) + \alpha^{\ast}(c^{\ast})]\cr
&& \quad + \omega[(\alpha(y) + \alpha^{\ast}(b^{\ast}))\ast_{\mathcal{A} \oplus \mathcal{A}^{\ast}} (\alpha(z) + \alpha^{\ast}(c^{\ast})), \alpha(x) + \alpha^{\ast}(a^{\ast})] \cr
&& \quad +\omega[(\alpha(z) + \alpha^{\ast}(c^{\ast}))\ast_{\mathcal{A} \oplus \mathcal{A}^{\ast}} (\alpha(x) + \alpha^{\ast}(a^{\ast})), \alpha(y) + \alpha^{\ast}(b^{\ast})] \cr
&&= -\langle \alpha(x) \ast_{\mathcal{A}} \alpha(y) + R^{\ast}_{\prec_{\mathcal{A}^{\ast}}}(\alpha^{\ast}(a^{\ast}))\alpha(y) +
L^{\ast}_{\succ_{\mathcal{A}^{\ast}}}(\alpha^{\ast}(b^{\ast}))\alpha(x), \alpha^{\ast}(c^{\ast})\rangle \cr
&& \quad + \langle \alpha^{\ast}(a^{\ast})
\ast_{\mathcal{A}^{\ast}}\alpha^{\ast}(b^{\ast}) + R^{\ast}_{\prec_{\mathcal{A}}}(\alpha(x))\alpha^{\ast}(b^{\ast})
 + L^{\ast}_{\succ_{\mathcal{A}}}(\alpha(y))\alpha^{\ast}(a^{\ast}), \alpha(z)\rangle \cr
&& \quad-\langle \alpha(y)\ast_{\mathcal{A}} \alpha(z) + R^{\ast}_{\prec_{\mathcal{A}^{\ast}}}(\alpha^{\ast}(b^{\ast}))\alpha(z) +
L^{\ast}_{\succ_{\mathcal{A}^{\ast}}}(\alpha^{\ast}(c^{\ast}))\alpha(y), \alpha^{\ast}(a^{\ast})\rangle \cr
&& \quad + \langle \alpha^{\ast}(b^{\ast})
\ast_{\mathcal{A}^{\ast}}\alpha^{\ast}(c^{\ast}) + R^{\ast}_{\prec_{\mathcal{A}}}(\alpha^{\ast}(y))\alpha^{\ast}(c^{\ast})
+ L^{\ast}_{\succ_{\mathcal{A}}}(\alpha(z))\alpha^{\ast}(b^{\ast}), \alpha(x)\rangle \cr
&& \quad -\langle \alpha(z)\ast_{\mathcal{A}} \alpha(x) + R^{\ast}_{\prec_{\mathcal{A}^{\ast}}}(\alpha^{\ast}(c^{\ast}))\alpha(x) +
L^{\ast}_{\succ_{\mathcal{A}^{\ast}}}(\alpha^{\ast}(a^{\ast}))\alpha(z), \alpha^{\ast}(b^{\ast})\rangle \cr
&& \quad + \langle \alpha^{\ast}(c^{\ast})\ast_{\mathcal{A}^{\ast}}\alpha(a^{\ast})+ R^{\ast}_{\prec_{\mathcal{A}}}(\alpha(z))\alpha^{\ast}(a^{\ast}) +
L^{\ast}_{\prec_{\mathcal{A}}}(\alpha(x))\alpha^{\ast}(c^{\ast}),  \alpha(y)\rangle \cr
&& = -\langle \alpha(x)\prec_{\mathcal{A}}\alpha(y), \alpha^{\ast}(c^{\ast})\rangle  -\langle \alpha(x)\succ_{\mathcal{A}} \alpha(y),
 \alpha^{\ast}(c^{\ast})\rangle \cr
&&\quad -\langle \alpha^{\ast}(c^{\ast})\prec_{\mathcal{A}^{\ast}} \alpha^{\ast}(a^{\ast}), \alpha(y)\rangle\cr
&& \quad - \langle \alpha^{\ast}(b^{\ast})\succ_{\mathcal{A}^{\ast}}\alpha^{\ast}(c^{\ast}), \alpha(x)\rangle + \langle \alpha^{\ast}(a^{\ast})\prec_{\mathcal{A}^{\ast}}\alpha^{\ast}(b^{\ast}), \alpha(z)\rangle \cr
&& \quad + \langle \alpha^{\ast}(a^{\ast})\succ_{\mathcal{A}^{\ast}} \alpha^{\ast}(b^{\ast}), \alpha(z)\rangle + \langle \alpha(z)\prec_{\mathcal{A}} \alpha(x),
\alpha^{\ast}(b^{\ast})\rangle\cr
&& \quad + \langle \alpha(y)\succ_{\mathcal{A}}\alpha(z), \alpha^{\ast}(a^{\ast})\rangle\cr
&& \quad - \langle \alpha(y)\succ_{\mathcal{A}}\alpha(z), \alpha^{\ast}(a^{\ast})\rangle - \langle \alpha(y)\prec_{\mathcal{A}}\alpha(z), \alpha^{\ast}(a^{\ast})\rangle\cr
&&  \quad - \langle \alpha^{\ast}(a^{\ast})\prec_{\mathcal{A}^{\ast}}\alpha^{\ast}(b^{\ast}), \alpha(z)\rangle - \langle \alpha^{\ast}(a^{\ast})
\succ_{\mathcal{A}^{\ast}}\alpha^{\ast}(c^{\ast}), \alpha(y)\rangle\cr
&& \quad + \langle \alpha^{\ast}(b^{\ast}) \prec_{\mathcal{A}^{\ast}}\alpha^{\ast}(c^{\ast}), \alpha(x)\rangle\cr
&& \quad + \langle \alpha^{\ast}(b^{\ast}) \succ_{\mathcal{A}^{\ast}}\alpha^{\ast}(c^{\ast}), \alpha(x)\rangle + \langle \alpha(x)\prec_{\mathcal{A}} \alpha(y), \alpha(c^{\ast})\rangle \cr
&& \quad + \langle \alpha(z)\succ_{\mathcal{A}} \alpha(x), \alpha^{\ast}(b^{\ast})\rangle \cr
&&  \quad - \langle \alpha(z)\prec_{\mathcal{A}}\alpha(x), \alpha^{\ast}(b^{\ast})\rangle - \langle \alpha(z)\succ_{\mathcal{A}}\alpha(x), \alpha^{\ast}(b^{\ast})\rangle \cr
&& \quad -\langle \alpha^{\ast}(b^{\ast})\prec_{\mathcal{A}^{\ast}}\alpha^{\ast}(c^{\ast}),
  \alpha(x)\rangle -  \langle\alpha^{\ast}(a^{\ast})\succ_{\mathcal{A}^{\ast}}\alpha^{\ast}(b^{\ast}), \alpha(z)\rangle\cr
&& \quad +  \langle \alpha^{\ast}(c^{\ast})\prec_{\mathcal{A}^{\ast}} \alpha^{\ast}(a^{\ast}), \alpha(y)\rangle  \cr
&& \quad +  \langle \alpha^{\ast}(a^{\ast})\succ_{\mathcal{A}^{\ast}}\alpha^{\ast}(c^{\ast}), \alpha(y) \rangle  + \langle \alpha(y) \prec_{\mathcal{A}} \alpha(z), \alpha^{\ast}(a^{\ast} \rangle \cr
&& \quad + \langle \alpha(x) \succ_{\mathcal{A}} \alpha(y), \alpha^{\ast}(c^{\ast})\rangle = 0.
\end{eqnarray*}
Conversely, if there exists a double construction of involutive symplectic Hom-associative algebras associated to $(\mathcal{A}, \ast_{\mathcal{A}}, \alpha)$
and $(\mathcal{A}, \ast_{\mathcal{A}^{\ast}}, \alpha^{\ast})$, then, the octuple
$(\mathcal{A}, \mathcal{A}^{\ast}, R^{\ast}_{\prec_{\mathcal{A}}},
L^{\ast}_{\succ_{\mathcal{A}}}, \alpha^{\ast}, R^{\ast}_{\prec_{\mathcal{A}^{\ast}}}, L^{\ast}_{\succ_{\mathcal{A}^{\ast}}}, \alpha)$
is a matched pair of the involutive
 Hom-associative algebras  given by the following equations:
\begin{eqnarray} \label{eq1.28}
R^{\ast}_{\prec_{\mathcal{A}}}(\alpha(x))(a^{\ast} \ast_{\mathcal{A}^{\ast}} b^{\ast})
= R^{\ast}_{\prec_{\mathcal{A}}}(L_{\prec_{\mathcal{A}^{\ast}}}(a^{\ast})x)\alpha^{\ast}(b^{\ast}) +
(R^{\ast}_{\prec_{\mathcal{A}}}(x)a^{\ast}) \ast_{\mathcal{A}^{\ast}}\alpha^{\ast}(b^{\ast}),\nonumber
\\
\label{eq1.29}
L^{\ast}_{\succ_{\mathcal{A}}}(\alpha(x))(a^{\ast} \ast_{\mathcal{A}^{\ast}} b^{\ast}) =
L^{\ast}_{\succ_{\mathcal{A}}}(R_{\prec_{\mathcal{A}}}(b^{\ast})x)\alpha^{\ast}(a^{\ast}) +
\alpha^{\ast}(a^{\ast})\ast_{\mathcal{A}^{\ast}} (L^{\ast}_{\succ_{\mathcal{A}}}(x)b^{\ast}), \nonumber
\\
\label{eq1.30}
 R^{\ast}_{\prec_{\mathcal{A}^{\ast}}}(\alpha^{\ast}(a^{\ast}))(x \ast_{\mathcal{A}} y) =
R^{\ast}_{\prec_{\mathcal{A}^{\ast}}}(L_{\succ_{\mathcal{A}}}(x)a^{\ast})\alpha(y) +
 (R^{\ast}_{\prec_{\mathcal{A}^{\ast}}}(a^{\ast})x) \ast_{\mathcal{A}}\alpha(y), \nonumber
\\
\label{eq1.31}
 L^{\ast}_{\succ_{\mathcal{A}^{\ast}}}(\alpha^{\ast}(a^{\ast}))(x \ast_{\mathcal{A}} y) =
L^{\ast}_{\succ_{\mathcal{A}^{\ast}}}(R_{\prec_{\mathcal{A}}}(y)a^{\ast})\alpha(x) +
 \alpha(x)\ast_{\mathcal{A}} (L^{\ast}_{\succ_{\mathcal{A}^{\ast}}}(a^{\ast})y),  \nonumber
\\
\label{eq1.32}
\begin{array}{ll}
R^{\ast}_{\prec_{\mathcal{A}}}(R^{\ast}_{\prec_{\mathcal{A}^{\ast}}}(a^{\ast})x)\alpha^{\ast}(b^{\ast}) & +
(L^{\ast}_{\prec_{\mathcal{A}}}(x)a^{\ast})\ast_{\mathcal{A}^{\ast}}\alpha^{\ast}(b^{\ast}) - \cr
& L^{\ast}_{\succ_{A}} (L^{\ast}_{\succ_{\mathcal{A}^{\ast}}}(b^{\ast})x)\alpha^{\ast}(a^{\ast}) - \alpha^{\ast}(a^{\ast}) \ast_{\mathcal{A}^{\ast}}(R^{\ast}_{\prec_{\mathcal{A}}}(x)b^{\ast}) = 0, \nonumber
\end{array}
\\
\begin{array}{ll} \label{eq1.33}
 R^{\ast}_{\prec_{\mathcal{A}}}(R^{\ast}_{\prec_{\mathcal{A}^{\ast}}}(x)a^{\ast})\alpha(y) & +
(L^{\ast}_{\succ_{\mathcal{A}^{\ast}}}(a^{\ast})x)\ast_{A}\alpha(y) - \cr
& L^{\ast}_{\succ_{A^{\ast}}}  (L^{\ast}_{\succ_{A}}(y)a^{\ast})\alpha(x) - \alpha(x)\ast_{A}(R^{\ast}_{\prec_{A^{\ast}}}(a^{\ast})y) = 0, \nonumber
\end{array}
\end{eqnarray}
since the operation $ \ast_{A \oplus A^{\ast}} $ is Hom-associative.
\qed
\end{proof}
\begin{corollary}
Let $ (\mathcal{A}, \succ, \prec, \alpha) $ be an involutive Hom-dendriform algebra, and the triple                            $(R^{\ast}_{\prec}, L^{\ast}_{\succ}, \alpha^{\ast})$ be
the bimodule of the associated involutive Hom-associative algebra $(\mathcal{A}, \ast, \alpha)$. Then,
 $ (T(\mathcal{A}) = \mathcal{A} \times _{R^{\ast}_{\prec}, L^{\ast}_{\succ}, \alpha, \alpha^{\ast}} \mathcal{A}^{\ast}, \omega)$ is a double construction of the involutive symplectic Hom-associative algebras.
\end{corollary}
\begin{theorem}
Let $ (\mathcal{A}, \succ_{\mathcal{A}}, \prec_{\mathcal{A}}, \alpha)$ be an involutive  Hom-dendriform algebra,  and
$ (\mathcal{A}, \ast_{\mathcal{A}}, \alpha) $  be the associated involutive Hom-associative algebra. Suppose that there is an involutive Hom-dendriform algebra structure
 $\succ_{\mathcal{A}^{\ast}}, \prec_{\mathcal{A}^{\ast}}, \alpha^{\ast}$ on its dual space $ \mathcal{A}^{\ast},$ and
$ (\mathcal{A}^{\ast}, \ast_{\mathcal{A}^{\ast}}, \alpha^{\ast}) $ is its associated involutive Hom-associative algebra. Then,
 $(\mathcal{A}, \mathcal{A}^{\ast}, R^{\ast}_{\prec_{\mathcal{A}}},L^{\ast}_{\succ_{\mathcal{A}}}, \alpha^{\ast},
R^{\ast}_{\prec{\mathcal{A}^{\ast}}}, L^{\ast}_{\succ_{\mathcal{A}^{\ast}}}, \alpha)$
is a matched pair of involutive Hom-associative algebras if and only if
\begin{multline*}
(\mathcal{A}, \mathcal{A}^{\ast}, R^{\ast}_{\succ_{\mathcal{A}}} + R^{\ast}_{\prec_{\mathcal{A}}}, -L^{\ast}_{\prec_{\mathcal{A}}}, -R^{\ast}_{\succ_{\mathcal{A}}}, L^{\ast}_{\succ_{\mathcal{A}}} + L^{\ast}_{\prec_{\mathcal{A}}}, \alpha^{\ast}, R^{\ast}_{\succ_{\mathcal{A}^{\ast}}}
+ R^{\ast}_{\prec_{\mathcal{A}^{\ast}}},
 \cr
 - L^{\ast}_{\prec_{\mathcal{A}^{\ast}}}, -R^{\ast}_{\succ_{\mathcal{A}^{\ast}}}, L^{\ast}_{\succ_{\mathcal{A}^{\ast}}} + L^{\ast}_{\prec_{\mathcal{A}^{\ast}}}, \alpha)
\end{multline*}
is a matched pair of involutive  Hom-dendriform algebras.
\end{theorem}
\begin{proof}
The   necessary condition  follows from  Corollary \ref{match dendriform-associative algebras}. We need to prove the
 sufficient condition  only.
If $(\mathcal{A}, \mathcal{A}^{\ast}, R^{\ast}_{\prec_{\mathcal{A}}}, L^{\ast}_{\succ_{\mathcal{A}}}, \alpha^{\ast}, R^{\ast}_{\prec{\mathcal{A}^{\ast}}},
L^{\ast}_{\succ_{\mathcal{A}^{\ast}}}, \alpha)$ is a matched pair of involutive Hom-associative algebras, then $ ( \mathcal{A} \bowtie^{R^{\ast}_{\prec_{\mathcal{A}}},
L^{\ast}_{\succ_{\mathcal{A}}}, \alpha^{\ast}}_{R^{\ast}_{\prec_{\mathcal{A}^{\ast}}}, L^{\ast}_{\succ_{\mathcal{A}^{\ast}}}, \alpha} \mathcal{A}^{\ast} , \omega) $
is a double construction of  involutive symplectic Hom-associative algebras. Hence,  there exists a compatible involutive Hom-dendriform algebra structure on
$\mathcal{A} \bowtie^{R^{\ast}_{\prec_{\mathcal{A}}}, L^{\ast}_{\succ_{\mathcal{A}}}, \alpha^{\ast}}_{R^{\ast}_{\prec_{\mathcal{A}^{\ast}}},
L^{\ast}_{\succ_{\mathcal{A}^{\ast}}}, \alpha} \mathcal{A}^{\ast}$ given by \eqref{dendriform-symplectic equation}. By a simple and direct computation, we show that
$ \mathcal{A} $ and $ \mathcal{A}^{\ast} $ are its subalgebras, and the other products are given for any $ x \in \mathcal{A}, a^{\ast} \in \mathcal{A}^{\ast} $ by
\begin{eqnarray*}
x \succ a^{\ast} &=& ( R^{\ast}_{\succ_{\mathcal{A}}} + R^{\ast}_{\prec_{\mathcal{A}}})(x)a^{\ast} - L^{\ast}_{\prec_{\mathcal{A}^{\ast}}}x, \cr
x \prec a^{\ast} &=& - R^{\ast}_{\succ_{\mathcal{A}}} (x)a^{\ast} + (L^{\ast}_{\succ_{\mathcal{A}^{\ast}}} +
L^{\ast}_{\prec_{\mathcal{A}^{\ast}}})(a^{\ast})x, \cr
a^{\ast} \succ x &=& ( R^{\ast}_{\succ_{\mathcal{A}^{\ast}}} + R^{\ast}_{\prec_{\mathcal{A}^{\ast}}})(a^{\ast})x -
L^{\ast}_{\prec_{\mathcal{A}}}(x)a^{\ast}, \cr
a^{\ast} \prec x &=& -  R^{\ast}_{\succ_{\mathcal{A}^{\ast}}}(a^{\ast})x + (L^{\ast}_{\succ_{\mathcal{A}}} +
L^{\ast}_{\prec_{\mathcal{A}}})(x)a^{\ast}.
\end{eqnarray*}
Hence,
\begin{eqnarray*}
(\mathcal{A}, \mathcal{A}^{\ast}, R^{\ast}_{\succ_{\mathcal{A}}} + R^{\ast}_{\prec_{\mathcal{A}}}, -L^{\ast}_{\prec_{\mathcal{A}}}, -
R^{\ast}_{\succ_{\mathcal{A}}},  L^{\ast}_{\succ_{\mathcal{A}}} + L^{\ast}_{\prec_{\mathcal{A}}}, \alpha^{\ast}, R^{\ast}_{\succ_{\mathcal{A}^{\ast}}} + R^{\ast}_{\prec_{\mathcal{A}^{\ast}}},  \cr
 - L^{\ast}_{\prec_{\mathcal{A}^{\ast}}}, -R^{\ast}_{\succ_{\mathcal{A}^{\ast}}},L^{\ast}_{\succ_{\mathcal{A}^{\ast}}} + L^{\ast}_{\prec_{\mathcal{A}^{\ast}}}, \alpha)
\end{eqnarray*}
is a matched pair of involutive Hom-dendriform algebras.
\qed
\end{proof}

 \subsection{Hom-dendriform D-bialgebras}
 \label{subsec:dblcnstrinvsymhomassal:homdendDbialg}
\begin{theorem}
Let $ (\mathcal{A}, \succ_{\mathcal{A}}, \prec_{\mathcal{A}}, \alpha) $ be an involutive Hom-dendriform algebra whose products are given by two
 linear maps $ \beta^{\ast}_{\succ}, \beta^{\ast}_{\prec} : \mathcal{A} \otimes \mathcal{A} \rightarrow \mathcal{A} $.
  Further,  suppose that there is an involutive Hom-dendriform algebra structure $ \succ_{\mathcal{A}^{\ast}},
   \prec_{\mathcal{A}^{\ast}}, \alpha^{\ast} $ on its dual space $ \mathcal{A}^{\ast} $ given by two linear maps
$ \Delta^{\ast}_{\succ}, \Delta^{\ast}_{\prec} : \mathcal{A}^{\ast} \otimes \mathcal{A}^{\ast} \rightarrow
    \mathcal{A}^{\ast} $. Then,  $ (\mathcal{A}, \mathcal{A}^{\ast}, R^{\ast}_{\prec_{\mathcal{A}}},  L^{\ast}_{\succ_{\mathcal{A}}}, \alpha^{\ast},
      R^{\ast}_{\prec_{\mathcal{A}^{\ast}}},  L^{\ast}_{\succ_{\mathcal{A}^{\ast}}}, \alpha) $ is a matched pair of
      involutive Hom-associ\-ative algebras if and only if
\begin{eqnarray} \label{eq55}
\Delta_{\prec}\circ\alpha(x \ast_{\mathcal{A}} y) = (\alpha\otimes L_{\prec_{\mathcal{A}}}(x))\Delta_{\prec}(y) + (R_{\mathcal{A}}(y)\otimes\alpha)\Delta_{\prec}(y),
\\
\label{eq56}
\Delta_{\succ}\circ\alpha(x \ast_{\mathcal{A}} y) = (\alpha\otimes L_{\prec_{\mathcal{A}}}(x))\Delta_{\succ}(y) + (R_{\prec_{A}}(y)\otimes\alpha)\Delta_{\succ}(y),
\\
\label{eq57}
\beta_{\prec}\circ\alpha^{\ast}(a^{\ast} \ast_{\mathcal{A}^{\ast}} b^{\ast}) = (\alpha^{\ast}\otimes L_{\prec_{\mathcal{A}^{\ast}}}(a^{\ast}))\beta_{\prec}(b^{\ast}) + (R_{\mathcal{A}^{\ast}}(b^{\ast})\otimes\alpha^{\ast})\beta_{\prec}(a^{\ast})
\\
\label{eq58}
\beta_{\succ}\circ\alpha^{\ast}(a^{\ast}\ast_{\mathcal{A}^{\ast}} b^{\ast}) = (\alpha^{\ast}\otimes L_{\mathcal{A}^{\ast}}(a^{\ast}))\beta_{\succ}(b^{\ast}) + (R_{\prec_{\mathcal{A}^{\ast}}}(b^{\ast})\otimes\alpha^{\ast})\beta_{\succ}(a^{\ast}),
\\
\label{eq59}
\begin{array}{ll}
(L_{\mathcal{A}}(x)\otimes\alpha & - \alpha\otimes R_{\prec_{\mathcal{A}}}(x))\Delta_{\prec}(y)\cr
& + \sigma[(L_{\succ_{\mathcal{A}}}(y)\otimes (-\alpha)\otimes R_{\mathcal{A}}(y))\Delta_{\prec}(y)] = 0 ,
\end{array}
\\
\label{eq60}
\begin{array}{ll}
(L_{\mathcal{A}^{\ast}}(a^{\ast})\otimes\alpha^{\ast} & - \alpha^{\ast}\otimes R_{\prec_{\mathcal{A}^{\ast}}}(a^{\ast}))\beta_{\prec}(b^{\ast}) \cr
& + \sigma[(L_{\succ_{\mathcal{A}^{\ast}}}(b^{\ast})\otimes(-\alpha^{\ast})\otimes R_{\mathcal{A}^{\ast}}(b^{\ast}))\beta_{\succ}(a^{\ast})] = 0
\end{array}
\end{eqnarray}
hold for any $ x, y \in \mathcal{A} $ and
$ a^{\ast}, b^{\ast} \in \mathcal{A}^{\ast} $,
where
\begin{alignat*}{4}
L_{\mathcal{A}} &= L_{\succ_{\mathcal{A}}} + L_{\prec_{\mathcal{A}}},
\quad & R_{\mathcal{A}} &= R_{\succ_{\mathcal{A}}} + R_{\prec_{\mathcal{A}}}, \cr
L_{\mathcal{A}^{\ast}} &= L_{\succ_{\mathcal{A}^{\ast}}} + L_{\prec_{\mathcal{A}^{\ast}}}, & R_{\mathcal{A}^{\ast}} &= R_{\succ_{\mathcal{A}^{\ast}}} + R_{\prec_{\mathcal{A}^{\ast}}}.
\end{alignat*}
\end{theorem}
\begin{proof}
Let $ \lbrace e_{1},\dots, e_{n} \rbrace $ be a basis of $ \mathcal{A}, $ and $ \lbrace e^{\ast}_{1},\dots, e^{\ast}_{n} \rbrace $ be its dual basis. Set
\begin{alignat*}{6}
e_{i} \succ_{\mathcal{A}} e_{j} &= \sum^{n}_{k = 1} a^{k}_{ij} e_{k},
\quad \ e_{i} \prec_{\mathcal{A}} e_{j} &&= \sum^{n}_{k = 1} b^{k}_{ij} e_{k},
& \quad \alpha(e_{i}) &= \sum^{n}_{q=1}f^{i}_{q}e_{q}, \cr
e^{\ast}_{i} \succ_{\mathcal{A}^{\ast}} e^{\ast}_{j} &= \sum^{n}_{k = 1} c^{k}_{ij} e^{\ast}_{k},
\quad  e^{\ast}_{i} \prec_{\mathcal{A}^{\ast}} e^{\ast}_{j} &&= \sum^{n}_{k = 1} d^{k}_{ij} e^{\ast}_{k},
&\quad \alpha^{\ast}(e^{\ast}_{i}) &= \sum^{n}_{q=1}f^{\ast i}_{q}e^{\ast}_{q}.
\end{alignat*}
We have
$\langle \alpha^{\ast}(e^{\ast}_{i}), e_{j}\rangle  =  f^{\ast j}_{i}=\langle e^{\ast}, \alpha(e_{j})\rangle=f^{i}_{j} \quad \Rightarrow \quad f^{\ast j}_{i}= f^{i}_{j}$,
\begin{alignat*}{4}
\beta_{\succ}(e^{\ast}_{k}) &= \sum^{n}_{i, j= 1}a^{k}_{ij} e^{\ast}_{i}\otimes e^{\ast}_{j}, & \beta_{\prec}(e^{\ast}_{k})
&=\sum^{n}_{i, j= 1}b^{k}_{ij} e^{\ast}_{i}\otimes e^{\ast}_{j}, \cr
\Delta_{\succ}(e_{k}) &= \sum^{n}_{i, j= 1}c^{k}_{ij} e_{i}\otimes e_{j},
& \Delta_{\prec}(e_{k}) &=\sum^{n}_{i, j}d^{k}_{ij} e_{i}\otimes e_{j}, \cr
R^{\ast}_{\succ_{\mathcal{A}}}(e_{i})e^{\ast}_{j} &= \sum^{n}_{k = 1}a^{j}_{ki} e^{\ast}_{k},
& R^{\ast}_{\prec_{\mathcal{A}}}(e_{i})e^{\ast}_{j}
&= \sum^{n}_{k = 1}b^{j}_{ki} e^{\ast}_{k}, \cr
R^{\ast}_{\succ_{\mathcal{A}^{\ast}}}(e^{\ast}_{i})e_{j}
&= \sum^{n}_{k = 1}c^{j}_{ki} e_{k},
& R^{\ast}_{\prec_{A^{\ast}}}(e^{\ast}_{i})e_{j}
&= \sum^{n}_{k = 1}d^{j}_{ki} e_{k}, \cr
L^{\ast}_{\succ_{\mathcal{A}}}(e_{i})e^{\ast}_{j} &= \sum^{n}_{k = 1}a^{j}_{ik} e^{\ast}_{k},
& L^{\ast}_{\prec_{\mathcal{A}}}(e_{i})e^{\ast}_{j}
& = \sum^{n}_{k = 1}b^{j}_{ik} e^{\ast}_{k}, \cr
L^{\ast}_{\succ_{\mathcal{A}^{\ast}}}(e^{\ast}_{i})e_{j} &= \sum^{n}_{k = 1}c^{j}_{ik} e_{k},
& L^{\ast}_{\prec_{\mathcal{A}^{\ast}}}(e^{\ast}_{i})e_{j}
& = \sum^{n}_{k = 1}d^{j}_{ik} e_{k}.
\end{alignat*}
Therefore,  the coefficient of $ e^{\ast}_{l} $ in
\begin{eqnarray*}
R^{\ast}_{\prec_{\mathcal{A}}}(\alpha(e_{i}))(e^{\ast}_{j} \ast_{\mathcal{A}^{\ast}} e^{\ast}_{k}) =
R^{\ast}_{\prec_{\mathcal{A}}}(L_{\prec_{\mathcal{A}^{\ast}}}(e^{\ast}_{j})e_{i})\alpha^{\ast}(e^{\ast}_{k}) +
(R^{\ast}_{\prec_{\mathcal{A}}}(e_{i})e^{\ast}_{j}) \ast_{\mathcal{A}^{\ast}}\alpha^{\ast}(e^{\ast}_{k})
\end{eqnarray*}
gives the following relation for any $ i, j, l, k, q:$
\begin{eqnarray} \label{eq1.40}
\sum^{n}_{m=1}f^{i}_{q}b^{m}_{lq}(c^{m}_{jk} + d^{m}_{jk}) = \sum^{n}_{m=1}f^{k}_{q}[c^{i}_{jm}b^{q}_{lm} + b^{j}_{mi}(c^{l}_{mq} + d^{l}_{mq})].
\end{eqnarray}
In fact, we have
\begin{eqnarray*}
&&R^{\ast}_{\prec_{\mathcal{A}}}(\alpha(e_{i}))(e^{\ast}_{j} \ast_{\mathcal{A}^{\ast}} e^{\ast}_{k}) = R^{\ast}_{\prec_{A}}(\alpha(e_{i})) (e^{\ast}_{j} \succ_{\mathcal{A}^{\ast}} e^{\ast}_{k} + e^{\ast}_{j} \prec_{\mathcal{A}^{\ast}} e^{\ast}_{k}) \cr
&& \quad = R^{\ast}_{\prec_{A}}(\sum^{n}_{q=1}f^{i}_{q}e_{q}) \sum^{n}_{m = 1}(c^{m}_{jk}  + d^{m}_{jk})e^{\ast}_{m}= \sum^{n}_{m,q= 1}f^{i}_{q}(c^{m}_{jk}  + d^{m}_{jk})R^{\ast}_{\prec_{\mathcal{A}}}(e_{q})e^{\ast}_{m}\cr
&& \quad = \sum^{n}_{m, q= 1}f^{i}_{q}(c^{m}_{jk}  + d^{m}_{jk})(\sum^{n}_{l = 1} b^{m}_{lq})e^{\ast}_{l} =\sum^{n}_{l= 1}[\sum^{n}_{m, q= 1}f^{i}_{q}b^{m}_{lq}(c^{m}_{jk}  + d^{m}_{jk})]e^{\ast}_{l}, \cr
&& R^{\ast}_{\prec_{A}}(L_{\prec_{\mathcal{A}^{\ast}}}(e^{\ast}_{j})e_{i})\alpha^{\ast}(e^{\ast}_{k}) =
R^{\ast}_{\prec_{\mathcal{A}}}(\sum^{n}_{m = 1}c^{i}_{jm}e_{m})(\sum^{n}_{q=1}f^{k}_{q} e^{\ast}_{q})\cr
&& \quad = \sum^{n}_{m, q= 1}c^{i}_{jm}f^{k}_{q}R^{\ast}_{\prec_{\mathcal{A}}}(e_{m})e^{\ast}_{q}
=\sum^{n}_{m, q= 1}c^{i}_{jm}f^{k}_{q}(\sum^{n}_{l= 1}b^{q}_{lm} e^{\ast}_{l})
= \sum^{n}_{l = 1}(\sum^{n}_{m, q= 1}c^{i}_{jm}f^{k}_{q}b^{q}_{lm})e^{\ast}_{l}, \cr
&&(R^{\ast}_{\prec_{\mathcal{A}}}(e_{i})e^{\ast}_{j}) \ast_{\mathcal{A}^{\ast}}\alpha(e^{\ast}_{k})= (\sum^{n}_{m = 1}b^{j}_{mi} e^{\ast}_{m}) \ast_{\mathcal{A}^{\ast}}(\sum^{n}_{q=1}f^{k}_{q}e^{\ast}_{q})  \cr
&& \quad = \sum^{n}_{m, q= 1}f^{k}_{q}b^{j}_{mi}(e^{\ast}_{m}\succ_{A^{\ast}}e^{\ast}_{q} +   e^{\ast}_{m} \prec_{A^{\ast}} e^{\ast}_{q})
=\sum^{n}_{m, q= 1}f^{k}_{q}b^{j}_{mi}[\sum^{n}_{l = 1}(c^{l}_{mq} + d^{l}_{mq})]e^{\ast}_{l}\cr
&& \quad =  \sum^{n}_{l = 1} [\sum^{n}_{m, q= 1}f^{k}_{q}b^{j}_{mi}(c^{l}_{mq} + d^{l}_{mq})]e^{\ast}_{l},
\end{eqnarray*}
giving the equation \eqref{eq1.40}.
Also, the coefficient of $ e^{\ast}_{l} \otimes  e^{\ast}_{i}  $ in
\begin{eqnarray*}
&& \beta_{\prec}\circ\alpha^{\ast}(e^{\ast}_{j} \ast_{\mathcal{A}^{\ast}} e^{\ast}_{k}) = (\alpha^{\ast}\otimes L_{\prec_{\mathcal{A}^{\ast}}}(e^{\ast}_{j}))\beta_{\prec}(e^{\ast}_{k}) + (R_{\mathcal{A}^{\ast}}(e^{\ast}_{k})\otimes \alpha^{\ast})\beta_{\prec}(e^{\ast}_{j}), \cr
&&\beta_{\prec}\circ\alpha^{\ast}(e^{\ast}_{j} \ast_{\mathcal{A}^{\ast}} e^{\ast}_{k}) = \beta_{\prec}\circ\alpha^{\ast}(e^{\ast}_{j} \succ_{\mathcal{A}^{\ast}} e^{\ast}_{k}  + e^{\ast}_{j} \prec_{\mathcal{A}^{\ast}} e^{\ast}_{k}) \cr
&& \quad =  \beta_{\prec}\circ\alpha^{\ast}[\sum^{n}_{m = 1}(c^{m}_{jk} + d^{m}_{jk})e^{\ast}_{m}] = \sum^{n}_{m = 1}(c^{m}_{jk} + d^{m}_{jk}) \beta_{\prec}\circ\alpha^{\ast}(e^{\ast}_{m}) \cr
&& \quad =  \sum^{n}_{m = 1}(c^{m}_{jk} + d^{m}_{jk}) (\sum^{n}_{l, i, q= 1}f^{m}_{q}b^{q}_{li} e^{\ast}_{l}\otimes e^{\ast}_{i}) = \sum^{n}_{l, i, q= 1}
[\sum^{n}_{m = 1}f^{m}_{q}b^{q}_{li}(c^{m}_{jk} + d^{m}_{jk})]e^{\ast}_{l}\otimes e^{\ast}_{i}, \cr
&&(\alpha^{\ast}\otimes L_{\prec_{\mathcal{A}^{\ast}}}(e^{\ast}_{j}))\beta_{\prec}(e^{\ast}_{k}) = (\alpha^{\ast}\otimes L_{\prec_{\mathcal{A}^{\ast}}}(e^{\ast}_{j}))(\sum^{n}_{l, m = 1} b^{k}_{lm} e^{\ast}_{l}\otimes e^{\ast}_{m}) \cr
&& \quad = \sum^{n}_{l, m= 1} b^{k}_{lm}\alpha^{\ast}(e^{\ast}_{l})\otimes(e^{\ast}_{j} \prec_{\mathcal{A}^{\ast}} e^{\ast}_{m}) =
\sum^{n}_{l, m, q= 1} b^{k}_{lm}f^{l}_{q}e^{\ast}_{q}\otimes(\sum^{n}_{i = 1} c^{i}_{jm}e^{\ast}_{i}) \cr
&& \quad = \sum^{n}_{l, i, q= 1} (\sum^{n}_{m = 1} f^{l}_{q}b^{k}_{lm}c^{i}_{jm})e^{\ast}_{q}\otimes e^{\ast}_{i}, \cr
&& (R_{\mathcal{A}^{\ast}}(e^{\ast}_{k})\otimes \alpha^{\ast})\beta_{\prec}(e^{\ast}_{j}) = (R_{\mathcal{A}^{\ast}}(e^{\ast}_{k})\otimes \alpha^{\ast})(\sum^{n}_{m, i = 1}b^{j}_{mi}e^{\ast}_{m}\otimes e^{\ast}_{i})\cr
&& \quad = \sum^{n}_{m, i, q= 1}f^{i}_{q}b^{j}_{mi}(e^{\ast}_{m} \ast_{\mathcal{A}^{\ast}} e^{\ast}_{k})\otimes e^{\ast}_{q} = \sum^{n}_{m, i, q = 1}f^{i}_{q}b^{j}_{mi}[(e^{\ast}_{m} \succ_{\mathcal{A}^{\ast}} e^{\ast}_{k} + e^{\ast}_{m} \prec_{\mathcal{A}^{\ast}} e^{\ast}_{k})\otimes e^{\ast}_{q}] \cr
&& \quad = \sum^{n}_{m, i, q= 1}f^{i}_{q} b^{j}_{mi}[(e^{\ast}_{m} \succ_{\mathcal{A}^{\ast}} e^{\ast}_{k})\otimes e^{\ast}_{q} +
( e^{\ast}_{m}\prec_{\mathcal{A}^{\ast}} e^{\ast}_{k})\otimes e^{\ast}_{q}] \cr
&& \quad = \sum^{n}_{m, i, q= 1}f^{i}_{q}b^{j}_{mi}[\sum^{n}_{l = 1}c^{l}_{mk}e^{\ast}_{l}\otimes e^{\ast}_{q} +  \sum^{n}_{l1}d^{l}_{mk}e^{\ast}_{l}\otimes e^{\ast}_{q}]\cr
&& \quad = \sum^{n}_{l, i, q= 1} [\sum^{n}_{m = 1}f^{i}_{q}b^{j}_{mi}(c^{l}_{mk} + d^{l}_{mk})]e^{\ast}_{l}\otimes e^{\ast}_{q}
\end{eqnarray*}
gives the relation
\begin{eqnarray}\label{1.40'}
 \sum^{n}_{m = 1}[f^{m}_{q}b^{q}_{li}(c^{m}_{jk} + d^{m}_{jk}) + f^{l}_{q}b^{k}_{lm}c^{i}_{jm}]= \sum^{n}_{m = 1}[f^{i}_{q}b^{j}_{mi}(c^{l}_{mk} + d^{l}_{mk})].
\end{eqnarray}
Thus, \eqref{eq1.40} corresponds to  \eqref{1.40'}. Therefore, \eqref{eq1.28} $ \Leftrightarrow $ \eqref{eq57}.

 So, in the case $ l_{\mathcal{A}} = R^{\ast}_{\prec_{\mathcal{A}}},  r_{\mathcal{A}} =
  L^{\ast}_{\succ_{\mathcal{A}}}, l_{\mathcal{B}} = l_{\mathcal{A}^{\ast}} =
   R^{\ast}_{\prec_{\mathcal{A}^{\ast}}}, r_{\mathcal{B}} = r_{\mathcal{A}^{\ast}} =
   L^{\ast}_{\succ_{\mathcal{A}^{\ast}}}$, we have
   \eqref{match. pair1} $ \Leftrightarrow $ \eqref{eq1.28} $ \Leftrightarrow $ \eqref{eq57}.
Similarly, in this situation,
\begin{eqnarray*}
&& \eqref{match. pair2} \Leftrightarrow \eqref{eq1.29} \Leftrightarrow \eqref{eq58},  \eqref{match. pair3} \Leftrightarrow \eqref{eq1.30} \Leftrightarrow \eqref{eq55}, \cr
&&\eqref{match. pair4} \Leftrightarrow \eqref{eq1.31} \Leftrightarrow \eqref{eq56}, \cr
&& \eqref{match. pair5} \Leftrightarrow \eqref{eq1.32} \Leftrightarrow \eqref{eq60}, \ \ \ \
\eqref{match. pair6} \Leftrightarrow \eqref{eq1.33} \Leftrightarrow \eqref{eq59}.
\end{eqnarray*}
Therefore, the conclusion holds due to Theorem \ref{theo. of matched pairs}.
\qed
\end{proof}
\begin{definition}
Let $ \mathcal{A} $ be a vector space. A \textbf{Hom-dendriform D-bialgebra} structure on $ \mathcal{A} $ is a set of linear maps
$ (\Delta_{\prec}, \Delta_{\succ}, \alpha, \beta_{\prec}, \beta_{\succ}, \alpha^{\ast}) $, $ \Delta_{\prec}, \Delta_{\succ} : \mathcal{A} \rightarrow \mathcal{A} \otimes \mathcal{A} $, \ \
$ \beta_{\prec},  \beta_{\succ} : \mathcal{A}^{\ast} \rightarrow \mathcal{A}^{\ast} \otimes \mathcal{A}^{\ast}, \alpha: \mathcal{A}\rightarrow\mathcal{A}, \alpha^{\ast}: \mathcal{A}^{\ast}\rightarrow\mathcal{A}^{\ast}$, such that
\begin{enumerate}[label=\upshape{\alph*)}]
\item $ (\Delta^{\ast}_{\prec}, \Delta^{\ast}_{\succ}, \alpha^{\ast}) : \mathcal{A}^{\ast} \otimes \mathcal{A}^{\ast} \rightarrow \mathcal{A}^{\ast} $
 defines a Hom-dendriform algebra structure $(\succ_{\mathcal{A}^{\ast}}, \prec_{\mathcal{A}^{\ast}}, \alpha^{\ast}) $ on $ \mathcal{A}^{\ast} $;
\item $ (\beta^{\ast}_{\prec}, \beta^{\ast}_{\succ}, \alpha) : \mathcal{A} \otimes \mathcal{A} \rightarrow \mathcal{A} $ defines a Hom-dendriform algebra
 structure $(\succ_{\mathcal{A}}, \prec_{\mathcal{A}}, \alpha)$ on $  \mathcal{A} $;
\item the equations \eqref{eq55} -  \eqref{eq60}    are satisfied.
\end{enumerate}
We  denote it by $ (\mathcal{A}, \mathcal{A}^{\ast}, \Delta_{\succ}, \Delta_{\prec}, \alpha, \beta_{\succ}, \beta_{\prec}, \alpha^{\ast})$
 or simply $(\mathcal{A}, \mathcal{A}^{\ast}, \alpha, \alpha^{\ast})$.
 \end{definition}
\begin{theorem}
Let $(\mathcal{A}, \prec_{\mathcal{A}}, \succ_{\mathcal{A}}, \alpha)$ and $(\mathcal{A}^{\ast}, \prec_{\mathcal{A}^{\ast}}, \succ_{\mathcal{A}^{\ast}}, \alpha^{\ast})$
be involutive Hom-dendriform algebras. Let $(\mathcal{A}, \ast_{\mathcal{A}}, \alpha)$ and   $(\mathcal{A}^{\ast}, \ast_{\mathcal{A}^{\ast}}, \alpha^{\ast})$ be the
  corresponding associated involutive Hom-associative algebras. Then, the following conditions are equivalent:
\begin{enumerate}[label=\upshape{(\roman*)}]
\item There is a double construction of involutive symplectic Hom-associative algebras associated to $(\mathcal{A}, \ast_{\mathcal{A}}, \alpha)$ and
$(\mathcal{A}^{\ast}, \ast_{\mathcal{A}^{\ast}}, \alpha^{\ast})$;
\item $ (\mathcal{A}, \mathcal{A}^{\ast}, R^{\ast}_{\prec_{\mathcal{A}}},  L^{\ast}_{\succ_{\mathcal{A}}}, \alpha^{\ast},  R^{\ast}_{\prec_{\mathcal{A}^{\ast}}},
 L^{\ast}_{\succ_{\mathcal{A}^{\ast}}}, \alpha) $ is a matched pair of involutive Hom-associative algebras;
 \item $(\mathcal{A}, \mathcal{A}^{\ast}, R^{\ast}_{\succ_{\mathcal{A}}} + R^{\ast}_{\prec_{\mathcal{A}}},
- L^{\ast}_{\prec_{\mathcal{A}}}, -R^{\ast}_{\succ_{\mathcal{A}}}, L^{\ast}_{\succ_{\mathcal{A}}} +
 L^{\ast}_{\prec_{\mathcal{A}}}, \alpha^{\ast}, R^{\ast}_{\succ_{\mathcal{A}^{\ast}}} + R^{\ast}_{\prec_{\mathcal{A}^{\ast}}},
-L^{\ast}_{\prec_{\mathcal{A}^{\ast}}},\\ -R^{\ast}_{\succ_{\mathcal{A}^{\ast}}}, L^{\ast}_{\succ_{\mathcal{A}^{\ast}}}
 + L^{\ast}_{\prec_{\mathcal{A}^{\ast}}}, \alpha)$  is a matched pair of involutive Hom-dendriform algebras;
\item $ (\mathcal{A}, \mathcal{A}^{\ast}, \alpha, \alpha^{\ast}) $ is an involutive Hom-dendriform $ D- $bialgebra.
\end{enumerate}
\end{theorem}
\section{Matched pairs of biHom-associative algebras}
\label{sec:matchedpbihomassal}
\subsection{Bihom-dendriform algebras}
\label{subsec:matchedpbihomassal:bihomdendal}
\begin{definition}
A biHom-dendriform algebra is a quintuple $(\mathcal{A}, \prec, \succ, \alpha, \beta)$ consisting of a vector space $\mathcal{A}$ on which the operations $\prec, \succ: \mathcal{A}\otimes \mathcal{A}\rightarrow \mathcal{A}$ and $\alpha, \beta: \mathcal{A}\rightarrow \mathcal{A}$ are linear maps satisfying
\begin{eqnarray*}
&&\alpha\circ\beta=\beta\circ\alpha,\cr
&&\alpha(x\prec y)=\alpha(x)\prec\alpha(y), \alpha(x\succ y)=\alpha(x)\succ\alpha(y),\cr
&&\beta(x\prec y)=\beta(x)\prec\beta(y), \beta(x\succ y)=\beta(x)\succ\beta(y),\cr
&&(x \prec y)\prec \beta(z) = \alpha(x)\prec (y\ast z), \cr
&&(x\succ y)\prec\beta(z)=\alpha(x)\succ(y \prec z), \cr
&&\alpha(x)\succ (y \succ z ) = (x\ast y)\succ\beta(z),
\end{eqnarray*}
where $ x \ast y = x \prec y + x \succ y $.
\end{definition}
\begin{definition}
Let $(\mathcal{A}, \prec, \succ, \alpha, \beta)$ and $(\mathcal{A}', \prec',  \succ', \alpha', \beta')$ be biHom-dendriform algebras. A linear map $f: \mathcal{A}\rightarrow \mathcal{A}'$ is a biHom-dendriform algebra morphism if
\begin{eqnarray*}
\prec'\circ(f\otimes f)= f\circ\prec,\ \succ'\circ(f\otimes f)= f\circ\succ, f\circ\alpha= \alpha'\circ f \mbox{ and } f\circ\beta= \beta'\circ f.
\end{eqnarray*}
\end{definition}
\begin{proposition}
Let $(\mathcal{A}, \prec, \succ, \alpha, \beta)$ be a biHom-dendriform algebra.

Then, $(\mathcal{A}, \ast, \alpha, \beta)$ is a biHom-associative algebra.
\end{proposition}
\begin{proof}
We have, for all $x, y, z\in\mathcal{A}$,
\begin{eqnarray*}
(x\ast y)\ast\beta(z)&=&(x\prec y)\prec\beta(z) + (x\prec y)\succ\beta(z) + (x\succ y)\prec\beta(z) + (x\succ y)\succ\beta(z)\cr
&=&(x\prec y)\prec\beta(z) + (x\succ y)\prec\beta(z) + (x\prec y)\succ\beta(z) + (x\succ y)\succ\beta(z)\cr
&=& (x\prec y)\prec\beta(z) + (x\succ y)\prec\beta(z) + (x\ast y)\succ\beta(z)\cr
&=& \alpha(x)\prec(y\ast z) + \alpha(x)\succ(y\prec z) + \alpha(x)\succ(y\succ z)\cr
&=& \alpha(x)\prec(y\ast z) + \alpha(x)\succ(y\ast z)= \alpha(x)\ast(y\ast z), \cr
\alpha(x\ast y)&=&\alpha(x\succ y) + \alpha(x\prec y)\cr
&=& \alpha(x)\succ\alpha(y) + \alpha(x)\prec\alpha(y)\cr
&=& \alpha(x)\ast\alpha(y),
\end{eqnarray*}
which completes the proof. \qed
\end{proof}
We call $ (\mathcal{A}, \ast, \alpha, \beta)$ the biHom-associative algebra of $(\mathcal{A}, \prec, \succ, \alpha, \beta),$ and the quintuple  $(\mathcal{A}, \succ, \prec, \alpha, \beta)$ is called a compatible biHom-dendriform algebra structure on the biHom-associative algebra $(\mathcal{A}, \ast, \alpha, \beta)$.
\begin{proposition}
Let $(\mathcal{A}, \prec, \succ, \alpha, \beta)$ be a biHom-dendriform algebra. Suppose that $(\mathcal{A}, \ast, \beta, \alpha)$ is a biHom-associative algebra.  Then, $ (L_{\succ}, R_{\prec}, \beta, \alpha, \mathcal{A})$ is a bimodule of $(\mathcal{A}, \ast, \beta, \alpha)$.
\end{proposition}
\begin{proof}
For $x, y, v\in\mathcal{A}$, we have
\begin{eqnarray*}
L_{\succ}(x\ast y)\beta(v) &=& (x\ast y)\succ\beta(v)=\alpha(x)\succ(y\succ v)=L_{\succ}(\alpha(x))L_{\succ}(y)v, \cr
R_{\prec}(x\ast y)\alpha(v) &=& \alpha(v)\prec(x\ast y)= (v\prec x)\prec\beta(y)=R_{\prec}(\beta(y))R_{\prec}(x)v, \cr
L_{\succ}(\alpha(x))R_{\prec}(y)v &=& \alpha(x)\succ(v\prec y)=(x\succ v)\prec\beta(y)= R_{\prec}(\beta(y))L_{\succ}(x)v,
\end{eqnarray*}
which completes the proof. \qed
\end{proof}
\begin{remark}
If $(\mathcal{A}, \prec, \succ, \alpha, \beta)$ is a biHom-dendriform algebra, $(L_{\succ}, R_{\prec}, \alpha, \beta, \mathcal{A})$ is not a bimodule of associated  biHom-associative algebra $(\mathcal{A}, \ast, \alpha, \beta)$ .
\end{remark}
\begin{proposition}
Let $(\mathcal{A}, \prec, \succ, \alpha, \beta)$ be a biHom-dendriform algebra, and suppose
$\alpha^{2}=\beta^{2}=\alpha\circ\beta=\beta\circ\alpha=\id.$ Then, $(\mathcal{A}, \prec,
\succ, \alpha, \beta)\cong(\mathcal{A}, \prec, \succ, \beta, \alpha).$
\end{proposition}
\begin{proof}
Let $x, y, z\in\mathcal{A}.$ We have
\begin{eqnarray*}
\alpha(x)(yz)&=&(xy)\beta(z) \Leftrightarrow\cr
\alpha((\alpha\circ\beta)(x))(yz)&=&(xy)\beta((\beta\circ\alpha)(z))\Leftrightarrow\cr
\alpha^{2}(\beta(x))(yz)&=&(xy)\beta^{2}(\alpha(z))\Leftrightarrow\cr
\beta(x)(yz)&=&(xy)\alpha(z).
\end{eqnarray*}
Then $(\mathcal{A}, \prec,
\succ, \alpha, \beta)\cong(\mathcal{A}, \prec, \succ, \beta, \alpha).$
\qed
\end{proof}
\subsection{$ \mathcal{O} $-operators and biHom-dendriform algebras}
\label{subsec:matchedpbihomassal:oopsbihomdendal}
\begin{definition}
Let $(\mathcal{A}, \cdot, \alpha_{1}, \alpha_{2})$ be a biHom-associative algebra, and $(l, r, \beta_{1}, \beta_{2}, V)$ be a bimodule. A linear map $T : V \rightarrow \mathcal{A}$
is called an \textbf{$ \mathcal{O} $-operator associated} to $(l, r, \beta_{1}, \beta_{2}, V)$,  if $ T $ satisfies
\begin{eqnarray*}
\alpha_{1} T= T\beta_{2}, \alpha_{2} T= T\beta_{1} \mbox{ and } T(u)\cdot T(v) = T(l(T(u))v + r(T(v))u) \mbox { for all } u, v \in V.
\end{eqnarray*}
\end{definition}

\begin{example}
Let $(\mathcal{A}, \cdot, \alpha_{1}, \alpha_{2})$ be a multiplicative biHom-associative algebra. Then,
 the identity map $\id$ is an $ \mathcal{O} $-operator associated to the bimodule $(L,0, \alpha_{1}, \alpha_{2})$ or $(0,R, \alpha_{1}, \alpha_{2})$.
\end{example}

\begin{example}
Let $(\mathcal{A}, \cdot, \alpha, \beta)$ be a multiplicative biHom-associative algebra.
A linear map $ f : \mathcal{A} \rightarrow \mathcal{A} $ is called a \textbf{Rota-Baxter operator} on $ \mathcal{A} $ of weight zero if $ f $
satisfies
\begin{eqnarray*}
f\circ\alpha=\alpha\circ f, f\circ\beta=\beta\circ f \mbox{ and } f(x)\cdot f(y) = f(f(x)\cdot y + x\cdot f(y)) \mbox { for all } x, y \in \mathcal{A}.
\end{eqnarray*}
A Rota-Baxter operator on $ \mathcal{A} $ is just an $ \mathcal{O} $-operator associated to the bimodule $(L, R, \alpha, \beta)$.
\end{example}
\begin{theorem}
Let $(\mathcal{A}, \cdot, \alpha_{1}, \alpha_{2})$ be a biHom-associative algebra, and $(l, r, \beta_{1}, \beta_{2}, V) $ be a bimodule.
Let $ T : V \rightarrow \mathcal{A} $ be an $ \mathcal{O} $-operator associated to $(l, r, \beta_{1}, \beta_{2}, V)$. Then, there exists a biHom-dendriform
algebra structure on $V$ given by
\begin{eqnarray*}
 u\succ v = l(T(u))v , \quad u\prec v = r(T(v))u
\end{eqnarray*}
for all $u, v \in V$. So, there is an associated biHom-associative algebra structure on $ V $ given by the equation \eqref{associative-dendriform},  and $ T $
is a homomorphism of biHom-associative algebras. Moreover, $ T(V) = \lbrace { T(v) \setminus v \in V }  \rbrace  \subset \mathcal{A} $ is a biHom-associative
subalgebra of $ \mathcal{A}, $ and there is an induced biHom-dendriform algebra structure on $ T(V) $ given by
\begin{eqnarray}
T(u) \succ T(v) = T(u \succ v),\quad T(u) \prec T(v) = T(u \prec v)
\end{eqnarray}
for all $ u, v \in V $. Its corresponding associated biHom-associative algebra structure on $ T(V) $ given by the equation \eqref{associative-dendriform} is
just the biHom-associative subalgebra structure of $ \mathcal{A}, $ and $ T $ is a homomorphism of biHom-dendriform algebras.
\end{theorem}
\begin{proof}
For any $x, y, z\in V,$ we have
\begin{eqnarray*}
(x\succ y)\prec\beta_{2}(z) &-& \beta_{1}(x)\succ(y\prec z)=l(T(x)y)\prec\beta_{2}(z)-\beta_{1}(x)\succ r(T(z)y)\cr
&=& r(T\beta_{2}(z))l(T(x))y-l(T\beta_{1}(x)y)r(T(z)y) \cr
&=& r(\alpha_{1}(T(z)))l(T(x))y-l(\alpha_{2}(T(x)))r(T(z))y=0.
\end{eqnarray*}
The two other axioms are checked in a similar way.
\qed
\end{proof}
\begin{corollary}\label{bidendriform-invertible operator}
Let $(\mathcal{A}, \ast, \alpha, \beta)$ be a multiplicative biHom-associative algebra. There is a compatible multiplicative biHom-dendriform algebra structure on $ \mathcal{A} $
if and only if there exists an invertible $ \mathcal{O} $-operator of $ (\mathcal{A}, \ast, \alpha, \beta)$.
\end{corollary}
\begin{proof}
In fact, if the homomorphism  $ T $ is an invertible $ \mathcal{O}- $operator associated to a bimodule $(l, r, \alpha, \beta, V)$, then
the compatible multiplicative biHom-dendriform algebra structure on $ \mathcal{A} $ is given by
\begin{eqnarray*}
x \succ y = T(l(x)T^{-1}(y)), x \prec y = T(r(y)T^{-1}(x)) \mbox { for all } x, y \in \mathcal{A}.
\end{eqnarray*}
Conversely, let $(\mathcal{A}, \succ, \prec, \alpha, \beta)$ be a multiplicative biHom-dendriform algebra, and $(\mathcal{A}, \ast, \alpha, \beta)$
 be its associated biHom-associative algebra. Then, the identity map $\id$ is an $ \mathcal{O}- $operator
associated to the bimodule $(L_{\succ}, R_{\prec}, \alpha, \beta)$ of $(\mathcal{A}, \ast, \alpha, \beta)$.
\qed
\end{proof}

\subsection{Bimodules and matched pairs of biHom-dendriform algebras}
\label{subsec:matchedpbihomassal:bimodmatchedpbihomdendal}
\begin{definition}
Let $(\mathcal{A}, \succ, \prec, \alpha_{1}, \alpha_{2})$ be a biHom-dendriform algebra,  and $V$  be a vector space.
Let $l_{\succ}, r_{\succ}, l_{\prec}, r_{\prec} : \mathcal{A} \rightarrow gl(V),$ and $\beta_{1}, \beta_{2}: V \rightarrow V$ be six linear maps. Then, ( $ l_{\succ}, r_{\succ}, l_{\prec}, r_{\prec}, \beta_{1}, \beta_{2}, V$) is called a \textbf{bimodule} of $ \mathcal{A} $
 if the following equations hold for any $ x, y \in \mathcal{A} $ and $v\in V$:
\begin{eqnarray*}
l_{\prec}(x \prec y)\beta_{2}(v)&=& l_{\prec}(\alpha_{1}(x))l_{\ast}(y)v, r_{\prec}(\alpha_{2}(x))l_{\prec}(y)v=l_{\prec}(\alpha_{1}(y))r_{\ast}(x)v,\cr
r_{\prec}(\alpha_{2}(y))r_{\prec}(y)v &=& r_{\prec}(x\ast y)\beta_{1}(v), l_{\prec}(x \succ y)\beta_{2}(v) = l_{\succ}(\alpha_{1}(x))l_{\prec}(y)v,\cr
r_{\prec}(\alpha_{2}(x))l_{\succ}(y)v &=& l_{\succ}(\alpha_{1}(y))r_{\prec}(x)v, r_{\prec}(\alpha_{2}(x))r_{\succ}(y)v = r_{\succ}(y\prec x)\beta_{1}(v),\cr
l_{\succ}(x\ast y)\beta_{2}(v) &=& l_{\succ}(\alpha_{1}(x))l_{\succ}(y)v, r_{\succ}(\alpha_{2}(x))l_{\ast}(y)v= l_{\succ}(\alpha_{1}(y))r_{\succ}(x)v,\cr
r_{\succ}(\alpha_{2}(x))r_{\ast}(y)v &=& r_{\succ}(y \succ x)\beta_{1}(v),
\end{eqnarray*}
where $ x \ast y = x \succ y + x \prec y, l_{\ast} = l_{\succ} + l_{\prec}, r_{\ast} = r_{\succ} + r_{\prec} $.
\end{definition}
\begin{proposition}
Let $(l_{\succ}, r_{\succ}, l_{\prec}, r_{\prec}, \beta_{1}, \beta_{2}, V)$ be a bimodule of a biHom-dendriform algebra $(\mathcal{A},\succ, \prec, \alpha_{1}, \alpha_{2}).$ Then, there exists a biHom-dendriform algebra structure on the direct sum $\mathcal{A}\oplus V $ of the underlying vector spaces of $ \mathcal{A} $ and $V$ given by
\begin{eqnarray*}
(x + u) \succ (y + v) &=& x \succ y + l_{\succ}(x)v + r_{\succ}(y)u, \cr
(x + u) \prec (y + v) &=& x \prec y + l_{\prec}(x)v + r_{\prec}(y)u
\end{eqnarray*}
for all $ x, y \in \mathcal{A}, u, v \in V $. We denote it by $ \mathcal{A} \times_{l_{\succ},r_{\succ}, l_{\prec}, r_{\prec}, \alpha_{1}, \alpha_{2}, \beta_{1}, \beta_{2}} V$.
\end{proposition}
\begin{proof}
Let $v_{1}, v_{2}, v_{3}\in V$ and $x_{1}, x_{2},x_{3}\in \mathcal{A}.$ Setting and computing
\begin{equation}\label{condit. of Bimod. of biHom-dendriform}
\begin{aligned}[l]
[(x_{1} + v_{1})\prec (x_{2} + v_{2})]\prec(\alpha_{2}(x_{3})+ \beta_{2}(v_{3}))=\cr
(\alpha_{1}(x_{1})+ \beta_{1}(v_{1}))\prec[(x_{2} + v_{2}) \ast (x_{3} + v_{3})],\cr
[(x_{1} + v_{1})\succ(x_{2} + v_{2})]\prec(\alpha_{2}(x_{3})+ \beta_{2}(v_{3}))=\cr
(\alpha_{1}(x_{1})+ \beta_{1}(v_{1}))\succ[(x_{2} + v_{2})\prec (x_{3} + v_{3})],\cr
[\alpha_{1}(x_{1})+ \beta_{1}(v_{1})]\succ[(x_{2})+ v_{2})\succ(x_{3} + v_{3})]=\cr
[(x_{1} + v_{1})\ast(x_{2} + v_{2})]\succ(\alpha_{2}(x_{3})+ \beta_{2}(v_{3})),
\end{aligned}
\end{equation}
 one obtains the conditions of the bimodule of a biHom-dendriform algebra, which  completes the proof.
\qed
\end{proof}
\begin{proposition}
Let ($l_{\succ}, r_{\succ}, l_{\prec}, r_{\prec}, \beta_{1}, \beta_{2}, V$) be a bimodule of a biHom-dendriform algebra $(\mathcal{A}, \succ, \prec, \alpha_{1}, \alpha_{2})$.
Then
\begin{enumerate}[label=\upshape{\arabic*)}]
\item $(l_{\succ}, r_{\prec}, \beta_{2}, \beta_{1}, V) $ and $ (l_{\succ} + l_{\prec}, r_{\succ} + r_{\prec}, \beta_{1}, \beta_{2}, V)$ are bimodules \\ of $ (\mathcal{A}, \ast, \alpha_{2}, \alpha_{1}); $
\item for any bimodule $(l, r, \beta_{1}, \beta_{2}, V)$ of $(\mathcal{A}, \ast, \alpha_{1}, \alpha_{2})$, $(l, 0, 0, r, \beta_{2}, \beta_{1}, V)$ is a bimodule of $(\mathcal{A}, \succ, \prec, \alpha_{2}, \alpha_{1}).$
\item $(l_{\succ} + l_{\prec}, 0, 0,  r_{\succ} + r_{\prec}, \beta_{1}, \beta_{2}, V)$ and $(l_{\succ}, 0, 0, r_{\prec}, \beta_{1}, \beta_{2}, V)$ are bimodules
 of\\ $(\mathcal{A}, \succ, \prec, \alpha_{1}, \alpha_{2});$
\item  the biHom-dendriform algebras
\begin{eqnarray*}
 \mathcal{A} \times_{l_{\succ}, r_{\succ}, l_{\prec}, r_{\prec}, \alpha_{1}, \alpha_{2}, \beta_{1}, \beta_{2}}V \mbox{ and }  \mathcal{A} \times_{l_{\succ} +  l_{\prec} , 0, 0, r_{\succ} + r_{\prec}, \alpha_{1}, \alpha_{2}, \beta_{1}, \beta_{2}} V
 \end{eqnarray*} have the same associated
biHom-associative algebra $$\mathcal{A} \times_{l_{\succ} +  l_{\prec}, r_{\succ} + r_{\prec}, \alpha_{1}, \alpha_{2}, \beta_{1}, \beta_{2}} V.$$
 \end{enumerate}
\end{proposition}
\begin{theorem}
Let $(\mathcal{A}, \succ_{\mathcal{A}}, \prec_{\mathcal{A}}, \alpha_{1}, \alpha_{2})$ and $(\mathcal{B}, \succ_{\mathcal{B}}, \prec_{\mathcal{B}}, \beta_{1}, \beta_{2})$
 be two biHom-dendriform algebras. Suppose that there are linear maps
$ l_{\succ_{\mathcal{A}}},   r_{\succ_{\mathcal{A}}},  l_{\prec_{\mathcal{A}}},  r_{\prec_{\mathcal{A}}} : \mathcal{A} \rightarrow gl(\mathcal{B}),$
and $ l_{\succ_{\mathcal{B}}},   r_{\succ_{\mathcal{B}}},  l_{\prec_{\mathcal{B}}},  r_{\prec_{\mathcal{B}}} : \mathcal{B} \rightarrow gl(\mathcal{A})$
such that ($ l_{\succ_{\mathcal{A}}},   r_{\succ_{\mathcal{A}}},  l_{\prec_{\mathcal{A}}},
r_{\prec_{\mathcal{A}}}, \beta_{1}, \beta_{2},\\ \mathcal{B}$) is a bimodule of $\mathcal{A},$ and
($l_{\succ_{\mathcal{B}}},   r_{\succ_{\mathcal{B}}},  l_{\prec_{\mathcal{B}}},  r_{\prec_{\mathcal{B}}},
\alpha_{1}, \alpha_{2}, \mathcal{A})$ ) is a bimodule  of $\mathcal{B},$
satisfy
\begin{eqnarray}
\label{bieq35}
r_{\prec_{\mathcal{A}}}(\alpha_{2}(x))(a \prec_{\mathcal{B}} b) = \beta_{1}(a)\prec_{\mathcal{B}}( r_{\mathcal{A}}(x)b) + r_{\prec_{\mathcal{A}}}(l_{\mathcal{B}}(x)\beta_{1}(a)),
\\
\label{bieq36}
\begin{array}{l}
l_{\prec_{\mathcal{A}}}(l_{\prec_{\mathcal{B}}}(x))\beta_{2}(b) + (r_{\prec_{\mathcal{A}}}(x)a) \prec_{\mathcal{B}}\beta_{2}(b)= \cr
\beta_{1}(a) \prec_{\mathcal{B}} (l_{\prec_{\mathcal{A}}}(x)b) + r_{\prec_{\mathcal{A}}}(r_{\prec_{\mathcal{B}}}(b)x)\beta_{1}(a),
\end{array} \\
\label{bieq37}
l_{\prec_{\mathcal{A}}}(\alpha_{1}(x))(a \ast_{\mathcal{B}} b) = (l_{\prec_{\mathcal{A}}}(x)a) \ast_{\mathcal{B}} \beta_{2}(b) +
 l_{\prec_{\mathcal{A}}}(r_{\prec_{\mathcal{A}}}(a)x)\beta_{2}(b),
\end{eqnarray}
\begin{eqnarray} \label{bieq38}
r_{\prec_{\mathcal{A}}}(\alpha_{2}(x))(a \succ_{\mathcal{B}} b) = r_{\succ_{\mathcal{A}}}(l_{\prec_{\mathcal{B}}}(b)x)\beta_{1}(a) +
\beta_{1}(a)\succ_{\mathcal{B}} (r_{\prec_{\mathcal{A}}}(x)b), \\
\label{bieq39}
\begin{array}{ll}
l_{\prec_{\mathcal{A}}}(l_{\succ_{\mathcal{B}}}(a)x)\beta_{2}(b) & + (r_{\succ_{\mathcal{A}}}(x)a) \prec_{\mathcal{B}}\beta_{2}(b)=\cr
& \beta_{1}(a)\succ_{B} (l_{\prec_{\mathcal{A}}}(x)b) + r_{\succ_{\mathcal{A}}}(r_{\prec_{\mathcal{B}}}(b)x)\beta_{1}(a)
\end{array} \\
\label{bieq40}
l_{\succ_{\mathcal{A}}}(\alpha_{1}(x))(a \prec_{\mathcal{B}} b) = ( l_{\succ_{\mathcal{A}}}(x)a) \prec_{\mathcal{B}}\beta_{2}(b) +
 l_{\prec_{\mathcal{A}}}(r_{\succ_{\mathcal{B}}}(a)x)\beta_{2}(b),
\\
\label{bieq41}
r_{\succ_{\mathcal{A}}}(\alpha_{2}(x))(a \ast_{\mathcal{B}} b)= \beta_{1}(a)\succ_{\mathcal{B}} (r_{\succ_{\mathcal{A}}}(x)b) +
 r_{\succ_{\mathcal{A}}}(l_{\succ_{\mathcal{B}}}(b)x)\beta_{1}(a),
\\
\label{bieq42}
\begin{array}{ll}
\beta_{1}(a)\succ_{\mathcal{B}} (l_{\succ_{\mathcal{A}}}(x)b) &+ r_{\succ_{\mathcal{A}}}(r_{\succ_{\mathcal{B}}}(b)x)\beta_{1}(a)=\cr
&l_{\succ_{\mathcal{A}}}(l_{\mathcal{B}}(a)x)\beta_{2}(b) + (r_{\mathcal{A}}(x)a) \succ_{\mathcal{B}}\beta_{2}(b),
\end{array} \\
\label{bieq43}
l_{\succ_{\mathcal{A}}}(\alpha_{1}(x))(a \succ_{\mathcal{B}} b) = (l_{\mathcal{A}}(x)a) \succ_{\mathcal{B}}\beta_{2}(b) + l_{\succ_{\mathcal{A}}}(r_{\mathcal{B}}(a)x)\beta_{2}(b),
\\
\label{bieq44}
r_{\prec_{\mathcal{B}}}(\beta_{2}(a))(x \prec_{\mathcal{A}} y) = \alpha_{1}(x)\prec_{\mathcal{A}} (r_{\mathcal{B}}(a)y) + r_{\prec_{\mathcal{B}}}(l_{\mathcal{A}}(y)a)\alpha_{1}(x),
\\
\label{bieq45}
\begin{array}{ll}
l_{\prec_{B}}(l_{\prec_{\mathcal{A}}}(x)a)\alpha_{2}(y) &+ (r_{\prec_{\mathcal{B}}}(a)x) \prec_{\mathcal{A}}\alpha_{2}(y)=\cr
&\alpha_{1}(x)\prec_{\mathcal{A}} (l_{\mathcal{B}}(a)y) + r_{\prec_{\mathcal{B}}}(r_{\mathcal{A}}(y)a)\alpha_{1}(x),
\end{array} \\
\label{bieq46}
l_{\prec_{\mathcal{B}}}(\beta_{1}(a))(x \ast_{\mathcal{A}} y) = (l_{\prec_{\mathcal{B}}}(a)x) \prec_{\mathcal{A}}\alpha_{2}(y) +
l_{\prec_{\mathcal{B}}}(r_{\prec_{\mathcal{A}}}(x)a)\alpha_{2}(y),
\\
\label{bieq47}
r_{\prec_{\mathcal{B}}}(\beta_{2}(a))(x \succ_{\mathcal{A}} y) = r_{\succ_{\mathcal{B}}}(l_{\prec_{\mathcal{B}}}(y)a)\alpha_{1}(x) +
 \alpha_{1}(x)\succ_{\mathcal{A}} (r_{\prec_{\mathcal{B}}}(a)y),
\\
\label{bieq48}
\begin{array}{ll}
l_{\prec_{\mathcal{B}}}(l_{\succ_{\mathcal{A}}}(x)a)\alpha_{2}(y) &+ (r_{\succ_{\mathcal{B}}}(a)x) \prec_{\mathcal{A}}\alpha_{2}(y)=\cr
&\alpha_{1}(x)\succ_{\mathcal{A}} (l_{\prec_{\mathcal{B}}}(a)y) + r_{\succ_{\mathcal{B}}}(r_{\prec_{\mathcal{A}}}(y)a)\alpha_{1}(x),
\end{array} \\
\label{bieq49}
l_{\succ_{\mathcal{B}}}(\beta_{1}(a))(x \prec_{\mathcal{A}} y) = (l_{\succ_{\mathcal{B}}}(a)x) \prec_{\mathcal{A}}\alpha_{2}(y) +
l_{\prec_{\mathcal{B}}}(r_{\succ_{\mathcal{A}}}(x)a)\alpha_{2}(y),
\\
\label{bieq50}
 r_{\succ_{\mathcal{B}}}(\beta_{2}(a))(x \ast_{\mathcal{A}} y)= \alpha_{1}(x)\succ_{\mathcal{A}} (r_{\succ_{\mathcal{B}}}(a)y) +
r_{\succ_{\mathcal{B}}}(l_{\succ_{\mathcal{A}}}(y)a)\alpha_{1}(x),
\\
\label{bieq51}
\begin{array}{ll}
\alpha_{1}(x)\succ_{\mathcal{A}} (l_{\succ_{\mathcal{B}}}(a)y) & + r_{\succ_{\mathcal{B}}}(r_{\succ_{\mathcal{A}}}(y)a)\alpha_{1}(x)=\cr
& l_{\succ_{\mathcal{B}}}(l_{\mathcal{A}}(x)a)\alpha_{2}(y) + (r_{B}(a)x) \succ_{\mathcal{A}}\alpha_{2}(y),
\end{array} \\
\label{bieq52}
l_{\succ_{\mathcal{B}}}(\beta_{1}(a))(x \succ_{\mathcal{A}} y) = (l_{\mathcal{B}}(a)x) \succ_{\mathcal{A}}\alpha_{2}(y) +
l_{\succ_{\mathcal{B}}}(r_{\mathcal{A}}(x)a)\alpha_{2}(y)
\end{eqnarray}
for any $ x, y \in \mathcal{A}, a, b \in \mathcal{B} $ and $ l_{\mathcal{A}} = l_{\succ_{\mathcal{A}}} +
 l_{\prec_{\mathcal{A}}}, r_{\mathcal{A}} =  r_{\succ_{\mathcal{A}}} +  r_{\prec_{\mathcal{A}}}, l_{\mathcal{B}} =
 l_{\succ_{\mathcal{B}}} +  l_{\prec_{\mathcal{B}}} , r_{\mathcal{B}} =  r_{\succ_{\mathcal{B}}} +  r_{\prec_{\mathcal{B}}} $.
 Then, there is a biHom-dendriform algebra structure on the direct sum $ \mathcal{A} \oplus \mathcal{B} $ of the underlying vector spaces of
 $ \mathcal{A} $ and $ \mathcal{B} $ given by
\begin{eqnarray*}
(x + a) \succ ( y + b ) &=& (x \succ_{\mathcal{A}} y + r_{\succ_{\mathcal{B}}}(b)x + l_{\succ_{\mathcal{B}}}(a)y)\cr
&+&(l_{\succ_{\mathcal{A}}}(x)b + r_{\succ_{\mathcal{A}}}(y)a + a \succ_{\mathcal{B}} b ), \cr
(x + a) \prec ( y + b ) &=& (x \prec_{\mathcal{A}} y + r_{\prec_{\mathcal{B}}}(b)x + l_{\prec_{\mathcal{B}}}(a)y)\cr
&+& (l_{\prec_{\mathcal{A}}}(x)b + r_{\prec_{\mathcal{A}}}(y)a + a \prec_{\mathcal{B}} b )
\end{eqnarray*}
for any $ x, y \in \mathcal{A}, a, b \in \mathcal{B} $.
\end{theorem}
\begin{proof}
The proof  is obtained in a similar way as for Theorem \ref{theo. of matched pairs}.
\qed
\end{proof}
 Let $ \mathcal{A} \bowtie^{l_{\succ_{\mathcal{A}}}, r_{\succ_{\mathcal{A}}},
l_{\prec_{\mathcal{A}}}, r_{\prec_{\mathcal{A}}}, \beta_{1}, \beta_{1}}_{l_{\succ_{\mathcal{B}}}, r_{\succ_{\mathcal{B}}}, l_{\prec_{\mathcal{B}}},
r_{\prec_{\mathcal{B}}}, \alpha_{1}, \alpha_{2}} \mathcal{B} $ denote this biHom-dendriform algebra.
\begin{definition}
Let $ (\mathcal{A}, \succ_{\mathcal{A}}, \prec_{\mathcal{A}}, \alpha_{1}, \alpha_{2}) $ and $  (\mathcal{B}, \succ_{\mathcal{B}}, \prec_{\mathcal{B}}, \beta_{1}, \beta_{2}) $
be two biHom-dendriform algebras. Suppose there exist linear maps
$ l_{\succ_{\mathcal{A}}}, r_{\succ_{\mathcal{A}}}, l_{\prec_{\mathcal{A}}}, r_{\prec_{\mathcal{A}}} : \mathcal{A} \rightarrow gl(\mathcal{B}),$
and $ l_{\succ_{\mathcal{B}}}, r_{\succ_{\mathcal{B}}}, l_{\prec_{\mathcal{B}}}, r_{\prec_{\mathcal{B}}} : \mathcal{B} \rightarrow gl(\mathcal{A}) $
 such that $(l_{\succ_{\mathcal{A}}}, r_{\succ_{\mathcal{A}}}, l_{\prec_{\mathcal{A}}}, r_{\prec_{\mathcal{A}}}, \beta_{1}, \beta_{2})$ is a bimodule of $ \mathcal{A},$
and $(l_{\succ_{\mathcal{B}}}, r_{\succ_{\mathcal{B}}}, l_{\prec_{\mathcal{B}}}, r_{\prec_{\mathcal{B}}}, \alpha_{1}, \alpha_{2})$ is a bimodule of $ \mathcal{B} $.
If  \eqref{bieq35} - \eqref{bieq52} are satisfied, $(\mathcal{A}, \mathcal{B}, l_{\succ_{\mathcal{A}}},
r_{\succ_{\mathcal{A}}}, l_{\prec_{\mathcal{A}}}, r_{\prec_{\mathcal{A}}}, \beta_{1}, \beta_{2}, l_{\succ_{\mathcal{B}}}, r_{\succ_{\mathcal{B}}}, l_{\prec_{\mathcal{B}}},
 r_{\prec_{\mathcal{B}}}, \alpha_{1}, \alpha_{2})$ is called a \textbf{matched pair of biHom-dendriform algebras}.
\end{definition}

\begin{corollary}
Let $(\mathcal{A}, \mathcal{B}, l_{\succ_{\mathcal{A}}}, r_{\succ_{\mathcal{A}}}, l_{\prec_{\mathcal{A}}}, r_{\prec_{\mathcal{A}}}, \beta_{1}, \beta_{2},
 l_{\succ_{\mathcal{B}}}, r_{\succ_{\mathcal{B}}}, l_{\prec_{\mathcal{B}}}, r_{\prec_{\mathcal{B}}}, \alpha_{1}, \alpha_{2}) $ be a matched pair of biHom-dendriform algebras.
Then, $(\mathcal{A}, \mathcal{B}, l_{\succ_{\mathcal{A}}} + l_{\prec_{\mathcal{A}}}, r_{\succ_{\mathcal{A}}} + r_{\prec_{\mathcal{A}}},
l_{\succ_{\mathcal{B}}} + l_{\prec_{\mathcal{B}}},  r_{\succ_{\mathcal{B}}} + r_{\prec_{\mathcal{B}}}, \alpha_{1} + \beta_{1}, \alpha_{2} + \beta_{2})$ is a matched pair of the associated
biHom-associative algebras $(\mathcal{A}, \ast_{\mathcal{A}}, \alpha_{1}, \alpha_{2})$ and  $(\mathcal{B}, \ast_{\mathcal{B}}, \beta_{1}, \beta_{2})$.
\end{corollary}
\begin{proof}
The associated biHom-associative algebra $(\mathcal{A} \bowtie \mathcal{B}, \ast, \alpha_{1} + \beta_{1}, \alpha_{2} + \beta_{2})$ is exactly the biHom-associative algebra obtained
from the matched pair of biHom-associative algebras, $(\mathcal{A}, \mathcal{B}, l_{\mathcal{A}}, r_{\mathcal{A}}, \beta_{1}, \beta_{2}, l_{\mathcal{B}}, r_{\mathcal{B}}, \alpha_{1}, \alpha_{2}),$  with
\begin{eqnarray*}
(x + a)\ast (y + b) = x\ast_{\mathcal{A}} y + l_{\mathcal{B}}(a)y + r_{\mathcal{B}}(b)x + a \ast_{\mathcal{B}} b +  l_{\mathcal{A}}(x)b +
r_{\mathcal{A}}(y)a
\end{eqnarray*}
for all $ x, y \in \mathcal{A}, a, b \in \mathcal{B} $, where 
$ l_{\mathcal{A}} = l_{\succ_{\mathcal{A}}} + l_{\prec_{\mathcal{A}}}$, $r_{\mathcal{A}}
= r_{\succ_{\mathcal{A}}} + r_{\prec_{\mathcal{A}}}$, $l_{\mathcal{B}} = l_{\succ_{\mathcal{B}}} + l_{\prec_{\mathcal{B}}}$, $r_{\mathcal{B}}=
r_{\succ_{\mathcal{B}}} + r_{\prec_{\mathcal{B}}}  $.
\qed
\end{proof}
\section{Concluding Remarks}
\label{sec:conclremarks}
In this work, we have
constructed a biHom-associative algebra with a decomposition into  direct sum of the underlying vector spaces of a biHom-associative algebra and its dual  such that both of them are biHom-subalgebras, with either the natural symmetric bilinear form being invariant, or the natural antisymmetric bilinear form being  a Connes cocycle. Then, we have performed the double constructions of biHom-Frobenius algebras and Connes cocycle, and provided the bialgebra structures.

 \section*{Aknowledgement}
This work is supported by TWAS Research Grant RGA No. 17 - 542 RG / MATHS / AF / AC \_G  -FR3240300147. The ICMPA-UNESCO Chair is in partnership with Daniel Iagolnitzer Foundation (DIF), France, and the Association pour la Promotion Scientifique de l'Afrique (APSA), supporting the development of mathematical physics in Africa. Mahouton Norbert Hounkonnou thanks  Professor Sergei D. Silvestrov  for the invitation to attend the $2^{\mbox{nd}}$ International Conference on Stochastic Processes and Algebraic Structures (SPAS2019) and the hospitality during his stay  as visiting professor at the Mathematics and Applied Mathematics research environment MAM, Division of Applied Mathematics, M\"alardalen University, V\"aster\'as, Sweden, where this work has been finalized. Partial support from the Swedish International Development Agency and International Science Program, (ISP) for capasity buildning in Mathematics in Africa is also gratefully acknowledged.

\end{document}